\documentclass[oneside]{article}
\usepackage[intlimits,reqno]{amsmath}
\usepackage{amsthm, graphics}
\usepackage{indentfirst}
\usepackage[a4paper,outer=3cm,inner=3cm]{geometry}
\usepackage{amssymb}
\usepackage{verbatim}
\include{diagxy}
\input{xy}
\xyoption{all}

\usepackage{enumerate}
\usepackage{amsmath}
\usepackage{verbatim}
\usepackage{anyfontsize}
\usepackage{bbm}

\newcommand{\A}{\mathcal{A}}
\newcommand{\G}{\mathcal{G}}
\newcommand{\J}{\mathcal{J}}
\newcommand{\AG}{\mathcal{AG}}

\newcommand{\U}{\mathcal{U}}
\newcommand{\V}{\mathcal{V}}
\newcommand{\T}{\mathcal{T}}

\renewcommand{\H}{\mathcal{H}}
\renewcommand{\ll}{l}
\newcommand{\rr}{r}
\newcommand{\Ll}{\mathcal{L}}
\newcommand{\Ss}{\textbf{s}}
\newcommand{\Tt}{\textbf{t}}
\newcommand{\tto}{\rightrightarrows}
\newcommand{\M}{\mathfrak{M}}
\newcommand{\X}{\mathfrak{X}}

\renewcommand{\to}{\rightarrow}
\newcommand{\prf}[1]{\begin{proof}#1\end{proof}}

\renewcommand{\geq}{\geqslant}

\newcommand{\ol}{\overline}

\newcommand{\be}{\begin{equation}}
\newcommand{\ee}{\end{equation}}
\newcommand{\ben}{\begin{enumerate}}
\newcommand{\een}{\end{enumerate}}
\newcommand{\bex}{\begin{example}}
\newcommand{\eex}{\begin{flushright}$\diamondsuit$\end{flushright}\end{example}}

\newtheorem{thm}{Theorem}[section]
\newtheorem{cor}[thm]{Corrolary}
\newtheorem{lem}[thm]{Lemma}

\newtheorem{prop}[thm]{Proposition}

\theoremstyle{definition}

\newtheorem{example}{Example}
\numberwithin{example}{section}
\newtheorem{rem}[thm]{Remark}

\numberwithin{equation}{section}

\renewcommand{\emph}[1]{{\bf #1}}

\begin{document}
\title{Fibre-wise linear Poisson structures related to $W^*$-algebras}
\author{Anatol Odzijewicz,\ Grzegorz Jakimowicz,\ Aneta Sli\.{z}ewska \\Institute of Mathematics\\
University in Bia{\l}ystok
\\ Cio{\l}kowskiego 1M, 15-245 Bia{\l}ystok, Poland}

\maketitle
\tableofcontents
\vspace{1cm}
\begin{abstract}
In this paper we investigate fiber-wise linear complex Banach sub-Poisson structures defined canonically  by the structure of a $W^*$-algebra $\M$. In particular we show that these structures are arranged  in the short exact sequence of complex Banach sub-Poisson $\mathcal{VB}$-groupoids with the groupoid  $\G(\M)\tto \Ll(\M)$ of partially invertible elements of $\M$ as the side groupoid.

\end{abstract}
\section*{Introduction}
Poisson geometry is of fundamental importance for description of properties of finite as well as infinite dimensional classical Hamiltonian systems, \cite{chern,zung,MR}.  Discovery of the symplectic groupoid  over a Poisson manifold, \cite{kar,wei3,zak1,zak2}, which generalize the cotangent groupoid $T^*G\tto \mathfrak{g}^*$ over the Lie-Poisson space $\mathfrak{g}^*$, where $\mathfrak{g}^*$ is the dual of the Lie algebra $\mathfrak{g}$ of a Lie group $G$, implemented the Lie groupoid theory to Poisson geometry. On the other hand the natural correspondence between fibre-wise Poisson structure and Lie algebroids allows us to consider the Lie algebroid theory as a part of Poisson geometry. From the above and from the fact that Lie algebroids are the infinitesimal  version of Lie groupoids, \cite{zung, mac}, one can incorporate the theory of them to the geometric methods of classical mechanics.

No less crucial then Poisson geometry for description of the classical physical systems is the theory of operator algebras, especially the theory of $W^*$-algebras, for the description of quantum physical systems. By definition  a $W^*$-algebra (von Neumann algebra) is a $C^*$-algebra $\M$ which has a Banach predual space $\M_*$, i.e. $\M=(\M_*)^*$, see \cite{sakai} for the details. This property guarantees  that the complete lattice $\Ll(\M)$ of orthogonal projections  from $\M$ has  plenty of elements and, thus  allows one to interpret the $W^*$-algebras theory as non-commutative probability theory and leads to the von Neumann theory of quantum measurement, \cite{jauch}. In consequence, the propositional calculus of quantum mechanics, called quantum logic, is based on the lattice structure of $\Ll(\M)$. The quantum observables are defined as $\sigma$-homomorphisms from the lattice $\mathcal{B}(\mathbb{R})$ of the Borel sets on the real line  into $\Ll(\M)$, i.e. they are $\Ll(\M)$-valued spectral measures. The states of the quantum system which corresponds to $\M$ are normalized positive elements of $\M_*$. The case when  the $W^*$-algebra $\M$ is the algebra of bounded linear operators $L^\infty(\H)$ on the complex separable Hilbert space $\H$ implements a standard model of quantum mechanics.

One of the most intriguing  problems of mathematical physics is to describe the passage from the classical Hamiltonian systems to the quantum ones, known as a quantization procedure. Though this question will not be touched upon in the present paper we show, however, that the $W^*$-algebra structure defines in a natural way a family of  complex (real) fibre-wise linear Banach sub-Poisson structures on the complex Banach vector bundles over $\Ll(\M)$ and over other Banach manifolds related to $\M$. As an example of such type of structure one can take the Banach Lie-Poisson structure on $\M_*$, which in the case $\M=L^\infty(\H)$ allows to consider Liouville-von Neumann equation as a Hamiltonian equation on $L^1(\H)\cong L^\infty(\H)_*$, see \cite{OR}.

Independently from this physical application the theory of $W^*$-algebras is one of the crucial topics of contemporary mathematics which interconnects  analysis with algebra, \cite{connes,Takes}.

It is rather unexpected that to the category of $W^*$-algebras corresponds in a functorial  way a category of Banach-Lie groupoids and since of that the Banach-Lie algebroids. Namely in \cite{OS}, the complex Banach-Lie groupoid $\G(\M)\tto\Ll(\M)$ of partially invertible elements of $\M$ was introduced and investigated. The Banach-Lie algebroids corresponding to these groupoids are described in detail in \cite{OJS}.
Here we will continue the investigation of the mentioned structures as well as we will study fibre-wise linear sub-Poisson structures which are related to them in a canonical way. Through out the paper we use the noncommutative coordinates for the description of the structures under investigation. The calculus in these coordinates is based on the combining  the algebra structure of $\M$ with the groupoid structure of $\G(\M)\tto \Ll(\M)$.

The structure of this paper is as follows. In Section \ref{sec:grou}, following \cite{OS}, 
 we briefly  discuss   the structure of the complex Banach-Lie groupoid  $\G(\M)\tto \Ll(\M)$ and show  that it splits into the transitive Banach-Lie subgroupoids $\G_{p_0}(\M)\tto \Ll_{p_0}(\M)$, $p_0\in \Ll(\M)$. These subgroupoids are closed-open complex Banach submanifolds of $\G(\M)\tto \Ll(\M)$ isomorphic to the corresponding gauge groupoids, see Proposition \ref{prop:11} and Proposition \ref{prop:14}.  Additionally to \cite{OS}, we prove that $\G(\M)\tto \Ll(\M)$ is a Hausdorf  topological groupoid with respect to the topology underlying its complex Banach manifold structure, see Proposition \ref{prop:Hausd}. In Proposition \ref{prop:central} we characterize the set of central projections of $\M$ in terms of this underlying topology.

The tangent prolongation groupoid $T\G_{p_0}(\M)\tto T\Ll_{p_0}(\M)$ of the groupoid $\G_{p_0}(\M)\tto \Ll_{p_0}(\M)$ is studied in Section \ref{sec:tangent}. The main results of this section are presented in diagrams (\ref{Atiyah11}), (\ref{Atiyah211}) and (\ref{duzy}).

The Banach-Lie algebroid $\A\G(\M)\to \Ll(\M)$ and the Atiyah sequence (\ref{Atiyahalgebr}) of the groupoid $\G(\M)\tto \Ll(\M)$ are described in Section \ref{sec:Atiyah}, see Proposition \ref{rem:31} and Proposition \ref{prop:32}. In particular we present the explicit "coordinate" formula for the algebroid Lie bracket and the algebroid anchor map, see Proposition \ref{prop:33} and Proposition \ref{prop:34}.

Generalizing  the results of \cite{MOS} to Banach sub Poisson case in the last two sections of this paper we investigete the fibre-wise linear Poisson structures related to a $W^*$-algebra. So, in Section \ref{sec:preAtiyah} we show that the short exact sequence (\ref{Atiyahalgebrdual}) predual to the Atiyah sequence (\ref{Atiyahalgebr}) is a short exact sequence of the fibre-wise linear complex Banach sub Poisson manifolds. The description of their structure, including the structure of foliation on the symplective leaves, is presented in Proposition \ref{prop:Palg}, Theorem \ref{prop:43} and Theorem \ref{prop:44}. The exact sequence (\ref{Atiyapredual}) of the complex Banach sub Poisson $\mathcal{VB}$-groupoids with $\G_{p_0}(\M)\tto \Ll_{p_0}(\M)$, $p_0\in \Ll(\M)$, as the side groupoid is investigated in Section \ref{sec:Poisson}. The main result is given in Theorem \ref{thm:54}.

In Section 6 we shortly mention some questions which naturally arise in the investigated theory and which will be the subject of subsequel papers.

In the Appendix we collect some basic notions concerning $\mathcal{VB}$-groupoids theory.


\section{Groupoid $\G(\M)\tto \Ll(\M)$ of partially invertible elements of $W^*$-algebra}\label{sec:grou}
According to \cite{OS} we consider the subset $\G(\M)$ of a $W^*$-algebra $\M$  which consists of such elements $x\in\M$ 
 for which
 $|x|=(x^*x)^{\frac{1}{2}}$ is an invertible element of the $W^*$-subalgebra
$p\M p\subset \M$, where the orthogonal projection $p$ is the support of $|x|$. We have natural
maps $\ll:\G(\M)\to \Ll(\M)$ and $\rr:\G(\M)\to \Ll(\M)$  of $\G(\M)$ on the lattice
$\Ll(\M)$ of orthogonal projections of the $W^*$-algebra $\M$, which are defined as the left and right support of $x\in \G(\M)$, respectively. The set $\G(\M)$ posseses a structure of the groupoid where the target $\Tt:\G(\M)\to \Ll(\M)$ and source $\Ss:\G(\M)\to \Ll(\M)$ maps are  given by  $\ll$ and $\rr$, respectively. The partial multiplication of $x,y\in \G(\M)$ is the
algebra product in $\M$ restricted to such pairs $(x,y)\in \G(\M)\times \G(\M) $ for which $\Tt(y)=\Ss(x)$. The identity section $\mathbf{1}:\Ll(\M)\to \G(\M)$ is defined as  the
 inclusion map and the inversion $\iota:\G(\M)\to \G(\M)$ is defined by 
 \be \iota(x)=x^{-1}:=|x|^{-1}u^*,\ee where the partial isometry  $u$  and $|x|$ are uniquely defined by the polar decomposition $x=u|x|$ of  $x\in \M$, if one assumes that $u^*u$ is equal to the support of $|x|$, \cite{sakai}.

In general the groupoid $\G(\M)\tto \Ll(\M)$ is not a topological groupoid with respect to any natural topology on $\M$, see \cite{OS}. However, the  complex Banach manifold structures, consistent with the groupoid structure of $\G(\M)\tto \Ll(\M)$,  could be defined on $\G(\M)$ and   on $\Ll(\M)$.  Therefore, following \cite{OS} for any $p\in \Ll(\M)$  we define  the subset $\Pi_p\subset\Ll(\M)$ of orthogonal projections $q\in\Ll(\M)$ for  which  the
Banach splitting \be\label{split}\M=q\M\oplus (1-p)\M\ee of  $\M$ is valid. Using  (\ref{split}) we decompose  
 \be\label{rozklad} p=x_p-y_p\ee

the projection  $p\in\Ll(\M)$, where $ x_p\in\  q\M p$ and $y_p\in (1-p)\M p$, and in this way  define a bijection 
\be\label{varphi}\Pi_p\ \ni\ q\mapsto y_p=:\varphi_p(q)=(pq)^{-1}-p \ \in\ (1-p)\M p,\ee
and the local section
\be\label{sigma}\Pi_p\ \ni\ q\mapsto x_p=:\sigma_p(q)= (pq)^{-1}\in\  \Tt^{-1}(\Pi_p)\ee
of the target map $\Tt:\G(\M)\to \Ll(\M)$. Note here that $pq\in \G_q^p(\M):=\Tt^{-1}(p)\cap\Ss^{-1}(q)\subset p\M q$.

In \cite{OS} it is shown that the atlas consisting of charts $(\Pi_p, \ \varphi_p)$, where $p\in  \Ll(\M)$,
defines a complex Banach manifold structure on $\Ll(\M)$. The transition maps $\varphi_{p'}\circ \varphi_p^{-1}:\varphi_{p'}(\Pi_{p'}\cap\Pi_p)\to \varphi_p(\Pi_{p'}\cap\Pi_p)$ for the charts (\ref{varphi}) are the following:
\be\label{trans} y_{p'}=(\varphi_{p'}\circ\varphi^{-1}_{p})(y_p)=(b+dy_p)(a+cy_p)^{-1},\ee
where $a=p'p$, $\ b=(1-p')p$, $\ c=p'(1-p)$ and $\ d=(1-p')(1-p)$.

The complex Banach manifold structure on $\G(\M)$ is defined by the atlas which consists of charts:
$$ \Omega_{p\tilde{p}}:=\Tt^{-1}(\Pi_{p})\cap \Ss^{-1}(\Pi_{\tilde{p}})\not=\emptyset\qquad\qquad$$ 
 \be\label{psippmap}\psi_{p\tilde{p}}:\Omega_{p\tilde{p}}\to (1-p)\M p\oplus p\M \tilde{p}\oplus(1-\tilde{p})\M\tilde{p},\ee 
where $( p,\tilde p)\in \Ll(\M)\times \Ll(\M)$ and
 \be\label{psipp}\psi_{p\tilde{p}}(x):=\left(\varphi_{p}(\Tt(x)),(\sigma_{p}(\Tt(x)))^{-1}x \sigma_{\tilde{p}}(\Ss(x)),
\varphi_{\tilde{p}}(\Ss(x)) \right)=:(y_p,z_{ p \tilde p},\tilde y_{\tilde{p}}).\ee  
We note that $z_{p\tilde{p}}\in \G^p_{\tilde{p}}(\M)$ and $\G^p_{\tilde{p}}(\M)$ is an open subset of $p\M\tilde{p}$.
The transition  maps $\psi_{p'\tilde{p'}}\circ\psi_{p\tilde{p}}^{-1}:\psi_{p\tilde{p}}(\Omega_{p'\tilde{p'}}\cap\Omega_{p\tilde{p}})\to
 \psi_{p'\tilde{p'}}(\Omega_{p'\tilde{p'}}\cap\Omega_{p\tilde{p}})$
 between the above type  charts are given by
\be\label{atlas2}\begin{array}{l} y_{p'}=(b+dy_p)(a+cy_p)^{-1},\\
 z_{p'\tilde p'}=(a+cy_p)z_{p\tilde p}(\tilde{a}+\tilde{c}\tilde{y}_{\tilde p})^{-1}\\
 \tilde{y}_{\tilde p'}=(\tilde{b}+\tilde {d}\tilde{y}_{\tilde p})(\tilde{a}+\tilde{c}\tilde{y}_{\tilde p})^{-1},\end{array}\ee
where
 $$(y_{p'},z_{p'\tilde p'},\tilde{y}_{\tilde p'})=(\psi_{p'\tilde{p'}}\circ\psi_{p\tilde{p}}^{-1})(y_p,z_{p\tilde{p}},\tilde{y}_{\tilde p})$$
and $\tilde a=\tilde p'\tilde p$, $\ \tilde b=(1-\tilde p')\tilde p$, $\ \tilde c=\tilde p'(1-\tilde p)$ and $\ \tilde d=(1-\tilde p')(1-\tilde p)$. For more details we refer to \cite{OS}.

The above complex Banach manifold structures on $\Ll(\M)$ and on $\G(\M)$ define the corresponding underlying topological structures. Recall that the base of the underlying topology of a manifold is given by the domains of charts of the maximal atlas, see \cite{bou}, \S 5.1.
\begin{prop}\label{prop:Hausd}
The complex Banach groupoid $\G(\M)\tto \Ll(\M)$ endowed with the underlying topology is a Hausdorff topological groupoid.
\end{prop}
\begin{proof} Let us assume that in $\G(\M)$ there exist  $x_1\not= x_2$ such that any open neighborhoods $\Omega_1\ni x_1$ and  $\Omega_2\ni x_2$ intersect $\Omega_1\cap\Omega_2\not=\emptyset$.  For $i=1,2$ let us take $\Omega_i=\Omega_{\varepsilon_i\delta_i\tilde{\varepsilon}_i}:=\psi^{-1}_{p_i\tilde{p}_i}(K_{\varepsilon_i}(0)\times K_{\delta_i}(z^0_{p_i\tilde{p}_i})\times K_{\tilde{\varepsilon}_i}(0))$, where $p_i=\Tt(x_i)$, $\tilde{p}_i=\Ss(x_i)$, $K_{\varepsilon_i}(0)\subset (1-p_i)\M p_i$ and 
$K_{\tilde{\varepsilon}_i}(0)\subset (1-\tilde{p}_i)\M \tilde{p}_i$ are open balls of the radiuses $\varepsilon_i$ and $\tilde\varepsilon_i$ centered at zero; also $K_{\delta_i}(z^0_{p_i\tilde{p}_i})\subset \G^{p_1}_{\tilde{p}_1}\subset p_i\M \tilde {p}_i$ is an open ball centered in $z^0_{p_i\tilde{p}_i}$, where $\psi_{p_i\tilde{p}_i}(x_i)=(0,z^0_{p_i\tilde{p}_i},0)$. If $x\in \Omega_1\cap\Omega_2$ then from (\ref{atlas2}) one obtains
\be\label{cos111}\begin{array}{l} y_{p_2}(a+cy_{p_1})=b+dy_{p_1},\\
 z_{p_2\tilde {p_2}}(\tilde{a}+\tilde{c}\tilde{y}_{\tilde {p}_1})=(a+cy_{p_1})z_{p_1\tilde{p}_1}\\
 \tilde{y}_{\tilde {p}_2}(\tilde{a}+\tilde{c}\tilde{y}_{\tilde p_1})=\tilde{b}+\tilde {d}\tilde{y}_{\tilde p_1},\end{array}\ee
where
 $$\psi_{p_2\tilde{p_2}}^{-1}(y_{p_2},z_{p_2\tilde p_2},\tilde{y}_{\tilde p_2})=\psi_{p_1\tilde{p_1}}^{-1}(y_{p_1},z_{p_1\tilde{p_1}},\tilde{y}_{\tilde p_1})=x.$$
The other quantities in (\ref{cos111}) are given by $a=p_2p_1$, $\ b=(1-p_2)p_1$, $\ c=p_2(1-p_1)$,  $\ d=(1-p_2)(1-p_1)$,
 $\tilde a=\tilde p_2\tilde p_1$, $\ \tilde b=(1-\tilde p_2)\tilde p_1$, $\ \tilde c=\tilde p_2(1-\tilde p_1)$ and $\ \tilde d=(1-\tilde p_2)(1-\tilde p_1)$. The elements $x\in \Omega_{\varepsilon_1\delta_1\tilde{\varepsilon}_1}\cap \Omega_{\varepsilon_2\delta_2\tilde{\varepsilon}_2}$ in the limit $\varepsilon_1,\delta_1,\tilde{\varepsilon}_1,\varepsilon_2,\delta_2,\tilde{\varepsilon}_2\to 0$  go to $x\to x_1$ and $x\to x_2$. Thus and from (\ref{cos111}) we have
$$ 0=b$$
\be z^0_{p_2\tilde{p}_2}\tilde a=az^0_{p_1\tilde{p}_1}\ee
$$ 0=\tilde{b}$$
The above gives 
$$ 0=p_1p_2-p_2$$
\be\label{cos11} z^0_{p_2\tilde{p}_2}=p_2 z^0_{p_1\tilde{p}_1}\tilde{p}_2\ee
$$  0=\tilde{p}_1\tilde{p}_2-\tilde{p}_2$$
Taking in (\ref{atlas2}) the transition map $\psi_{p_1\tilde{p_1}}\circ\psi_{p_2\tilde{p_2}}^{-1}$ instead of $\psi_{p_2\tilde{p_2}}\circ\psi_{p_1\tilde{p_1}}^{-1}$ and repeating  the analogous considerations we obtain
$$ 0=p_2p_1-p_1$$
\be\label{cos12} z^0_{p_1\tilde{p}_1}=p_1 z^0_{p_2\tilde{p}_2}\tilde{p}_1\ee
$$  0=\tilde{p}_2\tilde{p}_1-\tilde{p}_1$$
From (\ref{cos11}) and (\ref{cos12}) we find that $p_1=p_2$, $\ \tilde{p}_1=\tilde{p}_2$ and $z^0_{p_1\tilde{p}_1}=z^0_{p_2\tilde{p}_2}$. It means that $x_1=x_2$ which is in  the contradiction to the assumption that $x_1\not=x_2$. Thus we conclude that  the topology underlying the complex Banach manifold strucure of $\G(\M)$ is Hausdorff one. Since  $\Ll(\M)$ is a submanifold of $\G(\M)$  the underlying topology of $\Ll(\M)$ is also a Hausdorff one.
\end{proof}

Keeping in mind Proposition \ref{prop:Hausd} and the definition of Lie groupoid in the finite dimensional case (see e.g. \cite{zung, mac}) we will consider $\G(\M)\tto\Ll(\M)$ as a Banach-Lie groupoid.

\bigskip

	Let us fix $p_0\in \Ll(\M)$ and define   $\Ll_{p_0}(\M):=\{p\in \Ll(\M): \ p\sim p_0\}$  and $\G_{p_0}(\M):=\Tt^{-1}(\Ll_{p_0}(\M))\cap\Ss^{-1}(\Ll_{p_0}(\M))$, where $p\sim p'$ denotes the Murray - von Neumann equivalence of projections. Then $\G_{p_0}(\M)\tto \Ll_{p_0}(\M)$ is a Banach-Lie subgroupoid  of $\G(\M)\tto \Ll(\M)$ defined unambiguously by the choice of $p_0\in \Ll(\M)$. The total space  $\G_{p_0}(\M)$ as well as the base  $\Ll_{p_0}(\M)$ of the groupoid $\G_{p_0}(\M)\tto \Ll_{p_0}(\M)$ are subsets of $\G(\M)$ and  $\Ll(\M)$, respectively, open with respect to the topology defined by their Banach manifold structures. If  $p\sim p_0$  then the groupoid $\G_{p}(\M)\tto \Ll_{p}(\M)$ coincides with $\G_{p_0}(\M)\tto \Ll_{p_0}(\M)$. If $\Pi_p\cap \Ll_{p_0}(\M)\not=\emptyset$ then for $q\in \Pi_p\cap \Ll_{p_0}(\M)$ one has $q\sim  p$ and $q\sim p_0$. Thus one obtains $p\sim p_0$, and so $\Pi_p\subset \Ll_{p_0}(\M)$. Hence,  for $p\not\sim p_0$ one has $\Pi_p\cap \Ll_{p_0}(\M)=\emptyset$. From the above it follows that, for $p\in \Ll(\M)$,  domain of chart $\varphi_p:\Pi_p\to (1-p)\M p$ is contained in $\Ll_{p_0}(\M)$ if and only if  $p\sim p_0$.  As a conclusion we have
	\begin{prop}\label{prop:11}
	The Banach-Lie groupoid $\G(\M)\tto \Ll(\M)$ is a disjoint union of Banach-Lie subgroupoids $\G_{p_0}(\M)\tto \Ll_{p_0}(\M)$, $p_0\in \Ll(\M)$, which are its closed-open Banach subgroupoids.
	\end{prop}
	
	Let us denote by $C(\M)$ the center of $\M$. The next proposition characterizes the set $C(\M)\cap\Ll(\M)$ of central projections of the $W^*$-algebra $\M$ in the terms of the underlying topology of the complex Banach manifold structure of $\Ll(\M)$.
	
	\begin{prop}\label{prop:central} One has the following statements:
	\begin{enumerate}[(i)]
	\item  $p\in C(\M)\cap\Ll(\M)$  if and only if $\Pi_p=\{p\}$;
	\item If $p\notin C(\M)\cap\Ll(\M)$ then $\Pi_p\cap C(\M)=\emptyset$.
	\end{enumerate}
	\end{prop}
	\begin{proof} \begin{enumerate}[(i)]
	\item If $p$ is a central projection  then $(1-p)\M p=0$. So, for $q\in \Pi_p$ one has $\varphi_p(q)=y_p=0$, and thus $\Pi_p=\{p\}$.
	
	Let $p$ be a projection such that $\Pi_p=\{p\}$. This means that \be\label{1ort}(1-p)\M p=0,\ee because $\varphi_p:\Pi_p\to (1-p)\M p$ is a bijection. According to Lemma 1.7 of Chapter V in \cite{Takes} the condition (\ref{1ort}) is equivalent to 
	\be\label{2ort} c(1-p)c(p)=0,\ee
	where $ c(1-p)$ and $c(p)$ are central supports of $1-p$ and $p$, respectively. Since $c(1-p)\geq 1-p$ and $c(p)\geq p$ from (\ref{2ort}) one has
	\be\label{3ort} c(1-p)+c(p)=(c(1-p)+c(p))(1-p+p)=1-p+p=1.\ee
	From (\ref{2ort}) and (\ref{3ort}) follows that $c(1-p)=1-p$ and $c(p)=p$. So $p$ is a central projection.
	\item Let us assume that $\Pi_p\cap C(\M)\not=\emptyset$ and take $q\in \Pi_p\cap C(\M)$. It follows from (i) of this proposition that $\Pi_q=\{q\}$ and, thus $\varphi_q(q)=0$. From (\ref{trans}) one obtains 
	\be \varphi_p(q)=(\varphi_p\circ \varphi_q^{-1})(0)=ba^{-1},\ee
	where $a=pq$ and $b=(1-p)q$. Since $q$ is a central projection the elements $a$ and $b$ are also projections. So, $a^{-1}=a$, and thus $ba=0$. The above shows that $\varphi_p(q)=0$, i.e. $p=q\in C(\M)\cap\Ll(\M)$. This is in the contradiction with $p\notin \Pi_p\cap C(\M)$.
	\end{enumerate}
	\end{proof}
	
	We conclude from Proposition \ref{prop:central} that $ C(\M)\cap\Ll(\M)$ coincides with the set of elements of $\Ll(\M)$ which are open-closed one element subsets of $\Ll(\M)$. In the case when $\M$ is a commutative $W^*$-algebra one has $\Pi_p=\{p\}$ for any $p\in \Ll(\M)$. Therefore, for such $\M$ the Banach manifold structure of $\Ll(\M)$ defined by the atlas (\ref{sigma}) is trivial, i.e. $\Ll(\M)$ is $0$-dimensional manifold.
	
	Let us recall here that the inner subgroupoid  $\J\tto B$ of a groupoid $G\tto B$ is defined as $\J:=\bigcup_{b\in B}\Ss^{-1}(b)\cap\Tt^{-1}(b)$. The inner subgroupoid of $\G(\M)\tto \Ll(\M)$ we will denote by $\J(\M)\tto\Ll(\M)$. For a commutative algebra $\M$  one has $l(x)=xx^{-1}=x^{-1}x=r(x)$. So, the groupoid $\G(\M)\tto \Ll(\M)$ coincides with its inner subgroupoid  
	$\J(\M)\tto\Ll(\M)$.	
	
	Summing up the facts mentioned above we see that in the case of a commutative $W^*$-algebra $\M$ one can consider $\G(\M)\tto \Ll(\M)$ as the disjoint union of Banach-Lie groups $G(p\M p)$ enumerated by $p\in \Ll(\M).$
	
	\bigskip

From  Proposition \ref{prop:11} it follows that in order to study the Banach-Lie groupoid structure of 	$\G(\M)\tto \Ll(\M)$ it is enough  to investigate the structure of Banach-Lie subgroupoids  $\G_{p_0}(\M)\tto \Ll_{p_0}(\M)$, $p_0\in \Ll(\M)$. For this reason let us note that the source map $\Ss:\G_{p_0}(\M)\to \Ll_{p_0}(\M)$ is a surjective submersion which in the coordinates (\ref{psipp}) assumes the form $(y_p,z_{p\tilde{p}}, \tilde{y}_{\tilde{p}})\mapsto \tilde{y}_{\tilde{p}}$. So, the fibre $P_0:=\Ss^{-1}(p_0)$ of the source map $\Ss:\G(\M)\tto \Ll_{p_0}(\M)$ is a Banach submanifold of $\G_{p_0}(\M)$. Restricting (\ref{psippmap}) and (\ref{psipp}) to $P_0\cap\Omega _{pp_0}=\Tt^{-1}(\Pi_p)\cap\Ss^{-1}(p_0)=\pi_0^{-1}(\Pi_p)$,
where $\pi_0:=\Tt|_{P_0}$, one obtains the charts 
\be\label{othercharts} \psi_p:\pi_0^{-1}(\Pi_p)\stackrel{\sim}{\to}(1-p)\M p\times \G_{p_0}^p(\M)\ee
which define the atlas $(\pi_0^{-1}(\Pi_p),\psi_p),\ p\in \Ll_{p_0}(\M)$, on $P_0$. For  $y_p\in (1-p)\M p$, $z_{pp_0}\in \G_{p_0}^p(\M)$ and $\eta\in \pi_0^{-1}(\Pi_p)$ one has 
\be\label{eta} \eta=\psi_p^{-1} (y_p,z_{pp_0})=(p+y_p)z_{pp_0}\ee
and $\psi_p(\eta)=(y_p,z_{pp_0})$, where
\be\begin{array}{l}\label{1yp} y_p=\eta(p\eta)^{-1}-p\\
 z_{pp_0}=p\eta\end{array}.\ee

Let us note that $P_0$ is an open subset of the Banach space $\M p_0$ and, thus one can consider $(P_0, id)$ as a chart on $P_0$. The transition maps (\ref{eta}) and (\ref{1yp}) show that $(P_0,id)$ belongs to the maximal atlas of the manifold $P_0$ defined by $(\pi_0^{-1}(\Pi_p),\psi_p),\ p\in \Ll_{p_0}(\M)$. Hence we conclude that the topologies of $P_0$ inherited from $\G_{p_0}(\M)$ and from $\M p_0$  are the same. The free right actions of the Banach-Lie group $G_0:=G(p_0\M p_0)$ of the invertible elements of the $W^*$-subalgebra $p_0\M p_0$ on $P_0$  and on $P_0\times P_0$ are defined by 
    \be\label{Gaction} \kappa: P_0\times G_0\ni(\eta,g)\mapsto \eta g\in P_0\ee
		and  by     
		\be\label{Gaction2} \kappa_2:P_0\times P_0\times G_0\ni(\eta,\xi,g)\mapsto (\eta g, \xi g)\in P_0\times P_0,\ee
respectively. They are consistent with the atlas $(\pi_0^{-1}(\Pi_p),\psi_p),\ p\in \Ll(\M)$ defined in (\ref{othercharts}), and  the atlas $(\pi_0^{-1}(\Pi_p)\times \pi_0^{-1}(\Pi_{\tilde{p}}),\psi_p\times \psi_{\tilde{p}}),\ p,\tilde p\in \Ll(\M)$, defined by
\be\label{119} (\pi_0^{-1}(\Pi_p)\times \pi_0^{-1}(\Pi_{\tilde{p}})\ni (\eta,\xi)\mapsto (y_p,z_{pp_0},y_{\tilde p},z_{\tilde p p_0}), \ee
i.e.  one has $\psi_p(\eta g)=(y_p,z_{pp_0}g)$ and $(\psi_p\times\psi_q)(\eta g,\xi g)=(y_p,z_{pp_0}g,y_q,z_{qq_0}g)$. The orbits of $G_0$ on $P_0$ and $P_0\times P_0$ coincide with the fibres of the submersions
\be\label{1gauge} \varphi:P_0\ni\eta\mapsto \varphi(\eta):=\eta\eta^{-1}\in \Ll_{p_0}(\M),\ee
\be\label{2gauge} \phi:P_0\times P_0\ni(\eta,\xi)\mapsto \phi(\eta,\xi):=\eta\xi^{-1}\in \G_{p_0}(\M).\ee
So, the equivalence relations defined by the actions (\ref{Gaction}) and (\ref{Gaction2}) are regular in sense of the definition given in 5.9.5 of \cite{bou}. Thus the quotient spaces $P_0/G_0$ and $\frac{P_0\times P_0}{G_0}$ are Banach manifolds isomorphic to $\Ll_{p_0}(\M)$ and $\G_{p_0}(\M)$, respectively.

Let us consider the pair groupoid $P_0\times P_0\tto P_0$ and the action groupoid $P_0\rtimes G_0\tto P_0$ which are, as one can easily see, a Banach-Lie groupoids. The definition of the action groupoid one can find in \cite{mac}. We define $\iota:P_0\rtimes G_0\to P_0\times P_0$ by
\be \iota(\eta,g):=(\eta g,\eta).\ee
\begin{prop}\label{prop:14}
\begin{enumerate}[(i)]
\item One has the following (non exact) sequence of groupoid morphisms
	  \unitlength=5mm \be\label{1propduzy}\begin{picture}(11,4.6)
        \put(-2,4){\makebox(0,0){$P_0\rtimes G_0$}}
 \put(-2,-1){\makebox(0,0){$P_0$}}
		\put(5,4){\makebox(0,0){$P_0\times P_0$}}
       \put(5,-1){\makebox(0,0){$P_0$}}
 \put(12,4){\makebox(0,0){$\G_{p_0}(\M)$}}
    \put(12,-1){\makebox(0,0){$\Ll_{p_0}(\M).$}}
   
    \put(5.2,3){\vector(0,-1){3}}
    \put(4.7,3){\vector(0,-1){3}}
    \put(-2.3,3){\vector(0,-1){3}}
    \put(-1.8,3){\vector(0,-1){3}}
    \put(12.2,3){\vector(0,-1){3}}
    \put(11.7,3){\vector(0,-1){3}}
    \put(7,4){\vector(1,0){2.7}}
    \put(6.3,-1){\vector(1,0){4}}
 \put(-0.8,-1){\vector(1,0){4.5}}
 \put(-0.5,4){\vector(1,0){3}}
      \put(8.5,4.4){\makebox(0,0){$\phi$}}
    \put(8.5,-0.5){\makebox(0,0){$\varphi$}}
\put(1.5,-0.5){\makebox(0,0){$ id$}}
\put(1.5,4.4){\makebox(0,0){$\iota$}}
    \end{picture}\ee		
\bigskip\newline
where the pairs of maps $(\iota,id)$ and $(\phi,\varphi)$ define groupoid monomorphism and epimorphism, respectively.
\item The quotient groupoid (gauge groupoid) $\frac{P_0\times P_0}{G_0}\tto P_0/G_0$ and the groupoid $\G_{p_0}(\M)\tto\Ll_{p_0}(\M)$ are isomorphic, where the isomorphism
  \unitlength=5mm \be\label{gaugeisom}\begin{picture}(11,4.6)
    \put(1,4){\makebox(0,0){$\frac{P_0\times P_0}{G_0}$}}
    \put(8,4){\makebox(0,0){$\G_{p_0}(\M)$}}
    \put(1,0){\makebox(0,0){$P_0/G_0$}}
    \put(8,0){\makebox(0,0){$\Ll_{p_0}(\M)$}}
    \put(1.2,3){\vector(0,-1){2}}
    \put(0.7,3){\vector(0,-1){2}}
    \put(8.2,3){\vector(0,-1){2}}
    \put(7.7,3){\vector(0,-1){2}}
    \put(3,4){\vector(1,0){3}}
    \put(2.7,0){\vector(1,0){3.7}}
    \put(0.1,1.4){\makebox(0,0){$\ $}}
    \put(2.2,1.4){\makebox(0,0){$\ $}}
    \put(9.1,1.9){\makebox(0,0){$\ $}}
    \put(6.8,1.9){\makebox(0,0){$\ $}}
    \put(4.5,4.5){\makebox(0,0){$[\phi]$}}
    \put(4.5,0.5){\makebox(0,0){$[\varphi ]$}}
    \end{picture},\ee
    \bigskip
		\ 
		is given by the quotienting of (\ref{1gauge}) and (\ref{2gauge}).
\item The quotient groupoid $\frac{P_0\rtimes G_0}{G_0}\tto P_0/G_0$ of $P_0\rtimes G_0\tto P_0$ is isomorphic to the inner subgroupoid $\J_{p_0}(\M)\tto\Ll_{p_0}(\M)$ of the groupoid $\G_{p_0}(\M)\tto\Ll_{p_0}(\M)$.\end{enumerate}
All groupoid morhisms mentioned above are morphisms of Banach-Lie groupoids.
\end{prop}
\begin{proof} Straightforward  after observation that all arrows in (\ref{1propduzy}) are given by $G_0$-equivariant maps.
\end{proof}

We note that $\varphi=\pi_0:P_0\to \Ll_{p_0}(\M)$ and $\phi:P_0\times P_0\to \G_{p_0}(\M)$ are the projections on bases of the $G_0$-principal bundles.

The map $\phi:P_0\times P_0\ni (\eta,\xi)\mapsto \eta\xi^{-1}=x\in \G_{p_0}(\M)$ written in the coordinates (\ref{psipp})   and coordinates (\ref{119}) assumes the form 
		\be (y_p,z_{pp_0},\tilde y_{\tilde p},\tilde z_{\tilde pp_0})\mapsto (y_p,z_{pp_0}\tilde z_{\tilde pp_0}^{-1},\tilde y_{\tilde p})=(y_p,z_{p\tilde p},\tilde y_{\tilde p}). \ee

   In the subsequent considerations we will identify $\G_{p_0}(\M)\tto \Ll_{p_0}(\M)$  with the gauge groupoid $\frac{P_0\times P_0}{G_0}\tto P_0/G_0$. 
	
	\bigskip

	\section{Atiyah sequence of the principal bundle $P_0\to P_0/G_0$ }\label{sec:tangent}
	
	Following of \cite{OJS} we describe the Atiyah sequence of the principal bundle $P_0\to P_0/G_0$ as well as the principal bundle $P_0\times P_0\to \frac{P_0\times P_0}{G_0}$. 	Isomorphic realizations of the tangent groupoid $T\G_{p_0}\tto T\Ll_{p_0}(\M)$ will be presented in Proposition \ref{thm:22}. The atlases on $TP_0$ and on $T\G_{p_0}(\M)\tto T\Ll_{p_0}(\M)$ consistent with their bundle structures will be also described.

\bigskip

We start from the description of  the tangent group $TG_0$ of $G_0$ and the tangent bundle $TP_0$ of $P_0$. We will  identify  $TG_0$ with the semidirect product $p_0\M p_0\rtimes_{Ad_{G_0}} G_0$ of $G_0$ with its Lie algebra $T_e G_0\cong p_0\M p_0$ by 
\be TG_0\ni X_g\mapsto(TR_{g^{-1}}(g)X_g,g)=:(x,g)\in  p_0\M p_0\rtimes_{Ad_{G_0}}G_0.\ee
 Thus the group product $$ X_g\bullet Y_h=TL_g(h)Y_h+TR_h(g)X_g$$ of $X_g\in T_gG_0$ and  $Y_h\in T_h G_0$  can be written as
\be (x,g)\bullet (y,h)=(x+gyg^{-1},gh),\ee
where $x=TR_{g^{-1}}(g)X_g$ and $y=TR_{h^{-1}}(h)Y_h.$
One has the short exact sequence 
\be\label{short1} \{e\}\to T_eG_0\to TG_0\to G_0\to \{e\}\ee
of groups
which is isomorphic to the following one
\be \{e\}\to p_o\M p_0\to p_0\M p_0\rtimes_{Ad_{G_0}} G_0\to G_0\to \{e\}.\ee
We note that $T_eG_0\cong p_0\M p_0$ and $G_0$ are subgroups of $TG_0$.

The inclusion map $\iota:P_0\hookrightarrow \M p_0$ maps $P_0$ on the open subset of the Banach space $\M p_0$. So, its tangent map $T\iota:TP_0\stackrel{\sim}{\to}\M p_0\times P_0$ defines a chart on $TP_0$ with $(v,\eta)\in \M p_0\times  P_0$ as the coordinates of an element of $TP_0$.

The actions (\ref{Gaction}) and (\ref{Gaction2})  of $G_0$ on $P_0$ and on $P_0\times P_0$ define the corresponding actions of $TG_0$ on $TP_0$ and on $T(P_0\times P_0)$ which are given by 
  \be\label{tangent_action} T\kappa:TP_0\times TG_0 \ni ((\vartheta,\eta),(x,g))\mapsto (\vartheta g+\eta x g, \eta g)\in TP_0,\ee

and by
    \be\label{tangentprod2}  T\kappa_2:(TP_0\times TP_0)\times TG_0\ \ni\ \left((\vartheta,\eta),(\omega,\xi),(x,g)\right)\mapsto\\
		 \left((\vartheta g+\eta x g,\eta g),(\omega g+\xi x g,\xi g)\right)\ \in\ TP_0\times TP_0.\ee
    where $(\vartheta,\eta), \ (\omega,\xi) \in\M p_0\times P_0$ and  	$(x,g)\in 	p_0\M p_0\rtimes_{Ad_{G_0}} G_0$.
		Orbits of the normal subgroup $p_0\M p_0\cong T_eG_0\subset TG_0$ are the affine subspaces $\{(v+\eta x,\eta):x\in p_0\M p_0\}\subset T_\eta P_0$ of the tangent space $T_\eta P_0$ at $\eta\in P_0$. Thus the vertical tangent subspace  $T^VP_0\subset TP_0$ consists of the orbits generated from $(0,\eta)\in TP_0$, i.e. $T_\eta^V P_0=\{(\eta x,\eta):\ x\in p_0\M p_0\}$.
				
		It follows from (\ref{split}) that for any $\eta\in \pi_0^{-1}(\Pi_p)$ one has the Banach splitting 
		\be\label{1split} T_\eta P_0 \cong (q\M p_0\oplus (1-p)\M p_0)\times \{\eta\}\ee
		  of the tangent space $T_\eta P_0\cong \M p_0\times \{\eta\}$, where $q=\eta\eta^{-1}$, and $T_\eta^V P_0 \cong q\M p_0\times\{\eta\}=\eta \M p_0\times \{\eta\}$. Using (\ref{1split}) we decompose $(v,\eta)\in T(\pi_0^{-1}(\Pi_p))$ on $(v^V(q)+v^h(q),\eta)$ where $v^V(q)\in  q\M p_0$ and $v^h(q)\in (1-p)\M p_0$, and obtain in this way the local trivialization 
			\be \tau_p:T(\pi_0^{-1}(\Pi_p))\ni (v,\eta)\mapsto (\sigma_p^{-1}(q)v^V(q)+v^h(q),\eta)\in \left(p\M p_0\oplus (1-p)\M p_0\right)\times \pi_0^{-1}(\Pi_p)\ee
			of $TP_0$ 			which satisfies 
			\be \tau_p(T(\pi_0^{-1}(\Pi_p))\cap T^VP_0)=p\M p_0\times \pi_0^{-1}(\Pi_p).\ee
			Hence according to 7.5.1 in \cite{bou} the vertival bundle $T^VP_0$ is a Banach vector subbundle of $TP_0$. So, from 7.5.2 of \cite{bou} it follows that the quotient bundle $TP_0/T^VP_0\cong TP_0/T_eG_0$ is a Banach vector bundle over $P_0$. The above facts one summarizes as the short exact sequence 
			\be\label{short2} T^VP_0\longrightarrow TP_0\longrightarrow TP_0/T_eG_0\ee
			of the Banach vector bundles over $P_0$.
			
			In the following we will identify $p_0\M p_0\times  P_0$ with $T^VP_0$ by the vector bundle isomorphism
			\be\label{Iizo} I:p_0\M p_0\times  P_0\ni(x,\eta)\mapsto I(x,\eta):=(\eta x, \eta)\in T^VP_0.\ee
			Let us note here that the quotient map $A:TP_0\to TP_0/T_eG_0$ is defined as follows
			\be\label{Amorf}\M p_0\times P_0\ni (v,\eta)\mapsto A(v,\eta):=\{(v+\eta x,\eta):\ x\in p_0\M p_0\}.\ee
			
			The action (\ref{tangent_action}) restricted to the subgroup $G_0\subset TG_0$ is regular and preserves the structure of (\ref{short2}).  So, quotienting (\ref{short2}) by $G_0$ one obtains the short exact sequence of the Banach vector bundles 		
				  \unitlength=5mm \be\label{Atiyah11}\begin{picture}(11,4.6)
    \put(-3,4){\makebox(0,0){$p_0\M p_0\times _{Ad_{G_0}} P_0$}}
		\put(5,4){\makebox(0,0){$TP_0/G_0$}}
    \put(12,4){\makebox(0,0){$T(P_0/G_0)$}}
    \put(12,0){\makebox(0,0){$P_0/G_0$}}
    \put(5,0){\makebox(0,0){$P_0/G_0$}}
    \put(-2,0){\makebox(0,0){$P_0/G_0$}}
    
    \put(5,3){\vector(0,-1){2}}
       \put(-2,3){\vector(0,-1){2}}
       \put(12,3){\vector(0,-1){2}}
       \put(7,4){\vector(1,0){2.7}}
    \put(6.3,0){\vector(1,0){4}}
 \put(-0.8,0){\vector(1,0){4.5}}
 \put(0,4){\vector(1,0){3}}
    \put(8.5,4.4){\makebox(0,0){$a$}}
    \put(8.5,0.5){\makebox(0,0){$\sim$}}
\put(1.5,0.5){\makebox(0,0){$ \sim$}}
\put(1.5,4.4){\makebox(0,0){$\iota$}}
    \end{picture}\ee		
\bigskip\newline
		over $P_0/G_0$, which is the Atiyah sequence of the principal bundle $\pi_0:P_0\to P_0/G_0$.  One can find the definition of Atiyah sequence  of a principal bundle, e.g.  in \cite{Atiyah, mac}. In order to  obtain (\ref{Atiyah11}) we have used the bundles morphisms
		\be\label{1iota} \iota:=[I]:p_0\M p_0\times_{Ad_{G_0}}P_0\to T^VP_0/G_0,\ee
\be\label{1a} a:=[A]: TP_0/G_0\to (TP_0/T_{e}G_0)/G_0\cong TP_0/TG_0\cong T(P_0/G_0)\ee
which follow from (\ref{tangent_action}). 

The same argumentation applied to the action  of $G_0$ on $P_0\times P_0$  defined in (\ref{tangentprod2}) leads to the Atiyah sequence 
   \unitlength=5mm\be\label{Atiyah211}\begin{picture}(11,4.6)
    \put(-4,4){\makebox(0,0){$p_0\M p_0\times _{Ad_{G_0}}( P_0\times P_0)$}}
		\put(5,4){\makebox(0,0){$\frac{T( P_0\times P_0)}{G_0}$}}
    \put(12,4){\makebox(0,0){$T\left(\frac {P_0\times P_0}{G_0}\right)$}}
    \put(12,-1){\makebox(0,0){$\frac{P_0\times P_0}{G_0}$}}
    \put(5,-1){\makebox(0,0){$\frac{P_0\times P_0}{G_0}$}}
    \put(-2,-1){\makebox(0,0){$\frac{P_0\times P_0}{G_0}$}}
    
    \put(5,3){\vector(0,-1){3}}
       \put(-2,3){\vector(0,-1){3}}
       \put(12,3){\vector(0,-1){3}}
       \put(7,4){\vector(1,0){2.7}}
    \put(6.3,-1){\vector(1,0){4}}
 \put(-0.8,-1){\vector(1,0){4.5}}
 \put(0,4){\vector(1,0){3}}
    \put(8.5,4.4){\makebox(0,0){$a_2$}}
    \put(8.5,-0.5){\makebox(0,0){$\sim$}}
\put(1.5,-0.5){\makebox(0,0){$ \sim$}}
\put(1.5,4.4){\makebox(0,0){$\iota_2$}}
    \end{picture}\ee		
\bigskip\newline
of the $G_0$-principal bundle $\pi_{02}:P_0\times P_0\to \frac{P_0\times P_0}{G_0}$, where $\iota_2$ and $a_2$ are defined by the quotienting of 
\be\label{1I2} I_2: p_0\M p_0\times P_0\times P_0\ni(x,\eta,\xi)\mapsto (\eta x,\eta,\xi x,\xi)\in TP_0\times TP_0\ee
and
\be\label{1A2} A_2:TP_0\times TP_0\ni(v,\eta,w,\xi)\mapsto \{(v+\eta x,\eta,w+\xi x,\xi);\ x\in p_0\M p_0\}\in\frac{TP_0\times TP_0}{T_eG_0},\ee
 respectively.

\bigskip

The proposition given below presents the equivalent representations of the tangent groupoid $T\G_{p_0}(\M)\tto T\Ll_{p_0}(\M)$.
\begin{prop}\label{thm:22}
One has  the following groupoid  isomorphisms 
	  \unitlength=5mm\be\label{duzy}\begin{picture}(11,4.6)
    \put(5,4){\makebox(0,0){$T(\frac{P_0\times P_0}{G_0})$}}
    \put(12,4){\makebox(0,0){$T\G_{p_0}(\M)$}}
    \put(12,-1){\makebox(0,0){$T\Ll_{p_0}(\M).$}}
    \put(5,-1){\makebox(0,0){$T(\frac{P_0}{G_0})$}}
    \put(-2,-1){\makebox(0,0){$\frac{TP_0}{TG_0}$}}
    \put(-2,4){\makebox(0,0){$\frac{TP_0\times TP_0}{TG_0}$}}
    \put(5.2,3){\vector(0,-1){3}}
    \put(4.7,3){\vector(0,-1){3}}
    \put(-2.3,3){\vector(0,-1){3}}
    \put(-1.8,3){\vector(0,-1){3}}
    \put(12.2,3){\vector(0,-1){3}}
    \put(11.7,3){\vector(0,-1){3}}
    \put(7,4){\vector(1,0){2.7}}
    \put(6.3,-1){\vector(1,0){4}}
 \put(-0.8,-1){\vector(1,0){4.5}}
 \put(-0.5,4){\vector(1,0){3}}
      \put(8.5,4.4){\makebox(0,0){$T[\phi]$}}
    \put(8.5,-0.5){\makebox(0,0){$T[\varphi]$}}
\put(1.5,-0.5){\makebox(0,0){$ \delta$}}
\put(1.5,4.4){\makebox(0,0){$\Delta$}}
    \end{picture}\ee		
\bigskip\newline
\end{prop}
\begin{proof}
One has the canonically defined isomorphism between the tangent groupoid $T(P_0\times P_0)\tto TP_0$ of the pair groupoid $P_0\times P_0\tto P_0$ and the pair groupoid $TP_0\times TP_0\tto TP_0$. Therefore using vector bundles isomorphisms 
  \be\label{TGaction} \delta:TP_0/TG_0\stackrel{\sim}{\longrightarrow} T\left(P_0/G_0\right),\ee
	\be\label{Delta} \Delta:\frac{TP_0\times TP_0}{TG_0}\stackrel{\sim}{\longrightarrow} T\left(\frac{P_0\times P_0}{G_0}\right)\ee
	defined by the actions (\ref{tangent_action}) and (\ref{tangentprod2}) and applying the tangent functor to  the groupoid isomorphism given in (\ref{gaugeisom}) we obtain the isomorphisms mentioned in (\ref{duzy}). \end{proof}

		Now we  define the atlas on $TP_0$ consistent with the principal bundle structure of $P_0$. To this end we will use the charts  (\ref{othercharts}) defined by (\ref{eta}). In order to find the explicit formula for the chart $T\psi_p=T(\pi_0^{-1}(\Pi_p))\to (1-p)\M p\times p\M p_0\times (1-p)\M p\times \G^p_{p_0}(\M)$ tangent to $\psi_p$ we consider a smooth curve 
		\be\label{curve} ]-\varepsilon,\varepsilon[\ \ni\ t\mapsto \eta(t)=(p+y_p(t))z_{pp_0}(t)\ \in\  \pi_0^{-1}(\Pi_p),\ee
		such that $\eta(0)=\eta\in P_0$. Differentiating (\ref{curve}) with respect the parameter $t$ at $t=0$ and defining 
		
			\be\label{ap1}   a_p:=\frac{d}{dt} y_p(t)|_{t=0}\in (1-p)\M p\ee
			 \be\label{bp1} b_p:=\frac{d}{dt} z_{pp_0}(t)|_{t=0} z_{pp_0}^{-1}\in p\M p\ee
			we obtain 
			\be\label{v} v=(a_p+(p+y_p)b_p)z_{pp_0}.\ee
			Let us note here that equalities $\eta(0)=\eta$ and 
				\be\label{v2} v=\frac{d}{dt}\eta(t)|_{t=0}\in \M p_0\ee 
				define the bundle isomorphism $TP_0\cong \M p_0\times P_0$.
				
				By simple calculations we can invert formulas (\ref{eta}) and (\ref{v}) obtaining 
		\be\label{transa} a_p=(v-\eta(p\eta)^{-1}v)(p\eta)^{-1},\ee
\be\label{transb}  b_p=pv(p\eta)^{-1}.\qquad\qquad\qquad\ee
Therefore  the formulas (\ref{transa}) and (\ref{transb}) taken togather with (\ref{1yp}) define the chart
\be\label{ch} T(\pi_0^{-1}(\Pi_p))\ni(v,\eta)\mapsto T\psi_p(v,\eta)=(a_p,b_p,y_p,z_{pp_0})\ee
tangent to the chart $( \pi_0^{-1}(\Pi_p),\psi_p)$ defined in (\ref{1yp}). 

Note here that in (\ref{v}) we used $b_p$  instead of $\frac{d}{dt} z_{pp_0}(t)|_{t=0}$.		
Note also that the product $p\eta$ of $p$ and $\eta$ is not the product in sense of the groupoid $\G(\M)\tto\Ll(\M)$ multiplication. However, $p\eta\in \G(\M)$ and thus  one can take  its groupoid inverse $(p\eta)^{-1}$.

In order to find the transition map $(T\psi_{p'}\circ T\psi_p^{-1}):(a_p,b_p,y_p,z_{pp_0})\mapsto(a_{p'},b_{p'},y_{p'},z_{p'p_0})$ we will use for $(v,\eta)\in (T(\pi_0^{-1}(\Pi_p\cap \Pi_{p'})$ the equalities:
\be\label{tran1} \eta=(p+y_p)z_{pp_0}=(p'+y_{p'})z_{{p'}p_0}\qquad\qquad\qquad\ee
\be\label{tran2} v=[a_p+(p+y_p)b_p]z_{pp_0}=[a_{p'}+(p'+y_{p'})b_{p'}]z_{p'p_0}\ee
which follow from (\ref{eta}) and (\ref{v}), respectively. Solving equations (\ref{tran1}) and (\ref{tran2}) with respect to $(a_{p'},b_{p'},y_{p'},z_{p'p_0})$ we obtain:
\be\label{2transa} a_{p'}=(d-(b+dy_p)(a+cy_p)^{-1})a_p(a+cy_p)^{-1}\ee
\be\label{2transb} b_{p'}=(ca_p+(a+cy_p)b_p)(a+cy_p)^{-1}\qquad\qquad\ee
\be\label{2transy}  y_{p'}=(b+dy_p)(a+cy_p)^{-1}\qquad\qquad\qquad\qquad\ee
\be z_{p'p_0}=(a+cy_p)z_{pp_0}.\qquad\qquad\qquad\qquad\quad\qquad\ee

The chart (\ref{ch}) is equivariant with respect to the action of $G_0$ on $TP_0$ defined as the restriction of  (\ref{tangent_action})  to the subgroup $G_0\subset TG_0$ and the action on $(1-p)\M p\times p\M p\times (1-p)\M p\times p\M p_0$ is defined for $g\in G_0$ by $(a_p,b_p,y_p,z_{pp_0})\mapsto (a_p,b_p,y_p,z_{pp_0}g)$. So, after quotienting by $G_0$ one defines the chart 
\be T(\pi_0^{-1}(\Pi_p))/_{G_0}\ni\langle v,\eta\rangle\mapsto [T\psi_p](\langle v,\eta\rangle)=(a_p,b_p,y_p).\ee
The transition map $[T\psi_{p'}]\circ [T\psi_p]^{-1}:(a_p,b_p,y_p)\mapsto (a_{p'},b_{p'},y_{p'})$ between these charts is given by (\ref{2transa}-\ref{2transy}).

We end this section describing the atlas on $T\G_{p_0}(\M)\tto T\Ll_{p_0}(\M)$ consistent with the one defined by (\ref{psippmap}) and (\ref{psipp}). So, additionaly to (\ref{curve}) we define 
\be \tilde a_{\tilde p}:=\frac{d}{dt}\tilde y_{\tilde p}(t)|_{t=0}\in (1-\tilde p)\M \tilde p,\ee
\be\label{bptp}  b_{p\tilde p}:=\frac{d}{dt} z_{p\tilde p}(t)|_{t=0}\in p\M \tilde p,\ee
where $]-\varepsilon,\varepsilon[\ni t\mapsto x(t)=(p+y_p(t))z_{p\tilde p(t)}(\tilde p+\tilde y_{\tilde p}(t))^{-1}\in \Omega _{p\tilde p}$. Hence we obtain  the atlas on $T\G_{p_0}(\M)$ given by the coordinates $(a_p,  b_{p\tilde p}, \tilde a_{\tilde p}, y_p, z_{p\tilde p}, \tilde y_{\tilde p}, )$   tangent to $(y_p, z_{p\tilde p}, \tilde y_{\tilde p})$, enumerated by $(p,\tilde p)\in \Ll_{p_0}(\M)\times \Ll_{p_0}(\M)$. Using the tangent map of the transition map presented in  (\ref{atlas2})  one obtains the map
\be\begin{array}{l}  a_{p'}=da_p(a+cy_p)^{-1}-(b+dy_p)(a+cy_p)^{-1}ca_p(a+cy_p)^{-1}\\
 b_{p'\tilde p'}=ca_pz_{p\tilde p}(\tilde a+\tilde c\tilde y_{\tilde p})^{-1}+(a+cy_p)b_{p\tilde p}(\tilde a+\tilde c\tilde y_{\tilde p})^{-1}-(a+cy_p)z_{p\tilde p}(\tilde a+\tilde c\tilde y_{\tilde p})^{-1}\tilde c\tilde a_{\tilde p}(\tilde a+\tilde c\tilde y_{\tilde p})^{-1}\\
 \tilde  a_{\tilde p'}=\tilde d\tilde a_{\tilde p}(\tilde a+\tilde c\tilde y_{\tilde p})^{-1}-(\tilde b+\tilde d\tilde y_{\tilde p})(\tilde a+\tilde c\tilde y_{\tilde p})^{-1}\tilde c\tilde a_{\tilde p}(\tilde a+\tilde c\tilde y_{\tilde p})^{-1}\end{array}\ee 
which  together with (\ref{atlas2}) gives the  transition map from the coordinates $(a_p,  b_{p\tilde p}, \tilde a_{\tilde p}, y_{p}, z_{p\tilde p}, \tilde y_{\tilde p})$ to the coordinates $( a_{p'},  b_{p'\tilde p'}, \tilde a_{\tilde p'}, y_{p'}, z_{p'\tilde p'}, \tilde y_{\tilde p'})$. The above formulas will be used in the sequel for the coordinate expressions of the various structures and dependences between of them.


\section{Atiyah sequence of the groupoid $\G(\M)\tto\Ll(\M)$ }\label{sec:Atiyah}
In this section we discuss  some questions which arise in a natural way when one considers the algebroid $\A\G(\M)$ of the groupoid $\G(\M)\tto\Ll(\M)$ of partially invertible elements of a $W^*$-algebra $\M$.

We start by defining the following Banach vectors bundles over the lattice $\mathcal L(\M)$ of the orthogonal projectors of $\M$:
\be \A(\M):=\{(x,q)\in \M\times\mathcal L(\M):\quad x\in q\M q\}\ee		
\be\mathcal M^L(\M):=\{(x,q)\in \M\times\mathcal L(\M):\quad x\in\M q\}\ee
\be \T(\M):=\{(x,q)\in \M\times\mathcal L(\M):\quad x\in (1-q)\M q\}.\ee

From the Banach splitting
\be\label{split2} \M q=q\M q\oplus (1-q)\M q\ee
of the left ideal $\M q$ of $\M$ taken for any $q\in \Ll(\M)$ we find that these bundles form the following short exact sequence 
 \unitlength=5mm \be\label{Atiyah1}\begin{picture}(11,4.6)
    \put(-2,4){\makebox(0,0){$\A(\M)$}}
    \put(5,4){\makebox(0,0){$\mathcal M^L(\M)$}}
    \put(12,4){\makebox(0,0){$\T(\M)$}}
    \put(12,0){\makebox(0,0){$\mathcal L(\M),$}}
    \put(5,0){\makebox(0,0){$\mathcal L(\M)$}}
    \put(-2,0){\makebox(0,0){$\mathcal L(\M)$}}

    \put(-2,3){\vector(0,-1){2}}
    \put(5,3){\vector(0,-1){2}}
    \put(12,3){\vector(0,-1){2}}
    \put(7,4){\vector(1,0){2.7}}
    \put(6.3,0){\vector(1,0){4}}
 \put(-0.8,0){\vector(1,0){4.5}}
 \put(-0.5,4){\vector(1,0){3}}
    \put(1.5,4.4){\makebox(0,0){$\iota$}}
     \put(8.5,4.4){\makebox(0,0){$a$}}
    \put(8.5,0.5){\makebox(0,0){$\sim$}}
\put(1.5,0.5){\makebox(0,0){$ \sim$}}
  \put(-1,2){\makebox(0,0){$\pi_\mathcal{A}$}}
    \put(6,2){\makebox(0,0){$\pi_\mathcal{M}$}}
\put(13,2){\makebox(0,0){$ \pi_\mathcal{T}$}}
    \end{picture}\ee
where the bundle monomorphism $\iota:\A(\M)\to\mathcal M^L(\M)$ and the bundle epimorphism $a:\mathcal M^L(\M)\to \T(\M)$ are   defined by the inclusions $q\M q\hookrightarrow\M q$ and the projections $\M q\to (1-q)\M q$  of fibres given by the splitting (\ref{split2}). All  projections on the base in (\ref{Atiyah1}) are defined as the projections of the product $\M\times \mathcal L(\M)$ on its second component.

One has the short exact sequence of Banach-Lie groupoids
 \unitlength=5mm \be\label{Atiyah2}\begin{picture}(11,4.6)
    \put(-2,4){\makebox(0,0){$\J(\M)$}}
    \put(5,4){\makebox(0,0){$\G(\M)$}}
    \put(12,4){\makebox(0,0){$\Ll(\M)\times_{\mathcal{R}} \Ll(\M)$}}
    \put(12,0){\makebox(0,0){$\mathcal L(\M),$}}
    \put(5,0){\makebox(0,0){$\mathcal L(\M)$}}
    \put(-2,0){\makebox(0,0){$\mathcal L(\M)$}}

    \put(-2.1,3){\vector(0,-1){2}}
    \put(5.1,3){\vector(0,-1){2}}
    \put(11.9,3){\vector(0,-1){2}}
    \put(-1.9,3){\vector(0,-1){2}}
    \put(4.9,3){\vector(0,-1){2}}
    \put(12.1,3){\vector(0,-1){2}}
    \put(7,4){\vector(1,0){2}}
    \put(6.3,0){\vector(1,0){4}}
 \put(-0.8,0){\vector(1,0){4.5}}
 \put(-0.5,4){\vector(1,0){4}}
    \put(1.5,4.4){\makebox(0,0){$\hookrightarrow$}}
     \put(8.3,4.4){\makebox(0,0){$(\Tt,\Ss)$}}
    \put(8.5,0.5){\makebox(0,0){$\sim$}}
\put(1.5,0.5){\makebox(0,0){$ \sim$}}
    \end{picture}\ee
		 where  $\J(\M)\tto \Ll(\M)$ is the inner subgroupoid of the groupoid $\G(\M)\tto\Ll(\M)$ defined  as usually by
    \be\label{innergr} \J(\M):=\bigcup_{q\in \mathcal L(\M)}\left(\Ss^{-1}(q)\cap\Tt^{-1}(q)\right).\ee
		The groupoid 
		$\Ll(\M)\times_{\mathcal{R}} \Ll(\M):=\{ (q,p)\in \Ll(\M)\times\Ll(\M):\ \ q\sim p\}$ is a subgroupoid of the pair groupoid $\mathcal L(\M)\times \mathcal L(\M)\tto \mathcal L(\M)$.  Recall that the equivalence relation $q\sim p$ is the Murray-von Neumann equivalence of projections.
		Recall also that the inner groupoid $\J(\M)\tto \Ll(\M)$  is totally intransitive. All morphisms between the objects of the diagram are smooth maps with respect to their Banach manifold structures. So, one can consider (\ref{Atiyah2}) as a short exact sequence of Banach-Lie groupoids.
		
		\bigskip

    Since the construction of the algebroid of a Banach-Lie groupoid has functorial property one obtains from (\ref{Atiyah2}) the short exact sequence  of corresponding Banach-Lie algebroids
     \unitlength=5mm \be\label{Atiyahalgebr}\begin{picture}(11,4.6)
    \put(-2,4){\makebox(0,0){$\A\J(\M)$}}
    \put(5,4){\makebox(0,0){$\A\G(\M)$}}
    \put(12,4){\makebox(0,0){$T\Ll(\M)$}}
    \put(12,0){\makebox(0,0){$\mathcal L(\M).$}}
    \put(5,0){\makebox(0,0){$\mathcal L(\M)$}}
    \put(-2,0){\makebox(0,0){$\mathcal L(\M)$}}

    \put(-2,3){\vector(0,-1){2}}
    \put(5,3){\vector(0,-1){2}}
    \put(12,3){\vector(0,-1){2}}
    \put(6.3,0){\vector(1,0){4.5}}
 \put(-0.8,0){\vector(1,0){4.5}}
 \put(-0.5,4){\vector(1,0){3.4}}
  \put(6.3,4){\vector(1,0){4.4}}
    \put(1.5,4.4){\makebox(0,0){$\iota$}}
     \put(8.3,4.6){\makebox(0,0){$a$}}
    \put(8.5,0.5){\makebox(0,0){$\sim$}}
\put(1.5,0.5){\makebox(0,0){$ \sim$}}
    \end{picture}\ee

    Note here that the tangent bundle $T\Ll(\M)\to \Ll(\M)$ is the algebroid of the pair groupoid $\Ll(\M)\times \Ll(\M)\tto \Ll(\M)$. 
		\begin{prop}\label{rem:31} The short exact sequences (\ref{Atiyah1}) and (\ref{Atiyahalgebr}) of the Banach vector bundles are isomorphic in a canonical way.
		\end{prop}
		\begin{proof} One has a canonical inclusion of the bundles 
 \unitlength=5mm \be\label{inclusion}\begin{picture}(11,4.6)
    \put(1,4){\makebox(0,0){$\G(\M)$}}
    \put(8,4){\makebox(0,0){$\mathcal{M}^L(\M)$}}
    \put(1,0){\makebox(0,0){$\Ll(\M)$}}
    \put(8,0){\makebox(0,0){$\Ll(\M)$}}
     \put(0.7,3){\vector(0,-1){2}}
    \put(8.2,3){\vector(0,-1){2}}
       \put(2.5,4){\vector(1,0){3.7}}
    \put(2.7,0){\vector(1,0){3.7}}
    \put(0.2,2){\makebox(0,0){$\Ss$}}
      \put(9,2){\makebox(0,0){$\pi_\mathcal{M}$}}
     \put(4.5,4.3){\makebox(0,0){$\iota$}}
    \put(4.5,0.5){\makebox(0,0){$ id $}}
    \end{picture}\ee
    \newline
    defined by $\iota(x):=(x,\Ss(x))$, where the source map fibre $\Ss^{-1}(q)$ of 	$q\in \mathcal L(\M)$	 is an open subset of  the fibre $\pi_\mathcal{M}^{-1}(q)=\M q$. Thus one obtains the isomorphisms $T_{q}(\Ss^{-1}(q))\cong\M q$ of the space $T_{q}(\Ss^{-1}(q))$ tangent to $\Ss^{-1}(q)$ at $q$ with the fibre $\pi_\mathcal{M}^{-1}(q)$, for details see Proposition 5.2 in \cite{OJS}. From the above  we conclude that $\mathcal M^L(\M)){\to}\mathcal L(\M)$ is isomorphic  with the algebroid $\A\G(\M)\to\Ll(\M)$ of the groupoid $\G(\M)\tto\Ll(\M)$. 
    
		

   The  inner subgroupoid $\J(\M)\tto\Ll(\M)$ can be considered as a bundle $\Ss:\J(\M)\to\Ll(\M)$ of groups $\Ss^{-1}(q)\cap\Tt^{-1}(q)=G(q\M q)$. Similarly to (\ref{inclusion}) one has
    \unitlength=5mm \be\label{inclusion2}\begin{picture}(11,4.6)
    \put(1,4){\makebox(0,0){$\J(\M)$}}
    \put(8,4){\makebox(0,0){$\A(\M)$}}
    \put(1,0){\makebox(0,0){$\Ll(\M)$}}
    \put(8,0){\makebox(0,0){$\Ll(\M)$}}
     \put(0.7,3){\vector(0,-1){2}}
    \put(8.2,3){\vector(0,-1){2}}
       \put(2.5,4){\vector(1,0){3.7}}
    \put(2.7,0){\vector(1,0){3.7}}
    \put(0.2,2){\makebox(0,0){$\Ss$}}
      \put(9,2){\makebox(0,0){$\pi_\A$}}
     \put(4.5,4.3){\makebox(0,0){$\iota$}}
    \put(4.5,0.5){\makebox(0,0){$id $}}
    \end{picture}\ee
    \newline
    where $\Ss^{-1}(q)=G(q\M q)$ is an open subset of $q\M q=\pi_\A^{-1}(q)$. So, using the same arguments as for (\ref{inclusion}), we conclude that $\A(\M)\to\Ll(\M)$ is isomorphic with the algebroid $\A\J(\M)\tto \Ll(\M)$ of the inner subgroupoid.

     Let $]-\varepsilon, \varepsilon[\ni t\mapsto q(t)\in \Pi_q$ be a smooth curve through the point $q=q(0)$. Because of 
		$ a_q:=\frac{d}{dt}\varphi_q(q(t))|_{t=0}\in (1-q)\M q$ one obtains the isomorphism $T_q\Ll(\M)\cong (1-q)\M q$ for any $q\in \Ll(\M)$. Taking into account that the atlas  $(\Pi_q,\varphi_q:\Pi\to (1-q)\M q)$, $q\in \Ll(M)$,  is defined in a canonical way one has the canonical isomorphism between  $\T(\M)\to \Ll(\M)$ and $T\Ll(\M)\to \Ll(\M)$.
		\end{proof}
		Let us also mention that  two of the Banach vector bundles included in (3.5) are equipped with some additional structures:
		
	(i) the bundle $\pi_{\mathcal M}:\mathcal M^L(\M))\to\Ll(\M)$  is a bundle of the left $\M$-modules;
	
	(ii) the bundle $\pi_{\mathcal A}:\mathcal A(\M))\to\Ll(\M)$  is a bundle of $W^*$-algebras.
	
	All above statements are valid in the case if one takes the subgroupoid $\G_{p_0}(\M)\tto \Ll_{p_0}(\M)$ instead of $\G(\M)\tto \Ll(\M)$.
	
	\begin{prop}\label{prop:32}  The Atiyah sequence (\ref{Atiyah11}) is isomorphic with the short exact sequence of the Banach-Lie algebroids
	 \unitlength=5mm \be\label{prop:Atiyah32}\begin{picture}(11,4.6)
    \put(-2.8,4){\makebox(0,0){$\A\J_{p_0}(\M)$}}
    \put(5,4){\makebox(0,0){$\A\G_{p_0}(\M)$}}
    \put(12,4){\makebox(0,0){$\T\Ll_{p_0}(\M)$}}
    \put(12,0){\makebox(0,0){$\Ll_{p_0}(\M),$}}
    \put(5,0){\makebox(0,0){$\Ll_{p_0}(\M)$}}
    \put(-2,0){\makebox(0,0){$\Ll_{p_0}(\M)$}}

    \put(-2,3){\vector(0,-1){2}}
    \put(5,3){\vector(0,-1){2}}
    \put(12,3){\vector(0,-1){2}}
    \put(6.6,0){\vector(1,0){4}}
 \put(-0.8,0){\vector(1,0){4.5}}
 \put(-0.6,4){\vector(1,0){4}}
  \put(6.6,4){\vector(1,0){3.6}}
    \put(1.5,4.4){\makebox(0,0){$\ $}}
     \put(8.3,4.6){\makebox(0,0){$  \ $}}
    \put(8.5,0.5){\makebox(0,0){$\sim$}}
\put(1.5,0.5){\makebox(0,0){$ \sim$}}
    \end{picture}\ee
	\end{prop}	
    \prf{ Let us denote the $G_0$-orbits of $(x,\eta)\in p_0\M p_0\times P_0$ and $(\vartheta,\eta)\in\M p_0\times P_0$  by $\langle x,\eta\rangle \in p_0\M p_0\times_{Ad_{G_0}} P_0$ and $\langle \vartheta,\eta\rangle \in TP_0/G_0$, respectively, and the $TG_0$-orbit of $(\vartheta,\eta)\in TP_0$  by $\langle\langle\vartheta,\eta\rangle\rangle\in TP_0/TG_0$. The maps
 \be I_\A: p_0\M p_0\times_{Ad_{G_0}} P_0\ni\langle x,\eta\rangle\mapsto(\eta x \eta^{-1},\eta\eta^{-1})\in \A_{p_0}(\M)\cong \A\J_{p_0}(\M)\ee
 \be I_{\mathcal{M}}:\frac{\M p_0\times P_0}{G_0}\ni\langle \vartheta,\eta\rangle\mapsto(\vartheta\eta^{-1},\eta\eta^{-1})\in \mathcal{M}^L_{p_0}(\M)\cong \A\G_{p_0}(\M)\ee
\be I_\T: \frac{\M p_0\times P_0}{TG_0}\ni\langle\langle \vartheta,\eta\rangle\rangle\mapsto((1-\eta\eta^{-1})\vartheta\eta^{-1},\eta\eta^{-1})\in \T_{p_0}(\M)\cong T\Ll_{p_0}(\M)\ee
define isomorphisms  between  the corresponding Banach vector bundles appearing in the diagrams (\ref{Atiyah11}) and (\ref{prop:Atiyah32})
and they  commute with the horizontal arrows of these diagrams.  We recall that $\Ll_{p_0}(\M)\cong P_0/G_0$.}
    
		Taking into account Proposition \ref{rem:31} and Proposition \ref{prop:32}  we will call (\ref{Atiyahalgebr}) (as well as (\ref{Atiyah1})) the \textbf{Atiyah sequence} of the groupoid $\G(\M)\tto \Ll(\M)$ of partially invertible elements of $W^*$-algebra $\M$.

    Now following of \cite{OJS} we will present the formula for Lie bracket $[\X_1,\X_2]$ of sections $\X_1,\X_2\in \Gamma^\infty(\AG(\M))\cong \Gamma^\infty(\mathcal M^L(\M))$ of the bundle $\mathcal M^L(\M)\to \Ll(\M)$.
		
		To this end  let us recall that the one-parameter group $L_t\circ L_s=L_{t+s}$ of the left translations $L_t:\G(\M)\to \G(\M)$ of the groupoid $\G(\M)\tto \Ll(\M)$  by the definition has the following properties
	\be\label{lift}\begin{array}{rl}	
	(i)& L_t(xy)=L_{t}(x)y\\	
	(ii) &\Ss\circ L_t=\Ss\\	
	(iii)& \Tt\circ L_t=\lambda_t\circ\Tt,
	\end{array}\ee	
	where $\Ss(x)=\Tt(y)$ and the one-parameter group  $\lambda_t:\Ll(\M)\to \Ll(\M)$ is defined in a unique way by $L_t$. 
	
		From (\ref{lift}) and from $\Tt\circ \sigma_p=id_{\Pi_p}$ we obtain
	\be\label{tL} \Tt(L_t(\sigma_p( q))=(\lambda_t\circ\Tt\circ\sigma_p)( q)=\lambda_t( q)=\Tt(\sigma_p(\lambda_t( q)),\ee
		\be\label{Ss} \Ss(L_t(\sigma_p( q)))=\Ss(\sigma_p(\lambda_t( q))=p.\ee
	It follows from (\ref{tL}) and (\ref{Ss}) that $L_t(\sigma_p(q))$ and $\sigma_p(\lambda_t(q))$ belong to $\Ss^{-1}(p)\cap\Tt^{-1}(\lambda_t(q))$. Thus there exists uniquely defined  $c_p( q,t)\in G(p\M p)$ such that 
		\be\label{colift} L_t(\sigma_p( q))=\sigma_p(\lambda_t( q))c_p(t, q).\ee
	From $L_t\circ L_s=L_{t+s}$ it follows that the cocycle  property
			\be\label{cocycle} c_p( q,t+s)=c_p(\lambda_t( q),s)c_p( q,t)\ee
			is valid for $c_p:]-\varepsilon,\varepsilon[\times \Pi_p\to G(p\M p)$.
			
			\begin{prop}\label{prop:33} One has the following equalities:
			\be\label{propc1} c_p(q,t)=p\ L_t(\sigma_p(q))\ee
			\be\label{propc2} c_p(q,t)=z_{pp_0}(t)z_{pp_0}^{-1}\ee
			\be\label{propc3} z_{p\tilde p}(t)=z_{pp_0}(t)\tilde z_{\tilde pp_0}^{-1}\ee
			\be\label{propc4} b_p=\frac{d}{dt}c_p(q,t)|_{t=0}\ee
			\be\label{propc5} b_{p\tilde p}=b_pz_{p\tilde p}\ee
			where $z_{pp_0}(t)$ and $z_{p\tilde p}(t)$ are defined by 
			\be\label{Lt}L_t(\eta)=(p+y_p(t))z_{pp_0}(t)\ee
			and by
			\be\label{Ltx} L_t(x)=(p+y_p(t))z_{p\tilde p}(t)(\tilde p+\tilde y_{\tilde p}(t))^{-1}.\ee
			For the definition of $b_p$ and $b_{p\tilde p}$ see (\ref{bp1}) and (\ref{bptp}), respectively.\end{prop}
			\begin{proof} Multiplying   of (\ref{colift}) on the left hand side by $p$ and using the equality $p\sigma_p(q)=p$ we obtain (\ref{propc1}). From (\ref{lift}) we have 
			\be\label{prf1} L_t(\eta)=L_t(\sigma_p(q))z_{pp_0}.\ee
			Comparing (\ref{Lt}) with  (\ref{prf1}) and using (\ref{colift}) we obtain 			
						\be\label{prf2} L_t(\eta)=\sigma_p(\lambda_t(q))z_{pp_0}(t)=L_t(\sigma_p(q))z_{pp_0}=\sigma_p(\lambda_t(q))c_p(t,q)z_{pp_0}.\ee
			Canceling $\sigma_p(\lambda_t(q))$  in (\ref{prf2}) we obtain (\ref{propc2}). The equality (\ref{propc3}) follows from $L_t(x)=L_t(\eta)\xi^{-1}$  and from (\ref{Lt}) and (\ref{Ltx}). Equalities (\ref{propc4}) and (\ref{propc5}) are proved by the straightforward checking.
			\end{proof}

\bigskip
	
	From (\ref{lift}) it  follows that the vector field $\tilde\X\in \Gamma^\infty(T\G(\M))$ tangent to the flow $L_t:\G(\M)\to \G(\M)$ satisfies
	\be\label{fild}
	 \tilde \X (xy)=TR_y(x)\tilde\X(x).\ee	
	 Hence one has the one-to-one correspondence
	\be\label{fildL} \tilde \X (x)=TR_x(q)\X(q),\ee
	where $q=\Tt(x)=xx^{-1}$, between $\tilde\X$ and its restriction $\X$ to $ \Ll(\M)\hookrightarrow\G(\M)$ which, because of (\ref{inclusion}), is a section $\X\in \Gamma^\infty(\mathcal M^L(\M))$ of the bundle $\mathcal M^L(\M){\to} \Ll(\M)$. As we see from the definition 
	\be\label{bracket}[\tilde \X_1,\tilde \X_2](x):=\lim_{t\to 0}\left(\tilde\X_2(x)-TL^1_t(L^1_{-t}(x))\tilde\X_2(L^1_t(x))\right)\ee
	the Lie bracket  $[\tilde\X_1,\tilde\X_2]$ of vector fields  $\tilde\X_1$ and $\tilde\X_2$ tangent to $L^1_t$ and $L^2_t$, respectively, satisfies  the conditions (\ref{lift}), and thus the property (\ref{fild}). So, one can define the Lie bracket  $[\X_1,\X_2]$ of  $\X_1,\X_2\in \Gamma^\infty(
\mathcal M^L(\M))$ restricting $[\tilde \X_1,\tilde \X_2]$ to $\Ll(\M)$. 

Let us now express the Lie bracket (\ref{bracket}) in the coordinates $(y_p,z_{p\tilde p}, \tilde y_{\tilde p})$. Using (\ref{Ltx}) we have
\be\label{braketg} \tilde\X(f)=\frac{d}{dt}(f\circ L_t)|_{t=0}=
\left\langle\frac{\partial f}{\partial y_p},\frac{d y_p(t)}{d t}|_{t=0}\right\rangle+
\left\langle\frac{\partial f}{\partial z_{p\tilde p}},\frac{d  z_{p\tilde p}(t)}{d t}|_{t=0}\right\rangle+
\left\langle\frac{\partial f}{\partial \tilde y_{\tilde p}},\frac{d \tilde y_{\tilde p}(t)}{d t}|_{t=0}\right\rangle=\ee
$$=\left\langle\frac{\partial f}{\partial y_p},a_p\right\rangle+
\left\langle\frac{\partial f}{\partial z_{p\tilde p}},b_pz_{p\tilde p}\right\rangle$$
for $f\in C^\infty (\G(\M))$, where in order to obtain the last equality in (\ref{braketg}) we have used (\ref{ap1}), (\ref{bp1}) and (\ref{propc5}) and taken into account the independence of $\tilde y_{\tilde p}(t)=const$ on $t\in ]-\varepsilon, \varepsilon[$. Hence the vector field $\tilde\X$ tangent to the left translation flow $L_t$ written in the coordinates $(y_p,z_{p\tilde p}, \tilde y_{\tilde p})$ assumes the following form 
	\be\label{pole6} \tilde\X=a_p\frac{\partial}{\partial  y_p}+b_pz_{p\tilde p}\frac{\partial}{\partial  z_{p\tilde p}},\ee
	where $a_p(y_p)\in (1-p)\M p$ and $b_p(y_p)\in p\M p$. Restricting $\tilde \X$ to $P_0$ and to $\Ll(\M)$ we obtain 
	\be\label{poleV} \V=a_p\frac{\partial}{\partial  y_p}+b_pz_{p p_0}\frac{\partial}{\partial  z_{p p_0}}\ee
	and 	\be\label{poleL} \X=a_p\frac{\partial}{\partial  y_p}+b_p\frac{\partial}{\partial  z_{p p_0}},\ee
respectively,	where $\V\in \Gamma^\infty _{G_0}TP_0 $ is $G_0$-invariant vector field on $P_0$ and $\X\in \Gamma^\infty\mathcal M^L(\M)$ is a section of $\mathcal M^L(\M)$ defined by (\ref{fildL}). Sections of the vector bundles $T\G(\M)$, $TP_0$ and $\mathcal M^L(\M)$ presented above give the equivalent coordinate representations of sections of the algebroid $\AG(\M)\to \Ll(\M)$. Note here that  the transition map between $(a_p,b_p)$ and $(a_{p'},b_{p'})$ is given in  (\ref{2transa}) and (\ref{2transb}).
	\begin{prop}\label{prop:34}\begin{enumerate}[(i)]
	\item The anchor map $a:\A\Ll(\M)\to T\Ll(\M)$ acts on (\ref{poleL}) as follows:
\be a(\X)=a_p\frac{\partial}{\partial y_p};\ee
\item The vertical part of  (\ref{poleL}) is given by $b_p\frac{\partial}{\partial z_{pp_0}};$
\item The Lie bracket of $\X_1,\X_2\in \Gamma^\infty \mathcal M^L(\M)$assumes the form
\be\label{bra} [\X_1,\X_2]=a_p\frac{\partial}{\partial y_p}+b_p\frac{\partial}{\partial z_{pp}},\ee 
where 
\be a_p=\left\langle\frac{\partial a_{2p}}{\partial y_p},a_{1p}\right\rangle-
\left\langle\frac{\partial a_{1p}}{\partial y_p},a_{2p}\right\rangle\ee
and 
\be b_p=\left\langle\frac{\partial b_{2p}}{\partial y_p},a_{1p}\right\rangle-
\left\langle\frac{\partial b_{1p}}{\partial y_p},a_{2p}\right\rangle+[b_{2p},b_{1p}].\ee

\end{enumerate}
\end{prop}

\begin{proof} The proof can be done by the straightforward verification.
\end{proof}

Below we will also write $\V\in \Gamma^\infty_{G_0}TP_0$ using the global coordinate $\eta\in P_0$:
\be\V=v\frac{\partial}{\partial\eta},\ee
where $v:P_0\to \M p_0$ satisfies $v(\eta g)=v(\eta)g$ for $g\in G_0$. In this representation the Lie bracket of $\V_1,\V_2\in \Gamma^\infty_{G_0}TP_0$ is given by 
\be\label{braketVV} [\V_1,\V_2]=\left(\left\langle\frac{\partial v_2}{\partial\eta},v_1\right\rangle-\left\langle\frac{\partial v_1}{\partial\eta},v_2\right\rangle\right)\frac{\partial}{\partial\eta}.\ee
Note here that $\frac{\partial v}{\partial\eta}(\eta):\M p_0\to \M p_0$.


\section{ Predual Atiyah sequence of the groupoid $\G(\M)\tto\Ll(\M)$ }\label{sec:preAtiyah}

We recall that, by definition, $W^*$-algebra $\M$ is a $C^*$-algebra possessing a predual Banach space $\M_*$  (i.e. $(\M_*)^*=\M$) which is defined in the unique way by the structure of $\M$, e.g. see \cite{sakai}.  Hence, to the Banach vector bundles (algebroids) $\A(\M)\to \Ll(\M)$, $\A\G(\M)\to \Ll(\M)$ and $T\Ll(\M)\to \Ll(\M)$ appearing  in (\ref{Atiyahalgebr})  canonically correspond their predual counterparts $\A_*(\M)\to \Ll(\M)$, $\A_*\G(\M)\to \Ll(\M)$ and $T_*\Ll(\M)\to \Ll(\M)$. These bundles are Banach quasi subbundles 
\be\label{subbundles} \A_*(\M)\subset \A^*(\M), \quad \A_*\G(\M)\subset \A^*\G(\M) \quad {\rm and} \quad T_*\Ll(\M)\subset T^*\Ll(\M)\ee
 of the corresponding dual bundles, i.e. their fibres are Banach subspaces but without Banach complements.

 The bundle morphisms $a^*$ and $\iota^*$  dual to the ones from (\ref{Atiyah1}) preserve the predual subbundles (\ref{subbundles}). Thus their restrictions $a_*$ and $\iota_*$ define the short exact sequence 
  \unitlength=5mm \be\label{Atiyahalgebrdual}\begin{picture}(11,4.6)
    \put(-2,4){\makebox(0,0){$T_*\Ll(\M)$}}
    \put(5,4){\makebox(0,0){$\A_*\G(\M)$}}
    \put(12.5,4){\makebox(0,0){$\A_*\J(\M)$}}
    \put(12,0){\makebox(0,0){$\mathcal L(\M)$}}
    \put(5,0){\makebox(0,0){$\mathcal L(\M)$}}
    \put(-2,0){\makebox(0,0){$\mathcal L(\M)$}}

    \put(-2,3){\vector(0,-1){2}}
    \put(5,3){\vector(0,-1){2}}
    \put(12,3){\vector(0,-1){2}}
    \put(6.3,0){\vector(1,0){4.5}}
 \put(-0.8,0){\vector(1,0){4.5}}
 \put(-0.5,4){\vector(1,0){3.4}}
  \put(6.3,4){\vector(1,0){4.4}}
    \put(1.5,4.4){\makebox(0,0){$a_*$}}
     \put(8.3,4.6){\makebox(0,0){$\iota_*$}}
    \put(8.5,0.5){\makebox(0,0){$\sim$}}
\put(1.5,0.5){\makebox(0,0){$ \sim$}}
    \end{picture}\ee
	of Banach bundles 	which will be the main object of our considerations in this section. Dualizing (\ref{Atiyahalgebrdual}) we return to (\ref{Atiyahalgebr}). So, we will call (\ref{Atiyahalgebrdual}) the\textbf{ predual Atiyah sequence} of $\G(\M)\tto \Ll(\M)$.

				In order to make the above statements precise we note that the left action $L_a x:=ax$ (the right action $R_a x:=xa$) of $\M$ on  itself generates the right action (left action) of $\M$ on the dual $\M^*$
 \be\begin{array}{c}
 \langle R^*_a\varphi,\ x\rangle:= \langle \varphi,\ ax\rangle\\
\langle L^*_a\varphi,\ x\rangle:= \langle \varphi,\ xa\rangle\end{array}, \ee
where $a,\ x\in \M$ and $\varphi\in\M^*$. In a sequel we will write $\varphi a$ and $a\varphi$ instead of  $R^*_a\varphi$ and $L^*_a\varphi$. Using this notation we mention the following isomorphisms
\be\label{M}\begin{array}{c}
 (\M p)^*\cong p\M^*,\\
 (q\M)^*\cong \M^* q,\\
 (q\M p)^*\cong p\M^* q,\end{array} \ee
 where $p,\ q\in \Ll(\M)$ and $(\M p)^*$, $(q\M)^*$ and $(q\M p)^*$ are duals of corresponding Banach subspaces of $\M$. Since the predual Banach space $\M_*$ is a Banach subspace $\M_*\subset\M^*$ of $\M^*$ invariant with respect to $R^*_a$ and $L^*_a$ one has the respective predual maps $R_{*a}:\M_*\to \M_*$ and $L_{*a}:\M_*\to \M_*$ defined as restrictions of $R^*_a$ and $L^*_a$ to $\M_*$. Hence, similarly to (\ref{M})  one has the  isomorphisms
\be\label{Mpredual}\begin{array}{c}
 (\M p)_*\cong p\M_*,\\
 (q\M)_*\cong \M_* q,\\
 (q\M p)_*\cong p\M_* q.\end{array} \ee

The actions of $G_0$ on $T_*P_0\cong p_0\M_*\times P_0$ and on $p_0\M_* p_0\times P_0$, predual to the action of $G_0$  on $TP_0\cong \M p_0\times P_0$ and on $p_0\M p_0\times P_0$ are defined as follows 
	\be\label{groupdual}p_0\M_*\times P_0\ni(\varphi,\eta)\mapsto  \Sigma_{*g}(\varphi,\eta):=(g^{-1}\varphi,\eta g),\ee
\be \label{algdual}p_0\M_* p_0\times P_0\ni(\mathcal X,\eta)\mapsto (g^{-1}\mathcal X g,\eta g).\ee

Let us define a group structure on the precotangent bundle $T_*G_0$ identifying it with the semidirect product $p_0\M_* p_0\rtimes_{Ad^*_{G_0}} G_0$ of groups $G_0$ and $p_0\M_* p_0$, i.e.  the group product on $T_*G_0$ is given by
\be (\mathcal X,g)\star(\mathcal Y, h):=(\mathcal X+g\mathcal Y g^{-1}, gh).\ee
The precotangent group $T_*G_0$ acts on $T_*P_0\cong \M_* p_0\times P_0$ in the following way
\be (\varphi, \eta)\mapsto (g^{-1}(\varphi+\mathcal X\eta^{-1}), \eta g).\ee

\bigskip

The next proposition is the predual version of Proposition \ref{prop:32}.
\begin{prop}\label{prop41}
The predual Atiyah sequence of the principal bundle  $P_0\to P_0/G_0$
 \unitlength=5mm \be\label{Atiyahrestridual}\begin{picture}(11,4.6)
    \put(-2.8,4){\makebox(0,0){$T_*(P_0/G_0)$}}
    \put(5,4){\makebox(0,0){$T_*P_0/G_0$}}
    \put(13,4){\makebox(0,0){$p_0\M_* p_0\times _{Ad_{G^*_0}}P_0$}}
    \put(12,0){\makebox(0,0){$P_0/G_0,$}}
    \put(5,0){\makebox(0,0){$P_0/G_0$}}
    \put(-2,0){\makebox(0,0){$P_0/G_0$}}

    \put(-2,3){\vector(0,-1){2}}
    \put(5,3){\vector(0,-1){2}}
    \put(12,3){\vector(0,-1){2}}
    \put(6.6,0){\vector(1,0){4}}
 \put(-0.8,0){\vector(1,0){4.5}}
 \put(0,4){\vector(1,0){3}}
  \put(6.4,4){\vector(1,0){3.6}}
    \put(1.5,4.4){\makebox(0,0){$a_*$}}
     \put(8.3,4.6){\makebox(0,0){$ \iota_*$}}
    \put(8.5,0.5){\makebox(0,0){$\sim$}}
\put(1.5,0.5){\makebox(0,0){$ \sim$}}
    \end{picture}\ee
		is isomorphic with the short exact sequence 
		 \unitlength=5mm \be\label{dualbundles}\begin{picture}(11,4.6)
    \put(-2.1,4){\makebox(0,0){$\mathcal {T}_*\Ll_{p_0}(\M)$}}
    \put(5,4){\makebox(0,0){$\mathcal {A}_*\G_{p_0}(\M)$}}
    \put(12,4){\makebox(0,0){$\mathcal {A}_*\J_{p_0}(\M)$}}
    \put(12,0){\makebox(0,0){$\Ll_{p_0}(\M),$}}
    \put(5,0){\makebox(0,0){$\Ll_{p_0}(\M)$}}
    \put(-2,0){\makebox(0,0){$\Ll_{p_0}(\M)$}}

    \put(-2,3){\vector(0,-1){2}}
    \put(5,3){\vector(0,-1){2}}
    \put(12,3){\vector(0,-1){2}}
    \put(6.6,0){\vector(1,0){4}}
 \put(-0.8,0){\vector(1,0){4.5}}
 \put(0,4){\vector(1,0){3}}
  \put(6.4,4){\vector(1,0){3.6}}
    \put(1.5,4.4){\makebox(0,0){$a_*$}}
     \put(8.3,4.6){\makebox(0,0){$ \iota_*$}}
    \put(8.5,0.5){\makebox(0,0){$\sim$}}
\put(1.5,0.5){\makebox(0,0){$ \sim$}}
    \end{picture}\ee
i.e. the predual Atiyah sequence of the groupoid $\G_{p_0}(\M)\tto\Ll_{p_0}(\M)$.
\end{prop}
\begin{proof}  The isomorphism between  sequences (\ref{Atiyahrestridual}) and (\ref{dualbundles}) is given by the following isomorphisms 
 \be I_{*\T}: \frac{\M_* p_0\times P_0}{T_*G_0}\ni\langle\langle \varphi,\eta\rangle\rangle\mapsto(\eta\varphi(1-\eta\eta^{-1}),\eta\eta^{-1})\in \T_{*p_0}(\M)\cong T_*\Ll_{p_0}(\M)\ee
\be I_{*\mathcal{M}}:\frac{\M_* p_0\times P_0}{G_0}\ni\langle \varphi,\eta\rangle\mapsto(\eta\varphi,\eta\eta^{-1})\in \mathcal{M}^L_{*p_0}(\M)\cong \A_*\G_{p_0}(\M)\ee
 \be I_{*\A}: p_0\M_* p_0\times_{Ad^*_{G_0}} P_0\ni\langle \mathcal X,\eta\rangle\mapsto( \eta\mathcal X \eta^{-1},\eta\eta^{-1})\in \A_{*p_0}(\M)\cong \A_*\J_{p_0}(\M)\ee
of the corresponding Banach vector bundles, where  by $\langle\langle \varphi,\eta\rangle\rangle$, $\langle \varphi,\eta\rangle$, $\langle \mathcal X,\eta\rangle$ we denote equivalence classes defined by the corresponding group actions. We recall that $\Ll_{p_0}(\M)\cong P_0/G_0$.
\end{proof}

\bigskip


The action tangent to the action (\ref{groupdual}), after taking into account the bundle isomorphism $T(T_*P_0)\cong p_0\M_*\times \M p_0\times (p_0\M_*\times P_0)$, is the following
\be\label{grouptangdual}T\Sigma_{*g} (\varphi,\eta)\xi_{ (\varphi,\eta)}=(g^{-1}\theta,vg, g^{-1}\varphi,\eta g),\ee
where $\xi_{ (\varphi,\eta)}=(\theta,v,\varphi,\eta)\in T_{(\varphi,\eta)}(p_0\M_*\times P_0)\cong (p_0\M_*\times \M p_0)\times\{(\varphi,\eta)\}$.

Let $\pi_*:=pr_2:T_*P_0\cong p_0\M_*\times P_0\to P_0$ be the projection on the bundle  base. Since $T\pi_*\circ T\Sigma_{*g}=\Sigma_g\circ T\pi_*$ one easily sees  that the canonical 1-form
\be\label{form} \langle \gamma_{(\varphi,\eta)}, \xi_{(\varphi,\eta)}\rangle:=\langle {(\varphi,\eta)}, T\pi^*{(\varphi,\eta)}\xi_{(\varphi,\eta)}\rangle=\langle \varphi,v\rangle\ee
and, thus the 2-form $\omega:=d\gamma$
are invariant with respect to (\ref{grouptangdual}). For $\xi^1_{(\varphi,\eta)},\xi^2_{(\varphi,\eta)}\in T_{(\varphi,\eta)} (p_0\M_*\times P_0)$ one has 
\be\label{omega} \omega_{(\varphi,\eta)}(\xi^1_{(\varphi,\eta)},\xi^2_{(\varphi,\eta)})=\langle \theta_1,v_2\rangle - \langle \theta_2,v_1\rangle.\ee
The image of the bundle monomorphism 
\be\label{fibremono} \omega_{(\varphi,\eta)}(\xi_{(\varphi,\eta)},\cdot):T_{(\varphi,\eta)}(p_0\M_*\times P_0)\to T^*_{(\varphi,\eta)}(p_0\M_*\times P_0)\ee
is  a   Banach subspace $\M p_0\times p_0\M_*\times \{(\varphi,\eta)\}\subsetneq \M p_0\times (\M p_0) ^*\times \{(\varphi,\eta)\}\cong T^*_{(\varphi,\eta)}(p_0\M_*\times P_0)$ of the cotangent Banach space at $(\varphi,\eta)$. 
Hence, the 2-form $\omega$ is only a weak symplectic form on $T_*(P_0)$ in sense of the definition  presented for example in \cite{OR}. So, the fibre monomorphisms (\ref{fibremono})  define the bundle quasi immersion  $\flat:T(T_*P_0)\hookrightarrow T^*(T_*P_0) $, i.e.  $T^{\flat}(T_*P_0):=\flat(T(T_*P_0))$ is  a subbundle of the cotangent bundle $T^*(T_*P_0)$ but without the split rang in general, i.e. it is a quasi Banach subbundle.

Now, for $x\in p_0\M p_0$ we define $\xi^x\in \Gamma^\infty  T(T_*P_0)$ by
\be\label{pole2}
\xi^x(f)(\varphi,\eta):=\frac{d}{dt} f\left(\Sigma_{*exp(tx)}(\varphi,\eta)\right) |_{t=0}
\ee
where $f\in C^\infty(T_*P_0)$. One has 
\be\label{413} \omega(\xi^x,\cdot)=-d\langle\gamma,\xi^x\rangle=-d\langle J_0,x\rangle.\ee
The last term in (\ref{413}) contains  the momentum map $J_0:T_*P_0\to p_0\M_* p_0 $  defined as follows
\be\label{momentum} J_0(\varphi,\eta):=\varphi\eta,\ee
i.e. by definition, for any $x\in p_0\M p_0$, one has $\langle J_0(\varphi,\eta),x\rangle:=\langle\varphi,\eta x\rangle$. We note that the equivariance property
\be J_0\circ \Sigma_{*g}=Ad^*_{g^{-1}}\circ J_0\ee
with respect to  the group  $ G_0$ is valid for (\ref{momentum}).

For any $f\in C^\infty(T_*P_0)$, one has $\frac{\partial f}{\partial\eta}(\varphi,\eta)\in (\M p_0)^*$ and $\frac{\partial f}{\partial\varphi}(\varphi,\eta)\in (p_0\M_*)^*\cong \M p_0$ . Therefore, we can define the bracket 
\be\label{Pbracket} \{f,g\}:=\langle\frac{\partial g}{\partial\eta},\frac{\partial f}{\partial\varphi}\rangle-
\langle\frac{\partial f}{\partial\eta},\frac{\partial g}{\partial\varphi}\rangle\ee
of  $f,g \in C^\infty(T_*P_0)$, which is bilinear, anti-symmetric and satisfies the Leibniz property. However, for arbitrary smooth functions on $T_*P_0$ the Jacobi identity for (\ref{Pbracket}) is not fulfilled. For this reason we define the function space
\be\label{calP} \mathcal{P}^\infty(T_*P_0):=\left\{f\in C^\infty(T_*P_0):\frac{\partial f}{\partial \eta}(\varphi,\eta)\in (\M p_0)_*\subset (\M p_0)^*\right\}.\ee
\begin{prop}\label{prop:Palg}
The function  space $(\mathcal{P}^\infty(T_*P_0),\{\cdot,\cdot\})$ is a Poisson algebra with respect to the bracket (\ref{Pbracket}). The derivation $\{f,\cdot\}$ defined by $f\in \mathcal{P}^\infty(T_*P_0)$ is a vector field $\xi_f\in  \Gamma^\infty T(T_*P_0)$ satisfying 
\be\label{omega3} \omega(\xi_f,\cdot)=-df,\ee
i.e. it is a  Hamiltonian with respect to the weak symplectic form (\ref{omega}).\end{prop}

\begin{proof} At first we observe that $\mathcal{P}^\infty(T_*P_0)$ is a subalgebra of the associative algebra  of all smooth functions $C^\infty(T_*P_0)$. In order to prove that $\mathcal{P}^\infty(T_*P_0)$ is closed with respect to the bracket (4.16) we note that 
\be\label{pochodne1} \langle\frac{\partial ^2 f}{\partial\eta\partial \varphi};\dot \eta,\dot \varphi\rangle=\langle\frac{\partial ^2 f}{\partial \varphi\partial\eta};\dot \varphi,\dot \eta\rangle\ee
and
\be\label{pochodne2} \langle\frac{\partial ^2 f}{\partial\eta^2};\dot \eta_1,\dot \eta_2\rangle=\langle\frac{\partial ^2 f}{\partial\eta^2};\dot \eta_2,\dot \eta_1\rangle\ee
for any $\dot \eta,\  \dot \eta_1,\ \dot \eta_2\in \M p_0$ and $\dot  \varphi\in p_0\M_*$. From (\ref{Pbracket}) we find that for any $\dot \eta\in \M p_0$ one has
$$ \langle\frac{\partial\{f,g\}}{\partial\eta},\dot \eta\rangle =
\langle\frac{\partial ^2 g}{\partial \eta^2};\frac{\partial f}{\partial \varphi},\dot \eta\rangle
-\langle\frac{\partial ^2 f}{\partial \eta^2};\frac{\partial g}{\partial \varphi},\dot \eta\rangle
+\langle\frac{\partial ^2 f}{\partial \eta\partial\varphi};\dot \eta,\frac{\partial g}{\partial \eta},\rangle-
\langle\frac{\partial ^2 g}{\partial \eta\partial\varphi};\dot \eta,\frac{\partial f}{\partial \eta},\rangle=$$
\be\label{pochodne3} =\langle\frac{\partial ^2 g}{\partial \eta^2};\frac{\partial f}{\partial \varphi},\dot \eta\rangle-\langle\frac{\partial ^2 f}{\partial \eta^2};\frac{\partial g}{\partial \varphi},\dot \eta\rangle+
\langle\frac{\partial ^2 f}{\partial\varphi\partial \eta};\frac{\partial g}{\partial \eta},\dot \eta\rangle-
\langle\frac{\partial ^2 g}{\partial\varphi\partial \eta};\frac{\partial f}{\partial \eta},\dot \eta\rangle.
\ee
Since  $\frac{\partial ^2 f}{\partial \eta^2}(\varphi,\eta)\in\Ll(\M p_0,p_0\M_*)$ $\frac{\partial ^2 f}{\partial\varphi\partial \eta}(\varphi,\eta)\in\Ll(p_0\M_*,\M p_0)$ $\frac{\partial f}{\partial \varphi}(\varphi,\eta),\ \frac{\partial g}{\partial \varphi}(\varphi,\eta)\in (p_0\M_*)^*=\M p_0$ and $\frac{\partial f}{\partial \eta}(\varphi,\eta),\ \frac{\partial g}{\partial \eta}(\varphi,\eta)\in p_0\M_*$
we find that $\langle\frac{\partial ^2 g}{\partial \eta^2};\frac{\partial f}{\partial \varphi},\cdot\rangle
-\langle\frac{\partial ^2 f}{\partial \eta^2};\frac{\partial g}{\partial \varphi},\cdot\rangle+
\langle\frac{\partial ^2 f}{\partial\varphi\partial \eta};\frac{\partial g}{\partial \eta},\cdot\rangle-
\langle\frac{\partial ^2 g}{\partial\varphi\partial \eta};\frac{\partial f}{\partial \eta},\cdot\rangle\in p_0\M_*.$
Thus  and from (\ref{pochodne3}) we see that $\frac{\partial\{f,g\}}{\partial\eta}(\varphi,\eta)\in p_0\M_*$. So, we have proved that $\{f,g\}\in \mathcal{P}^\infty(T_*P_0)$.

For proving the Jacobi identity for the bracket (\ref{Pbracket}) we take 
\be\label{Jac} \{\{f,g\},h\}=\langle\frac{\partial h}{\partial\eta},\frac{\partial\{f,g\}}{\partial \varphi}\rangle-
\langle\frac{\partial \{f,g\}}{\partial\eta},\frac{\partial h}{\partial \varphi}\rangle=\ee

$$=\langle\frac{\partial^2g}{\partial \varphi\partial\eta};\frac{\partial h}{\partial\eta},\frac{\partial f}{\partial\varphi}\rangle+
\langle\frac{\partial^2f}{\partial \varphi^2};\frac{\partial h}{\partial\eta},\frac{\partial g}{\partial\eta}\rangle-
\langle\frac{\partial^2f}{\partial \varphi\partial\eta};\frac{\partial h}{\partial\eta},\frac{\partial g}{\partial\varphi}\rangle-
\langle\frac{\partial^2g}{\partial \varphi^2};\frac{\partial h}{\partial\eta},\frac{\partial f}{\partial\eta}\rangle+$$
$$+\langle\frac{\partial^2g}{\partial\eta^2};\frac{\partial h}{\partial\varphi},\frac{\partial f}{\partial\varphi}\rangle+
\langle\frac{\partial^2f}{\partial \eta\partial\varphi};\frac{\partial h}{\partial\varphi},\frac{\partial g}{\partial\eta}\rangle-
\langle\frac{\partial^2f}{\partial\eta^2};\frac{\partial h}{\partial\varphi},\frac{\partial g}{\partial\varphi}\rangle-
\langle\frac{\partial ^2g}{\partial \eta\partial\varphi};\frac{\partial h}{\partial\varphi},\frac{\partial f}{\partial\eta}\rangle.$$
Adding cyclic permutations of (\ref{Jac}) and taking into account (\ref{pochodne1}) and (\ref{pochodne2}) we find that Jacobi identify for $f,g,h\in \mathcal{P}^\infty(T_*P_0)$ is satisfied.
\end{proof}

\begin{rem}\label{rem:42} \begin{enumerate}[(i)]
\item \textsl{The bracket (\ref{Pbracket}) after restriction to $\mathcal{P}^\infty(T_*P_0)$ is  the Poisson bracket defined by the weak symplectic form (\ref{omega}).}

\item \textsl{ If $f\in \mathcal{P}^\infty(T_*P_0)$ then the equality (\ref{omega3}) defines a vector field $\xi_f\in \Gamma^\infty T(T^*P_0)$. But if $f\notin \mathcal{P}^\infty(T_*P_0)$ then $\{f,\cdot\}$ is only a section of the bundle  $T^{**}(T_*P_0)$  which contains $T(T_*P_0)$ as a  quasi Banach subbundle, i.e. the bundle inclusion $T(T_*P_0)\hookrightarrow T^{**}(T_*P_0)$ has closed range but without the Banach split.}

\item \textsl{$f\in \mathcal{P}^\infty(T_*P_0)$ if  and only if $df\in \Gamma^\infty T^\flat(T_*P_0)$, i.e. the Banach subbundle $ T^\flat(T_*P_0)$ is defined by  $\mathcal{P}^\infty(T_*P_0)$.}
\end{enumerate}\end{rem}

Since in a general case $T^\flat(T_*P_0)\subsetneq T^*(T_*P_0)$  the Banach bundle morphism $\#:T^\flat(T_*P_0)\to T(T_*P_0)$ inverse to $\flat:T(T_*P_0)\hookrightarrow T^*(T_*P_0)$ is not defined  on the whole of  $T^*(T_*P_0)$. So, following of \cite{pell},   it will be called a \textbf{sub Poisson anchor}. Note here that the bracket (\ref{Pbracket}) and $\#$ define Banach algebroid structure on $T^\flat(T_*P_0)$.

Using the fibre bundle isomorphisms:
\be\begin{array}{l}
T(T_*P_0)\cong p_0\M_*\times \M p_0\times p_0\M_*\times P_0\\
T^*(T_*P_0)\cong \M p_0\times (\M p_0)^*\times p_0\M_*\times P_0\\
T^\flat(T_*P_0)\cong \M p_0\times p_0\M_*\times p_0\M_*\times P_0\end{array}\ee
we introduce the following notations:
\be\label{coo1} (\dot\varphi,\dot\eta,\varphi,\eta)\in p_0\M_*\times \M p_0\times p_0\M_*\times P_0\ee
\be\label{coo2} (\stackrel{\circ}{\varphi},\stackrel{\circ}{\eta},\varphi,\eta)\in \M p_0\times p_0\M_*\times p_0\M_*\times P_0\ee
for the coordinates of the elements of $T(T_*P_0)$ and $T^\flat(T_*P_0)$, respectively. The sub Poisson anchor $\#_1:T^\flat(T_*P_0)\to T(T_*P_0)$ in the coordinates (\ref{coo1}) and (\ref{coo2}) assumes the form 
\be\label{anch4}\#_1(\stackrel{\circ}{\varphi},\stackrel{\circ}{\eta},\varphi,\eta)=(-\stackrel{\circ}{\eta},\stackrel{\circ}{\varphi},\varphi,\eta)\ee

Let us denote by  $\mathcal{P}^\infty_{G_0}(T_*P_0)\subset \mathcal{P}^\infty(T_*P_0)$  the subalgebra of $G_0$-invariant functions. From $G_0$-invariance of the bracket (\ref{Pbracket}) it follows that $\mathcal{P}^\infty_{G_0}(T_*P_0)$ is a Poisson subalgebra of $\mathcal{P}^\infty(T_*P_0)$. We will identify $f\in \mathcal{P}^\infty_{G_0}(T_*P_0)$ with a function on the quotient space $T_*P_0/G_0$ which can be considered as a Banach vector bundle $(\M p_0)_*\times _{G_0} P_0$ associated with the $G_0$-principal bundle $P_0\to P_0/G_0$. Thus the quotient projection  $Q_0:T_*P_0\to T_*P_0/G_0$ is a submersion, see 6.5.1 in \cite{bou}. Therefore one can consider $\mathcal{P}^\infty_{G_0}(T_*P_0)$ as a subalgebra $\mathcal{P}^\infty(T_*P_0/G_0)$ of the algebra $C^\infty(T_*P_0/G_0)$. Hence  the Poisson bracket $\{F,G\}_{G_0}$ of $F,G\in \mathcal{P}^\infty(T_*P_0/G_0)$ one defines as follows
\be \{F,G\}_{G_0}:=\left\{F\circ Q_0,G\circ Q_0\right\}.\ee

In the case of $T^*P_0/G_0$  we can define $T^\flat(T_*P_0/G_0)$ as the bundle of germs of  1-forms $df$, where $f\in \mathcal{P}^\infty(T_*P_0/G_0)$. We note that one has the following isomorphisms:
\be\label{tang}  T_{[(\varphi,\eta)]}(T_*P_0/G_0)\cong (\M p_0)_*\times (1-p_0)\M p_0\times \{[(\varphi,\eta)]\}\ee
\be\label{gw}  T^*_{[(\varphi,\eta)]}(T_*P_0/G_0)\cong \M p_0\times p_0 \M^* (1-p_0)\times \{[(\varphi,\eta)]\}\ee
\be\label{bem}  T^\flat_{[(\varphi,\eta)]}(T_*P_0/G_0)\cong \M p_0\times p_0\M_*(1-p_0) \times \{[(\varphi,\eta)]\}\ee
where $[(\varphi,\eta)]\in T_*P_0/G_0$. From (\ref{gw}) and (\ref{bem}) we see that $ T^\flat(T_*P_0/G_0)$ is a proper quasi Banach subbundle of $ T^*(T_*P_0/G_0)$.

Similarly as for  $TP_0$ and $TP_0/G_0$ let us define the atlases on $T_*P_0$ and $T_*P_0/G_0$ consistent with their vector  bundle structures. We begin from the $(\varphi,\eta)$ coordinates on $T_*P_0$ predual to the coordinates $(v,\eta)$ on $TP_0$. Using (\ref{v}) we find 
\be \langle\varphi,v\rangle= \langle\varphi,[a_p+(p+y_p)b_p ]z_{pp_0}\rangle=\ee
$$=  \langle z_{pp_0}\varphi(1-p),a_p\rangle+\langle z_{pp_0}\varphi(p+y_p),b_p\rangle=\langle \alpha_p,a_p\rangle+\langle\beta_p,b_p\rangle,$$
where $\alpha_p\in p\M_*(1-p)$ and $  \beta_p\in p\M_* p$ are defined by
\be \alpha_p:=z_{pp_0}\varphi(1-p)\ee
\be \beta_p:= z_{pp_0}\varphi(p+y_p).\ee
Thus, similarly to (\ref{transa}-\ref{transb}) we obtain the equalities
\be\label{transalpha} \alpha_p=(p\eta)\varphi(1-p)\qquad\qquad\ee
\be\label{transbbeta}  \beta_p=p\eta\varphi\eta(p\eta)^{-1}\qquad\qquad\qquad\ee
 which togather with  (\ref{1yp}) define the chart 
\be\label{chartTdual} T_*(\pi_0^{-1}(\Pi_p))\ni(\varphi,\eta)\mapsto T_*\psi_p(\varphi,\eta)=(\alpha_p,\beta_p,y_p,z_{pp_0})\ee
on $T_*P_0$, where $p\in \Ll_{p_0}(\M)$.
The dependence inverse to (\ref{transalpha}-\ref{transbbeta}) is given by
\be \varphi=z_{pp_0}^{-1}(\alpha_p+\beta_p-\alpha_p y_p)\ee
\be \eta=(p+y_p)z_{pp_0}.\ee
The transition map from the coordinates $(\alpha_p,\beta_p, y_p,z_{pp_0})$ to the coordinates $(\alpha_{p'},\beta_{p'}, y_{p'},z_{p'p_0})$ is the following
\be\label{5talp} \alpha_{p'}=(a+cy_p)(\alpha_p+\beta_p-\alpha_p y_p)(1-p')\ee
\be\label{5tbeta} \beta_{p'}=(a+cy_p)\beta_p(a+cy_p)^{-1}\ee
\be\label{5ty}  y_{p'}=(b+dy_p)(a+cy_p)^{-1}\qquad\ee
\be z_{p'p_0}=(a+cy_p)z_{pp_0}.\qquad\ee

Quotienting the chart (\ref{chartTdual}) by $G_0$ we obtain  on $T_*P_0/G_0$ the chart 
\be T_*(\pi_0^{-1}(\Pi_p))/_{G_0}\ni\langle \varphi,\eta\rangle\mapsto [T_*\psi_p](\langle \varphi,\eta\rangle)=(\alpha_p,\beta_p,y_p).\ee
The transition map $[T_*\psi_{p'}]\circ [T_*\psi_p]^{-1}:(\alpha_p,\beta_p,y_p)\mapsto (\alpha_{p'},\beta_{p'},y_{p'})$ between these charts is given by (\ref{5talp}-\ref{5ty}).

\bigskip

Finally let us note that the momentum map $J_1:T_*P_0\to p_0\M_* p_0$ in the coordinates $(\alpha_p,\beta_p, y_p,z_{pp_0})$ assumes the form
\be\label{Jcoor} J_1(\alpha_p,\beta_p, y_p,z_{pp_0})=z_{pp_0}^{-1}\beta_p z_{pp_0}\ee

\bigskip

Since $p_0\M_* p_0$ is the predual Banach space of the $W^*$-algebra $p_0\M p_0$, see (\ref{Mpredual}), the structure of Banach Lie-Poisson space is defined on it in a canonical way. Namely, according  to \cite{OR},  the bracket
\be\label{LPbr} \{F,G\}_{LP}(\beta):=\left\langle\beta,\left[\frac{\partial F}{\partial \beta}(\beta),\frac{\partial G}{\partial \beta}(\beta)\right]\right\rangle\ee
is a Lie-Poisson bracket of $F,G\in C^\infty(p_0\M_* p_0)$. The followig theorem is valid.

\begin{thm}\label{prop:43}
 \begin{enumerate}[(i)]
\item One has the surjective Poisson submersions:

  \unitlength=5mm
 \be\label{diagsymplpair}
 \begin{picture}(11,4.6)
    \put(4.5,4){\makebox(0,0){$T_*P_0$}}
    \put(0,-1){\makebox(0,0){$T_*P_0/G_0$}}
    \put(9,-1){\makebox(0,0){$ p_0\M _* p_0$}}
    \put(4,3.7){\vector(-1,-1){4}}
     \put(5,3.7){\vector(1,-1){4}}
       \put(1,2){\makebox(0,0){$Q_0$}}
    \put(8,2){\makebox(0,0){$ J_1$}}
     \end{picture}
\ee
    \bigskip\newline
of the weak symplectic manifold $(T_*P_0,\omega)$ on the sub Poisson manifold $(T_*P_0/G_0,\{\cdot,\cdot\}_{G_0})$ and on the Banach Lie-Poisson space  $(p_0\M p_0,\{\cdot,\cdot\}_{LP})$.
\item  The Poisson subalgebras $J_0^*(C^\infty(p_0\M_* p_0)$ and $(Q_0)^*(\mathcal{P}^\infty(T_*P_0/G_0))=\mathcal{P}^\infty_{G_0}(T_*P_0) $ of the Poisson algebra $\mathcal{P}^\infty(T_*P_0)$ are polar one to another with respect to the weak symplectic form $\omega$.
\end{enumerate}\end{thm}
\begin{proof} (i) In order to see that $Q_0:T_*P_0\to T_*P_0/G_0$ is a surjective submersion we note that $T_*P_0/G_0\to P_0/G_0$ is a Banach vector bundle associate with the principal bundle $\pi_0:P_0\to P_0/G_0$.

 Substituting  $\varphi=\beta$ and $\eta=p_0$ into (\ref{momentum}) we find that $J_1(\beta,p_0)=\beta$. This shows the surjectivity of  $J_1$. 

For $\dot \varphi\in p_0\M_*$ and $\dot \eta\in \M p_0$ one has 
\be\label{TJ1} T J_1(\varphi,\eta)(\dot  \varphi,\dot \eta)=\frac{\partial J_1}{\partial \varphi}(\varphi,\eta)\dot  \varphi+\frac{\partial J_1}{\partial \eta}(\varphi,\eta)\dot  \eta=\dot  \varphi\eta+\varphi\dot  \eta.\ee
Substituting  $\eta=p_0$ and $\dot \eta=0$ into (\ref{TJ1}) we obtain that any $x\in p_0\M _* p_0$ can be written as $x=TJ_1(\varphi,p_0)(\dot \varphi,0)=\dot  \varphi p_0$. Thus, $TJ_1(\varphi,\eta):T_{(\varphi,\eta)}(T_*P_0)\to T_{J_1(\varphi,\eta)}p_0\M_* p_0$ is a surjection.

Now let us  show that $ker TJ_1(\varphi,\eta)$ has a Banach complement. For this reason we will use the coordinate expression (\ref{Jcoor}) for the momentum map $J_1$. Differentiating (\ref{Jcoor}) we obtain 
\be\label{TJ} TJ_1(\alpha_p,\beta_p,y_p, z_{pp_0})(\dot \alpha_p,\dot \beta_p,\dot  y_p, \dot  z_{pp_0})=z_{pp_0}^{-1}(\dot \beta_p+\beta_p\dot  z_{pp_0} z_{pp_0}^{-1}-\dot  z_{pp_0} z_{pp_0}^{-1}\beta_p)z_{pp_0},\ee
where by
\be\label{4coor}(\dot \alpha_p,\dot \beta_p,\dot  y_p, \dot  z_{pp_0})\in p\M_*(1-p)\oplus p\M_* p\oplus (1-p)\M p\oplus p\M p_0\ee
we denote coordinates of the tangent vectors at $(\alpha_p, \beta_p,  y_p,   z_{pp_0})$.
 From (\ref{TJ}) we see that $(\dot \alpha_p,\dot \beta_p,\dot  y_p, \dot  z_{pp_0})\in ker TJ_1(\alpha_p,\beta_p,y_p, z_{pp_0})$ iff
\be\label{4Delta} \dot \beta_p=\dot  z_{pp_0} z_{pp_0}^{-1}\beta_p-\beta_p\dot  z_{pp_0} z_{pp_0}^{-1}.\ee
It follows from (\ref{4coor}) and (\ref{4Delta}) that $ker TJ_1(\alpha_p,\beta_p,y_p, z_{pp_0})$ is complemented by the Banach subspace $\{0\}\oplus p\M_* p\oplus\{0\}\oplus \{0\}$. Thus we conclude that the momentum map $J_1:T_*P_0\to p_0\M_* p_0$ is a surjective submersion.

For $F,G\in C^\infty(p_0\M _* p_0) $ and $\beta=J_1(\varphi,\eta)$ we have
$$\{ F\circ  J_1, G\circ J_1\}(\varphi,\eta)=\left\langle \frac{\partial G}{\partial \beta}\frac{\partial J_1}{\partial \eta}(\varphi,\eta),\frac{\partial F}{\partial \beta}\frac{\partial J_1}{\partial \varphi}(\varphi,\eta)\right\rangle-
\left\langle \frac{\partial F}{\partial \beta}\frac{\partial J_1}{\partial \eta}(\varphi,\eta),\frac{\partial G}{\partial \beta}\frac{\partial J_1}{\partial \varphi}(\varphi,\eta)\right\rangle=$$
\be\label{Poiss}=\left\langle \frac{\partial F}{\partial \beta}(\beta)(\cdot)\eta,\frac{\partial G}{\partial \beta}(\beta)\varphi\right\rangle-
\left\langle \frac{\partial G}{\partial \beta}(\beta)(\cdot)\eta,\frac{\partial F}{\partial \beta}(\beta)\varphi\right\rangle=\ee
$$=\left\langle \left[ \frac{\partial F}{\partial \beta}(\beta),  \frac{\partial G}{\partial \beta}(\beta)\right],\varphi\eta\right\rangle=\left(\{F,G\}_{LP}\circ J_1\right)(\varphi,\eta).$$
We note here that $\frac{\partial F\circ J_1}{\partial \eta}(\varphi,\eta)=\frac{\partial F}{\partial \beta}(\beta)\varphi$ and $\ \frac{\partial G\circ J_1}{\partial \eta}(\varphi,\eta)=\frac{\partial G}{\partial \beta}(\beta)\varphi$ belong to $ p_0\M_*$. From (\ref{Poiss}) we conclude that $J_1$ is a Poisson map.

(ii) The polarity of Poisson subalgebras $J_1^*(C^\infty(p_0\M_* p_0)$ and $Q_0^*(\mathcal{P}^\infty(T_*P_0/G_0))$ follows from the fact that one has $\xi^x(f)=0$  for the vector field $\xi^x $  defined in (\ref{pole2}) and $f\in C^\infty_{G_0}(T_*P_0/G_0)$.
\end{proof}

\begin{example} As an example let us consider the case when $\M$ is a finite $W^*$-algebra and the projection $p_0$ is the unit element of $\M$. Then $P_0$ is equal to the group $G(\M)$ of the invertible elements of $\M$. In this case $Q_0$ and $J_1$ are the left $J_L:T_*G(\M)\to \M_*$ and right $J_R:T_*G(\M)\to \M_*$ momentum maps, respectively, of the weak symplectic manifold $T_*G_0$. From Theorem \ref{prop:43} it follows that 
  \unitlength=5mm
 \be\label{ex:diagsymplpair}
 \begin{picture}(11,4.6)
    \put(4.5,4){\makebox(0,0){$T_*G(\M)$}}
    \put(4.5,0){\makebox(0,0){$\M_*$}}
    \put(4,3.7){\vector(0,-1){3}}
     \put(5,3.7){\vector(0,-1){3}}
       \put(3.5,2){\makebox(0,0){$J_R$}}
    \put(5.5,2){\makebox(0,0){$ J_L$}}
     \end{picture}
\ee
is the precotangent weak symplectic groupoid of the Banach Lie Poisson space $(\M_*,\{\cdot,\cdot\}_{LP})$ with Lie-Poisson bracket $\{\cdot,\cdot\}_{LP}$ defined in (\ref{LPbr}).

\end{example}

In order to define the sub Poisson structure on $p_0\M_* p_0\times _{Ad^*_{G_0}}P_0$ let us firstly introduce such kind of structure on $p_0\M_* p_0\times P_0$. For this reason we take the subalgebra  of smooth functions 
\be\label{PG0} \mathcal{P}^\infty_{G_0}(p_0\M_* p_0\times P_0):= \left\{F\in C^\infty(p_0\M_* p_0\times P_0):\ \frac{\partial F}{\partial \eta}(\beta,\eta)\in p_0\M_*\ {\rm and}\ F(Ad^*_g\beta,\eta g)=F(\beta,\eta)\right\}
\ee
and the Poisson bracket of $F,G\in  \mathcal{P}^\infty_{G_0}(p_0\M_* p_0\times P_0)$  we define by
\be\label{sPbra}
\{F,G\}_{sP}(\beta,\eta):=\left\langle\beta,\left[\frac{\partial F}{\partial \beta}(\beta,\eta),\frac{\partial G}{\partial \beta}(\beta,\eta)\right]\right\rangle.\ee
One can easly check that $\{F,G\}_{sP}\in \mathcal{P}^\infty_{G_0}(p_0\M_* p_0\times P_0)$. Considering $\mathcal{P}^\infty_{G_0}(p_0\M_* p_0\times P_0)$ as the subalgebra $\mathcal{P}^\infty(p_0\M_* p_0\times_{Ad_{G_0}} P_0)$ of $C^\infty(p_0\M_* p_0\times_{Ad_{G_0}} P_0)$  and taking into account that the bracket (\ref{sPbra}) is $G_0$-invariant we find that it defines a sub Poisson structure on $p_0\M_* p_0\times_{Ad_{G_0}} P_0$. Note here that $\mathcal{P}^\infty(p_0\M_* p_0\times_{Ad_{G_0}} P_0)\subsetneq C^\infty(p_0\M_* p_0\times_{Ad_{G_0}} P_0)$ in general.

Let us mention that the bracket (\ref{sPbra}) is also well define if $F,G\in C^\infty_{G_0}(p_0\M_* p_0\times P_0)$. However, we have assumed in (\ref{PG0}) the condition $\frac{\partial F}{\partial \eta}(\beta,\eta)\in p_0\M_*$for the consistency with the sub Poisson structure on $T_*P_0$ defined by (\ref{Pbracket}) and (\ref{calP}).

The next proposition describe the sub Poisson structure of the predual Atiyah sequence (\ref{Atiyahrestridual})
\begin{thm}\label{prop:44} The predual Atiyah sequence (\ref{Atiyahrestridual}) is a short exact sequence of the fibre-wise linear sub Poisson complex Banach vector  bundles, i.e.
\begin{enumerate}[(i)]
\item The Banach vector bundle map $\iota_*:T_*P_0/G_0\to p_0\M_* p_0\times_{Ad^*_{G_0}} P_0$ is a sub  Poisson submersion.
\item One has $ker\ \iota_*=J_1^{-1}(0)/G_0$, where $J_1^{-1}(0)/G_0$ is the weak symplectic leaf in $T_*P_0/G_0$ obtained by the Marsden-Weinstein symplectic reduction procedure, \cite{MW}. The predual anchor map $a_*:T_*(P_0/G_0)\hookrightarrow T_*P_0/G_0$ is an immersion which defines the isomorphism $T_*(P_0/G_0)\cong J_1^{-1}(0)/G_0$ of weak symplectic manifolds, if the precotangent bundle $T_*(P_0/G_0)$ is endowed with the canonical weak symplectic structure. 
\end{enumerate}
\end{thm}
\begin{proof} 
\begin{enumerate}[(i)] 
\item In order to describe $\iota_*:T_*P_0/G_0\to p_0\M_* p_0 \times_{Ad^*_{G_0}}P_0$, see (\ref{Atiyahrestridual}), in detail, we consider  the map 
\be I_*:T_*P_0\cong p_0\M_*\times P_0\ \ni\ (\varphi,\eta)\mapsto (J_1(\varphi,\eta), \eta)\ \in\ p_0\M_* p_0\times  P_0.\ee

Note that $I_*=J_1\times pr_2$, where $pr_2:p_0\M_*\times P_0$ is the projection on the second component of the Cartesian product. Since $J_1$ and $pr_2$ are surjective submersion we conclude that $I_*$ has the same property, see 5.9.3 in \cite{bou}.

The equivariance property $I_*(g^{-1}\varphi,\eta g)=(Ad^*_{g^{-1}}(J_1(\varphi,\eta),\eta g)$, $g\in G_0$ allows us to define $\iota_*$ by 
\be \label{iotapre} \iota_*([\varphi,\eta)]):=[(J_1(\varphi,\eta),\eta]=[(\varphi \eta,\eta)],\ee
where $[(\varphi,\eta)]$  and $[(\beta,\eta)]$ are the $G_0$-orbits of $(\varphi,\eta)\in  p_0\M_*\times P_0$ and $(\beta,\eta)\in p_0\M_* p_0\times P_0$, respectively.

In order to show that the bundle epimorphism $\iota_*:T_*P_0/P_0\to p_0\M_* p_0\times _{Ad^*_{G_0}} P_0$ is a submersion we note that 
$$\iota_*\circ Q_0=\pi_{Ad^* _{G_0}}\circ  I_*,$$
where $\pi_{Ad^* _{G_0}}:p_0\M p_0\times P_0\to p_0\M_* p_0\times _{Ad^*_{G_0}} P_0$ is the quotient map. Since the maps $Q_0$, $\pi_{Ad^* _{G_0}}$ and $I_*$ are surjective submersions we conclude that $\iota_*$ is a submersion too, see 5.9.2 in \cite{bou}.
  
	For $F,G\in \mathcal{P}^\infty_{G_0}(p_0\M_* p_0\times P_0)\cong \mathcal{P}^\infty(p_0\M_* p_0\times_{Ad^* _{G_0}}P_0) $ we have
\be\label{proof1}\{ F\circ I_*,G\circ I_*\}(\varphi,\eta)=\ee
$$=\left\langle\frac{\partial (G\circ I_*)}{\partial\eta}(\varphi,\eta),\frac{\partial (F\circ I_*)}{\partial\varphi}(\varphi,\eta)\right\rangle-
\left\langle\frac{\partial (F\circ I_*)}{\partial\eta}(\varphi,\eta),\frac{\partial (G\circ I_*)}{\partial\varphi}(\varphi,\eta)\right\rangle=$$
$$=\left\langle\frac{\partial G}{\partial\beta}(J_1(\varphi,\eta),\eta)\frac{\partial J_1}{\partial\eta}(\varphi,\eta)+\frac{\partial G}{\partial\eta}(J_1(\varphi,\eta),\eta),\frac{\partial F}{\partial\beta}(J_1(\varphi,\eta),\eta)(\cdot)\eta\right\rangle+$$
$$-\left\langle\frac{\partial F}{\partial\beta}(J_1(\varphi,\eta),\eta)\frac{\partial J_1}{\partial\eta}(\varphi,\eta)+\frac{\partial F}{\partial\eta}(J_1(\varphi,\eta),\eta),\frac{\partial G}{\partial\beta}(J_1(\varphi,\eta),\eta)(\cdot)\eta\right\rangle=$$
$$=\left\langle\frac{\partial F}{\partial\beta}(J_1(\varphi,\eta),\eta), \frac{\partial G}{\partial\beta}(J_1(\varphi,\eta),\eta) \varphi\eta+\frac{\partial G}{\partial\eta}(J_1(\varphi,\eta),\eta)\eta\right\rangle+$$
$$-\left\langle\frac{\partial G}{\partial\beta}(J_1(\varphi,\eta),\eta), \frac{\partial F}{\partial\beta}(J_1(\varphi,\eta),\eta) \varphi\eta+\frac{\partial F}{\partial\eta}(J_1(\varphi,\eta),\eta)\eta\right\rangle=$$
$$=\left\langle J_1(\varphi,\eta),[\frac{\partial F}{\partial\beta}(J_1(\varphi,\eta),\eta), \frac{\partial G}{\partial\beta}(J_1(\varphi,\eta),\eta)]\right\rangle+$$
$$+\left\langle\frac{\partial G}{\partial\eta}(J_1(\varphi,\eta),\eta),\eta\frac{\partial F}{\partial\beta}(J_1(\varphi,\eta),\eta) \right\rangle-
\left\langle\frac{\partial F}{\partial\eta}(J_1(\varphi,\eta),\eta),\eta\frac{\partial G}{\partial\beta}(J_1(\varphi,\eta),\eta) \right\rangle$$
Since the functions $F$ and $G$ are $G_0$-invariant then
\be\label{proof2}\left\langle\frac{\partial G}{\partial\eta}(\beta,\eta),\eta x\right\rangle=
-\left\langle\frac{\partial G}{\partial\beta}(\beta,\eta),ad_x^*\beta \right\rangle=
-\left\langle \left[x,\frac{\partial G}{\partial\beta}(\beta,\eta)\right],\beta \right\rangle\ee
for any $x\in p_0\M_* p_0$.

Thus,  taking $x=\frac{\partial F}{\partial\beta}(\beta,\eta)$ and $x=\frac{\partial G}{\partial\beta}(\beta,\eta)$, respectively, we obtain 
$$\{ F\circ I_*,G\circ I_*\}(\varphi,\eta)=\left\langle J_1(\varphi,\eta),\left[\frac{\partial F}{\partial\beta}
(J_1(\varphi,\eta),\eta),\frac{\partial G}{\partial\beta}(J_1(\varphi,\eta),\eta)\right]\right\rangle=$$
\be\label{Pbra4}=(\{F,G\}_{LP}\circ J_1)(\varphi,\eta)=(\{F,G\}_{sP}\circ I_*)(\varphi,\eta).\ee
Since both Poisson brackets in (\ref{Pbra4}) and functions $F,\ G$ are $G_0$-invariant one can take the quotient of (\ref{Pbra4}) by $G_0$. Hence we obtain 
\be\label{Pbra5}\{ F\circ \iota_*,G\circ \iota_*\}=(\{F,G\}_{sP}\circ \iota_*),\ee
where  $\{F,G\}_{sP}$  is the Poisson bracket on $\mathcal{P}^\infty(p_0\M_* p_0\times_{Ad^* _{G_0}}P_0)$ defined by the Poisson bracket (\ref{sPbra}).
\item Since the momentum map $J_1:T_*P_0\to p_0\M_* p_0$ is a submersion the fibre $J_1^{-1}(0)$ of $0\in p_0\M_* p_0$ is a submanifold of $T_*P_0$. So, $J_1^{-1}(0)/G_0$ is a submanifold of $T_*P_0/G_0$.

The equality $ker\ \iota_*=J_1^{-1}(0)/G_0$ follows directly from (\ref{iotapre}). In order to show that the Marsden-Weinstein symplectic reduction applied to $J_1^{-1}(0)$ leads to the weak symplectic manifold structure on $J_1^{-1}(0)/G_0$ we define the local trivialization $a_{*p}:\nu_*^{-1}(\Pi_p)\to (\pi_*\circ\pi_0)^{-1}(\Pi_p)$ of the predual anchor map $a_*$, where $\nu_*:T_*(P_0/G_0)\to P_0/G_0$ is the bundle projection of the precotangent bundle $T_*(P_0/G_0)$.

For any $p\in \Ll_{p_0}(\M)\cong P_0/G_0$ we choose $\eta_{pp_0}\in P_0$  such that $\Tt(\eta_{pp_0})=p$ and define the principal bundle section $\sigma_{pp_0}:\Pi_p\to \pi_0^{-1}(\Pi_p)\subset P_0$ by
\be\label{prisec}\sigma_{pp_0}(q):=\sigma_{p}(q)\eta_{pp_0},\quad q\in \Pi_p.\ee
Using $\sigma_{pp_0}$ we define $a_{*p}$ as follows
\be\label{prisec2} a_{*p}(\rho):=(T\pi_0(\sigma_{pp_0}(q))^*(\rho)=(\rho\circ T\pi_0)(\sigma_{pp_0}(q)),\ee
where $\rho\in\nu^{-1}_*(q)$.

For $\xi_\rho\in T_\rho(T_*(P_0/G_0))$  and the pullback $(a_{*p})^*\gamma$ of the canonical 1-form (\ref{form}) we have 
\be\label{form3} \langle((a_{*p})^*\gamma)_\rho,\xi_\rho\rangle= \langle\gamma_{a_{*p}(\rho)},Ta_{*p}(\rho)\xi_\rho\rangle=\ee
$$=\langle{a_{*p}(\rho)},T\pi_*(a_{*p}(\rho))Ta_{*p}(\rho)\xi_\rho\rangle=\langle{a_{*p}(\rho)},T(\pi_*\circ a_{*p})(\rho)\xi_\rho\rangle=
\langle{\rho},T\pi_0(\sigma_{pp_0(q)})\circ T(\pi_*\circ a_{*p})(\rho)\xi_\rho\rangle=$$
$$=\langle{\rho},T(\pi_0\circ \pi_*\circ a_{*p})(\rho)\xi_\rho\rangle=\langle{\rho},T\nu_*(\rho)\xi_\rho\rangle=:\langle\tilde\gamma_\rho,\xi_\rho\rangle$$
The last equality in (\ref{form3}) follows from $\pi_0\circ \pi_*\circ {a_{*p}}=\nu_*$. From (\ref{form3}) we conclude that the pullback  $(a_{*p})^*\gamma$ does not depend on the trivialization and is equal to the weak canonical form $\tilde\gamma$ of $T_*(P_0/G_0)$.
It also follows from (\ref{form3}) that $a_*:T_*(P_0/G_0)\stackrel{\sim}{\to}J_1^{-1}(0)/G_0$ is an isomorphism  of weak symplectic manifolds.

\end{enumerate}
\end{proof}

\begin{cor} All statement of Theorem \ref{prop:44} are valid for (\ref{dualbundles}). Thus, since of Proposition \ref{prop:11} they are also valid for (\ref{Atiyahalgebrdual}).

\end{cor}
\begin{proof} It follows from Proposition \ref{prop41}. 
\end{proof}

\bigskip

 The bracket (\ref{Pbracket}) written in the local coordinates (\ref{chartTdual}) assumes the following form 
\be\label{nawias3} \{f,g\}=\left\langle \frac{\partial g}{\partial y_p},\frac{\partial f}{\partial \alpha_p}\right\rangle-\left\langle \frac{\partial f}{\partial y_p},\frac{\partial g}{\partial \alpha_p}\right\rangle+\ee
$$+\left\langle\beta_p,[ \frac{\partial g}{\partial \beta_p},\frac{\partial f}{\partial \beta_p}]\right\rangle+
\left\langle z_{pp_0}\frac{\partial g}{\partial z_{pp_0}},\frac{\partial f}{\partial \beta_p}\right\rangle-\left\langle z_{pp_0}\frac{\partial f}{\partial z_{pp_0}},\frac{\partial g}{\partial \beta_p}\right\rangle.$$
We note here that for $f\in \mathcal{P}^\infty(T_*P_0)$, i.e. $\frac{\partial f}{\partial \eta}(\eta)\in p_0\M_*$, the partial derivative $\frac{\partial f}{\partial  y_p}(\alpha_p,\beta_p, y_p,z_{pp_0})$ belongs to $p\M_*(1-p)$. The sub Poisson anchor $\#_1:T^\flat(T_*P_0)\to T(T_*P_0)$ in these coordinates is written as
\be\label{anchbemol}\#_1(\stackrel{\circ}{\alpha_p},\stackrel{\circ}{\beta_p},\stackrel{\circ}{y_p},\stackrel{\circ}{z}_{pp_0},\alpha_p,\beta_p,y_p,z_{pp_0})=(-\stackrel{\circ}{y_p},-ad^*_{\stackrel{\circ}{\beta_p}}(\beta_p)-{z}_{pp_0}\stackrel{\circ}{z}_{pp_0},\stackrel{\circ}{\alpha_p},\stackrel{\circ}{\beta_p}{z}_{pp_0},\alpha_p,\beta_p,y_p,{z}_{pp_0})
\ee

If $f,g\in \mathcal{P}^\infty_{G_0}(T_*P_0)\cong \mathcal{P}^\infty(T_*P_0/G_0)$ then $z_{pp_0}\frac{\partial g}{\partial z_{pp_0}}=0$ and $z_{pp_0}\frac{\partial f}{\partial z_{pp_0}}=0$. Hence the two last terms in (\ref{nawias3}) disappear and one obtains the local coordinate formula 
\be\label{nawias3a} \{f,g\}_{G_0}=\left\langle \frac{\partial g}{\partial y_p},\frac{\partial f}{\partial \alpha_p}\right\rangle-\left\langle \frac{\partial f}{\partial y_p},\frac{\partial g}{\partial \alpha_p}\right\rangle+\left\langle\beta_p,[ \frac{\partial g}{\partial \beta_p},\frac{\partial f}{\partial \beta_p}]\right\rangle\ee
for the Poisson bracket $\{\cdot,\cdot\}_{G_0}$ on  $T_*P_0/G_0$. The sub Poisson anchor $\#_{1G_0}:T^\flat(T^*P_0/G_0)\to T(T^*P_0/G_0)$ according to (\ref{nawias3a}) is as follows
\be \#_{1G_0}(\stackrel{\circ}{\alpha_p},\stackrel{\circ}{\beta_p},\stackrel{\circ}{y_p},\alpha_p,\beta_p,y_p)=(-\stackrel{\circ}{y_p},-ad^*_{\stackrel{\circ}{\beta_p}}(\beta_p),\stackrel{\circ}{\alpha_p},\alpha_p,\beta_p,y_p),\ee
where $(\stackrel{\circ}{\alpha_p},\stackrel{\circ}{\beta_p},\stackrel{\circ}{y_p})\in (1-p)\M p\times p\M p\times p\M_*(1-p)$ are the coordinates along fibres of $T^\flat(T^*P_0/G_0)\to T^*P_0/G_0$.

 Therefore, the coordinates $(y_p,\alpha_p,\beta_p)$ are the canonical coordinates in sense of Weinstein local splitting theorem presented in  \cite{wei2}, see also formulas (1.41), (1.42) and (1.43) in Chapter 1 of \cite{zung}. Hence, one can consider the predual Atiyah sequence (\ref{Atiyahrestridual}) as a global version of the local splitting theorem for the sub Poisson of the complex Banach manifold $T_*P_0/G_0$.

Using the $G_0$-invariance of $(\alpha_p,\beta_p, y_p)$ one easily concludes that $(\alpha_p,y_p)$ are local coordinates on $T_*(P_0/G_0)$ and $(\beta_p,y_p)$ on $p_0\M_* p_0\times_{Ad^*_{G_0}}P_0$. The canonical Poisson bracket $\{\cdot,\cdot\}$ on  $T_*(P_0/G_0)$ and  the Poisson  bracket $\{\cdot,\cdot\}_{sP}$ on $p_0\M_* p_0\times_{Ad^*_{G_0}}P_0$ written in the above coordinates  are given by appropriate parts of (\ref{nawias3a}) if we substitute to it the functions $f$ and $g$ dependent only on $(\alpha_p,y_p)$ and $(\beta_p,y_p)$, respectively. 

 The coordinate formula for the sub Poisson anchor $\tilde{\#}_1:T^\flat(T_*(P_0/G_0))\to T(T_*(P_0/G_0))$ of the weak symplectic manifold $T^*(T_*(P_0/G_0))$ is as follows
		\be\label{sPois1} \tilde{\#}_1(\stackrel{\circ}{\alpha_{p}},\stackrel{\circ}{y_{p}},\alpha_{p},y_p)=(-\stackrel{\circ}{y_{p}},\stackrel{\circ}{\alpha_{p}},\alpha_{p},y_p).\ee

In these coordinates the predual anchor map $a_*$ is given by $(\alpha_p,y_p)\mapsto (\alpha_p,0,y_p)$ and the map $\iota_*$ is given by $(\alpha_p,\beta_p, y_p)\mapsto(\beta_p,y_p)$. These observations taken together show again that the predual Atiyah sequence (\ref{Atiyahrestridual}) is a short exact sequence of sub Poisson Banach  bundles. 

As we see from Theorem \ref{prop:43} the weak symplectic realization (\ref{diagsymplpair}) of sub Poisson manifold $T_*P_0/G_0$ and the Banach-Lie Poisson space $p_0\M_* p_0$ gives the correspondence between their symplectic leaves. Namely,  a coadjoint orbit $\mathcal{O}\subset p_0\M_*p_0$, which is a weak symplectic leave in the Banach-Lie Poisson space $p_0\M_*p_0$, corresponds to the syplectic leave 
\be\pi_{*G_0}(J_1^{-1}(0))=\iota_*^{-1}(\mathcal{O}\times _{Ad^*_{G_0}}P_0)\ee
in $T_*P_0/G_0$.

Let us mention that the symplectic leaves of Banach-Lie Poisson spaces  were investigated in \cite{bel, bona, OR}. For example, in the case  of $L^1(\H)$, which is predual of $L^\infty(\H)$, the coadjoin orbit $\mathcal{O}_\rho$ of the finite rank trace class operator $\rho\in L^1(\H)$ is a submanifold of $L^1(\H)$ and its symplectic structure is given by the strong symplectic form. However, the investigation of the symplectic leaves of $\M_*$, and thus $T_*P_0/G_0$, needs the advanced functional analytical methods and is not easy even in a concrete case.


\section{Predual short exact sequence of $\mathcal{VB}$-groupoids  with $\G_{p_0}(\M)\tto \Ll_{p_0}(\M)$ as the side groupoid}\label{sec:Poisson}
Through this section we will study the sub Poisson structure of some Banach-Lie $\mathcal{VB}$-groupoids which have the gauge groupoid $\frac{P_0\times P_0}{G_0}\tto P_0/G_0$ as the side groupoid, see diagram (\ref{C}). An introduction to the theory of 
$\mathcal{VB}$-groupoids can be found in \cite{mac}. Indispensable ingredients of this theory are also presented in \cite{MOS} and in the appendix of this paper.

Applying the tangent functor to a Banach Lie groupoid $G\tto M$ one obtains its tangent $\mathcal{VB}$-groupoid $TG\tto TM$. In particular case, one obtains  the tangent group $TG\tto \{0\}$  of a Banach Lie group $G\tto \{e\}$. The tangent groupoid $TG\tto TM$ as well as its dual $T^*G\tto A^*G$, where $AG$ is the algebroid of the groupoid $G$, yield important examples of $\mathcal{VB}$-groupoids.

We modify the definition of the finite dimensional Poisson groupoid presented in Chapter 11.4 of \cite{mac}  to the sub Poisson Banach case considered here. According to this modification the Banach-Lie groupoid $G\tto M$ is a \textbf{sub Poisson groupoid} with a sub Poisson anchor $\#:T^\flat G\to TG$  if there exists a Banach subgroupoid $T^\flat G\tto A^\flat G$ of the  Banach groupoid $T^*G\tto A^*G$  dual to $TG\tto TM$ and Banach bundles morphism $a_*:A^\flat G\to TM$ such that
			\be\label{prop:Pmorph}\begin{picture}(11,4.6)
    \put(1,4){\makebox(0,0){$T^\flat G$}}
    \put(7.7,4){\makebox(0,0){$TG$}}
    \put(1,-1){\makebox(0,0){$A^\flat G$}}
    \put(8,-1){\makebox(0,0){$TM$}}
    \put(1.2,3){\vector(0,-1){3}}
    \put(0.7,3){\vector(0,-1){3}}
    \put(8.2,3){\vector(0,-1){3}}
    \put(7.7,3){\vector(0,-1){3}}
    \put(3,4){\vector(1,0){2.8}}
    \put(2.4,-1){\vector(1,0){3.7}}
    \put(0.1,1.4){\makebox(0,0){$ \  $}}
    \put(2.2,1.4){\makebox(0,0){$ \ $}}
    \put(8.8,1.4){\makebox(0,0){$\   $}}
    \put(7.1,1.4){\makebox(0,0){$\   $}}
    \put(4.5,4.5){\makebox(0,0){$\#  $}}
    \put(4.5,-0.5){\makebox(0,0){$a_*  $}}
    \end{picture}\ee
    \bigskip\newline 
		is a morphism of $\mathcal{VB}$-groupoids, where by $AG$ we have denoted the algebroid of $G\tto P$.


\bigskip

Our considerations we begin observing that Atiyah sequences (\ref{Atiyah11}) and (\ref{Atiyah211}) could be incorporated in the short exact sequence of $\mathcal{VB}$-groupoids

		
	 \be\label{duzyVtrojkaG}\begin{picture}(11,4.6)
    \put(-7,4){\makebox(0,0){{$p_0\M p_0\times_{Ad_{G_0}}(P_0\times P_0)$}}}
    \put(-7,0){\makebox(0,0){{$p_0\M p_0\times_{Ad_{G_0}}P_0$}}}
    \put(-1,4){\makebox(0,0){$\frac{P_0\times  P_0}{G_0}$}}
    \put(-1,0){\makebox(0,0){$ \frac{P_0}{G_0}$}}
     \put(-5.1,3.5){\vector(0,-1){2.7}}
    \put(-4.9,3.5){\vector(0,-1){2.7}}
      \put(-1.1,3.5){\vector(0,-1){2.7}}
    \put(-0.9,3.5){\vector(0,-1){2.7}}
    \put(-3,4){\vector(1,0){0.8}}
    \put(-4,0){\vector(1,0){2.2}}
    \put(-4.5,3.5){\vector(4,-1){4.4}}
    \put(-4.2,-0.3){\vector(4,-1){5.2}}
      \put(-0,3.9){\line(4,-1){5}}
       \put(-0.1,3.8){\line(4,-1){5}}
         \put(-0.4,-0.1){\line(4,-1){5.5}}
       \put(-0.5,-0.2){\line(4,-1){5.5}}
			
			  \put(-3,2.7){\makebox(0,0){$\iota_2$}}
     \put(-3,-1){\makebox(0,0){$\iota$}}

    \put(2,2){\makebox(0,0){{$\frac{TP_0\times  TP_0}{G_0}$}}}
     \put(1.6,-2){\makebox(0,0){{$TP_0/G_0$}}}
      \put(6,2){\makebox(0,0){$\frac{P_0\times  P_0}{G_0}$}}
     \put(6,-2){\makebox(0,0){$ \frac{P_0}{G_0}$}}
       \put(1.9,1.5){\vector(0,-1){2.7}}
    \put(2.1,1.5){\vector(0,-1){2.7}}
      \put(5.9,1.5){\vector(0,-1){2.7}}
    \put(6.1,1.5){\vector(0,-1){2.7}}
      \put(4,2){\vector(1,0){0.8}}
    \put(3,-2){\vector(1,0){2.2}}
     \put(2.5,1.5){\vector(4,-1){4.8}}
    \put(2.6,-2.3){\vector(4,-1){4.7}}
     \put(6.4,-2){\line(4,-1){5.9}}
       \put(6.3,-2.1){\line(4,-1){5.9}}
       \put(7,2){\line(4,-1){5.5}}
       \put(6.9,1.9){\line(4,-1){5.5}}
			
			  \put(4,0.7){\makebox(0,0){$a_2$}}
     \put(4,-3){\makebox(0,0){$a$}}

    \put(9,0){\makebox(0,0){$T(\frac{P_0\times   P_0}{G_0})$}}
   \put(9,-4){\makebox(0,0){{$T(P_0/G_0)$}}}
     \put(13,0){\makebox(0,0){$\frac{P_0\times  P_0}{G_0}$}}
   \put(13,-4){\makebox(0,0){$ \frac{P_0}{G_0}$}}
        \put(8.9,-0.5){\vector(0,-1){2.7}}
    \put(9.1,-0.5){\vector(0,-1){2.7}}
      \put(12.9,-0.5){\vector(0,-1){2.7}}
    \put(13.1,-0.5){\vector(0,-1){2.7}}
        \put(10.5,0){\vector(1,0){1.3}}
    \put(10.5,-4){\vector(1,0){1.7}}
\end{picture}\ee
\vspace{2.5 cm}
\newline
which have the gauge groupoid $\frac{P_0\times P_0}{G_0}\tto P_0/G_0$ as their common side groupoid, where the vertical arrows in (\ref{duzyVtrojkaG}) are the respective source and target maps.

Let us shortly describe the $\mathcal{VB}$-groupoids included in the short exact sequence (\ref{duzyVtrojkaG}).
\begin{enumerate}[(i)]

\item  The structural maps of the left hand side $\mathcal{VB}$-groupoid in (\ref{duzyVtrojkaG}) are defined as follows
			\begin{gather*}
\tilde\Ss(\langle x,\eta,\xi\rangle):=\langle x,\xi\rangle,\\
\tilde\Tt(\langle x,\eta,\xi\rangle):=\langle x,\eta\rangle,\\
\tilde{\mathbf 1}(\langle x,\eta\rangle)=\langle x,\eta,\eta\rangle,\\
\tilde 0(\langle \eta,\xi\rangle)=\langle 0, \eta,\xi\rangle,
\end{gather*}
The groupoid product of the elements $\langle x,\eta,\xi\rangle$, $\langle y,\zeta,\delta\rangle\in p_0\M p_0\times _{Ad_{G_0}}(P_0\times P_0)$ such that $\tilde\Ss(\langle x,\eta,\xi\rangle)= \tilde \Tt(\langle y,\zeta,\delta\rangle)$, what means that $\langle x,\xi\rangle=\langle y,\zeta\rangle$, is defined by 
\be \langle x,\eta,\xi\rangle\langle y,\zeta,\delta\rangle=\langle x,\eta,\delta g^{-1}\rangle,\ee
where $g\in G_0$ satisfies $\zeta=\xi g$ and $y=Ad_g x$. 
The groupoid inverse map is $$\tilde\iota(\langle x,\eta,\xi\rangle):=\langle x,\xi,\eta\rangle.$$

\item The central $\mathcal{VB}$-groupoid in (\ref{duzyVtrojkaG}) is the quotient by $G_0$ of the tangent groupoid $TP_0\times TP_0\tto TP_0$ of the pair groupoid $P_0\times P_0\tto P_0$. 

\item The right-hand side $\mathcal{VB}$-groupoid in (\ref{duzyVtrojkaG}) is the tangent groupoid of $\frac{P_0\times P_0}{G_0}\tto P_0/G_0$.

\end{enumerate}


We see that the short exact sequence  (\ref{duzyVtrojkaG})  involves  various fundamental structures, i.e. the vector bundle, the principal bundle, the groupoid and the algebroid structures, which are consistently related one with another.

\bigskip

Let us now apply the dualization procedure, disscussed in the subsections 7.2 and 7.3 of the Appendix, to (\ref{duzyVtrojkaG}). For this reason we observe that as in the case of Atiyah sequences (\ref{Atiyah11}) and (\ref{Atiyah211}) one can define (\ref{duzyVtrojkaG}) as the quotient of 
 \be\label{duzyVtrojkaG3}\begin{picture}(11,4.6)
    \put(-7,4){\makebox(0,0){{$p_0\M p_0\times(P_0\times P_0)$}}}
    \put(-6.6,0){\makebox(0,0){{$p_0\M p_0\times P_0$}}}
    \put(-1,4){\makebox(0,0){${P_0\times  P_0}$}}
    \put(-1,0){\makebox(0,0){$ {P_0}$}}
     \put(-5.1,3.5){\vector(0,-1){2.7}}
    \put(-4.9,3.5){\vector(0,-1){2.7}}
      \put(-1.1,3.5){\vector(0,-1){2.7}}
    \put(-0.9,3.5){\vector(0,-1){2.7}}
    \put(-3,4){\vector(1,0){0.8}}
    \put(-4,0){\vector(1,0){2.2}}
    \put(-4.5,3.5){\vector(4,-1){4.4}}
    \put(-4.2,-0.3){\vector(4,-1){5.2}}
      \put(-0,3.9){\line(4,-1){5}}
       \put(-0.1,3.8){\line(4,-1){5}}
         \put(-0.4,-0.1){\line(4,-1){5.5}}
       \put(-0.5,-0.2){\line(4,-1){5.5}}
			
			  \put(-3,2.7){\makebox(0,0){$I_2$}}
     \put(-3,-1){\makebox(0,0){$I$}}

    \put(2,2){\makebox(0,0){{${TP_0\times  TP_0}$}}}
     \put(1.6,-2){\makebox(0,0){{$TP_0$}}}
      \put(6,2){\makebox(0,0){${P_0\times  P_0}$}}
     \put(6,-2){\makebox(0,0){$ {P_0}$}}
       \put(1.9,1.5){\vector(0,-1){2.7}}
    \put(2.1,1.5){\vector(0,-1){2.7}}
      \put(5.9,1.5){\vector(0,-1){2.7}}
    \put(6.1,1.5){\vector(0,-1){2.7}}
      \put(4,2){\vector(1,0){0.8}}
    \put(3,-2){\vector(1,0){2.2}}
     \put(2.5,1.5){\vector(4,-1){4.8}}
    \put(2.6,-2.3){\vector(4,-1){4.7}}
     \put(6.4,-2){\line(4,-1){5.9}}
       \put(6.3,-2.1){\line(4,-1){5.9}}
       \put(7,2){\line(4,-1){5.5}}
       \put(6.9,1.9){\line(4,-1){5.5}}
			
			  \put(4,0.7){\makebox(0,0){$A_2$}}
     \put(4,-3){\makebox(0,0){$A$}}

    \put(9,0){\makebox(0,0){$\frac{T(P_0\times   P_0)}{p_0\M p_0}$}}
   \put(8.7,-4){\makebox(0,0){{$TP_0/p_0\M p_0$}}}
     \put(13,0){\makebox(0,0){${P_0\times  P_0}$}}
   \put(13,-4){\makebox(0,0){$ {P_0}$}}
        \put(8.9,-0.5){\vector(0,-1){2.7}}
    \put(9.1,-0.5){\vector(0,-1){2.7}}
      \put(12.9,-0.5){\vector(0,-1){2.7}}
    \put(13.1,-0.5){\vector(0,-1){2.7}}
        \put(10.5,0){\vector(1,0){1.3}}
    \put(10.5,-4){\vector(1,0){1.7}}
\end{picture}\ee
\vspace{1.5 cm}
\newline  by $G_0$.

\begin{prop}\label{prop:cores} The cores of Banach  $\mathcal{VB}$-groupoids included in the short exact sequence (\ref{duzyVtrojkaG3})
are:
\be\label{core1}core(p_0\M p_0\times P_0\times  P_0\tto p_0\M p_0\times P_0)\cong P_0\times \{0\},\ee
\be\label{core2} core(TP_0\times TP_0\tto TP_0)\cong TP_0\cong \M p_0\times P_0,\ee
\be\label{core3} core\left(\frac{T(P_0\times P_0)}{p_0\M p_0}\tto TP_0/p_0\M p_0\right)\cong TP_0\cong \M p_0\times P_0.\ee
\end{prop}
\begin{proof} From the definition (\ref{core}) in the appendix we see that $(\eta,x,\xi)\in core (p_0\M p_0\times P_0\times P_0\tto p_0\M p_0\times P_0)$ if and only if $\eta=\xi$ and $x=0$. Thus one has (\ref{core1}).

The element  $(v,\eta,w,\xi)\in core(TP_0\times TP_0\tto TP_0)$ if and only if $\eta=\xi$ and $(w,\xi)=(0,\xi)$. So, we have
$core( TP_0\times TP_0\tto TP_0)=\{(v,\eta,0,\eta)\in TP_0\times TP_0\}\cong TP_0.$

 Any element $\langle v,\eta,w,\xi\rangle\in \frac{TP_0\times T P_0}{p_0\M p_0}$ is defined by
\be\label{Xquot}
\langle v,\eta,w,\xi\rangle:=\{(v+\eta x,\eta, w+\xi x,\xi);\quad x\in p_0\M p_0\}.
\ee
Then $\langle v,\eta,w,\xi\rangle\in core\left(\frac{T(P_0\times P_0)}{p_0\M p_0}\tto TP_0/p_0\M p_0\right)$ if and only
if $\eta=\xi$ and $\langle w,\xi\rangle=\langle 0,\xi\rangle$. Thus $w=\xi y$ for some $y\in p_0\M p_0$. If $x=-y$ then we find that
$$
core\left(\frac{TP_0\times T P_0}{p_0\M p_0}\tto TP_0/p_0\M p_0\right)=\{\langle v,\eta,0,\eta\rangle;\quad (v,\eta)\in TP_0\}\cong \ TP_0.
$$

\end{proof}

Using isomorphisms from Proposition \ref{prop:cores} and applying the dualization procedure described in the subsection \ref{Ap:shortexact} of the Appendix, to  (\ref{duzyVtrojkaG3}) we obtain the short exact sequence of Banach $\mathcal{VB}$-groupoids
\be\label{duzyVtrojkadual}\begin{picture}(11,4.6)
    \put(-5,4){\makebox(0,0){$T^0(P_0\times P_0)$}}
    \put(-5,0){\makebox(0,0){$T^*P_0$}}
    \put(-1,4){\makebox(0,0){${P_0\times  P_0}$}}
    \put(-1,0){\makebox(0,0){$ {P_0}$}}
     \put(-5.1,3.5){\vector(0,-1){2.7}}
    \put(-4.9,3.5){\vector(0,-1){2.7}}
      \put(-1.1,3.5){\vector(0,-1){2.7}}
    \put(-0.9,3.5){\vector(0,-1){2.7}}
    \put(-3,4){\vector(1,0){0.8}}
    \put(-4,0){\vector(1,0){2.2}}
    \put(-4.5,3.5){\vector(4,-1){4.4}}
    \put(-4.2,-0.3){\vector(4,-1){5.2}}
      \put(-0,3.9){\line(4,-1){5}}
       \put(-0.1,3.8){\line(4,-1){5}}
         \put(-0.4,-0.1){\line(4,-1){5.5}}
       \put(-0.5,-0.2){\line(4,-1){5.5}}
        \put(-3,2.5){\makebox(0,0){$A_2^*  $}}

    \put(2,2){\makebox(0,0){${T^*P_0\times  T^*P_0}$}}
     \put(2,-2){\makebox(0,0){$T^*P_0$}}
      \put(6,2){\makebox(0,0){${P_0\times  P_0}$}}
     \put(6,-2){\makebox(0,0){$ {P_0}$}}
       \put(1.9,1.5){\vector(0,-1){2.7}}
    \put(2.1,1.5){\vector(0,-1){2.7}}
      \put(5.9,1.5){\vector(0,-1){2.7}}
    \put(6.1,1.5){\vector(0,-1){2.7}}
      \put(4,2){\vector(1,0){0.8}}
    \put(3,-2){\vector(1,0){2.2}}
     \put(2.5,1.5){\vector(4,-1){4.5}}
    \put(2.6,-2.3){\vector(4,-1){4.7}}
     \put(6.4,-2){\line(4,-1){5.9}}
       \put(6.3,-2.1){\line(4,-1){5.9}}
       \put(7,2){\line(4,-1){5.5}}
       \put(6.9,1.9){\line(4,-1){5.5}}
       \put(4.7,0.4){\makebox(0,0){$I_2^*  $}}

    \put(8.4,0){\makebox(0,0){${(p_0\M p_0)^*\!\!\times\!\! P_0\!\!\times\!\!P_0}$}}
   \put(9,-4){\makebox(0,0){$\{0\}\times P_0$}}
     \put(13,0){\makebox(0,0){${P_0\times  P_0}$}}
   \put(13,-4){\makebox(0,0){$ {P_0}$}}
        \put(8.9,-0.5){\vector(0,-1){2.7}}
    \put(9.1,-0.5){\vector(0,-1){2.7}}
      \put(12.9,-0.5){\vector(0,-1){2.7}}
    \put(13.1,-0.5){\vector(0,-1){2.7}}
        \put(10.5,0){\vector(1,0){1.3}}
    \put(10.5,-4){\vector(1,0){1.7}}
    \end{picture}\ee
   \vspace{2.5 cm}
   \newline
 dual to (\ref{duzyVtrojkaG3}).	The Banach vector subbundle $T^0(P_0\times P_0)\to P_0\times P_0$ of the Banach vector bundle $T^*(P_0\times P_0)\to P_0\times P_0$ consists of such covectors which annihilate $I_2(p_0\M p_0\times P_0\times P_0)$, i.e. by the definition one has
	\be T^0(P_0\times P_0):=\{(\varphi,\eta, \psi,\xi)\in T_*P_0\times T_*P_0:\quad \varphi\eta+\psi\xi=0\}=J_2^{-1}(0),\ee
	where 
	\be\label{J2} J_2(\varphi,\eta, \psi,\xi)= \varphi\eta+\psi\xi\ee
	is the momentum map for weak symplectic manifold $T_*(P_0\times  P_0)\cong T_*P_0\times  T_*P_0$.
	The bundle monomorphism $A_2^*$ dual to $A_2$ is an inclusion map and the bundle epimorphism $I_2^*$ dual to $I_2$ is given by
	\be I_2^*(\varphi,\eta, \psi,\xi):=(\varphi\eta+\psi\xi,\eta,\xi)\ee
	
	The structural maps of a $\mathcal{VB}$-groupoid, see the subsection \ref{Ap:VB} of the Appendix, in the case of 
	  \be\label{VBdualpair}\begin{picture}(11,4.6)
    \put(1,4){\makebox(0,0){${T^*P_0\times T^*P_0}$}}
    \put(8,4){\makebox(0,0){${P_0\times P_0}$}}
    \put(1,-1){\makebox(0,0){$T^*P_0$}}
    \put(8,-1){\makebox(0,0){$P_0$}}
    \put(1.2,3){\vector(0,-1){3}}
    \put(0.7,3){\vector(0,-1){3}}
    \put(8.2,3){\vector(0,-1){3}}
    \put(7.7,3){\vector(0,-1){3}}
    \put(3,4){\vector(1,0){3}}
    \put(2.4,-1){\vector(1,0){3.7}}
    \put(0.1,1.4){\makebox(0,0){$ \  $}}
    \put(2.2,1.4){\makebox(0,0){$ \  $}}
    \put(9.1,1.4){\makebox(0,0){$ \  $}}
    \put(6.8,1.4){\makebox(0,0){$\   $}}
    \put(4.5,4.5){\makebox(0,0){$\  $}}
    \put(4.5,-0.5){\makebox(0,0){$\  $}}
    \end{picture},\ee
    \bigskip\newline
	
	are the following:
	\be\label{TPproddual} \begin{array}{l}
 \tilde\Ss_*(\varphi,\eta,\psi,\xi)=(-\psi,\xi),\\
 \tilde\Tt_*(\varphi,\eta,\psi,\xi)=(\varphi,\eta),\\
 \tilde{\mathbf{1}}(\varphi,\eta)=(\varphi,\eta,-\varphi,\eta),\\
\tilde\lambda_*(\varphi,\eta,\psi,\xi)=(\eta,\xi),\\
\tilde 0_*(\eta,\xi)=(0,\eta,0,\xi),\\
\lambda_*(\varphi,\eta)=\eta,\\
 0_*(\eta)=(0,\eta).\end{array}\ee

Its inverse map and the groupoid product  are given by
\be\label{TPproddualcd} 
 \iota_*(\varphi,\eta,\psi,\xi)=(-\psi,\xi, -\varphi,\eta),\ee
and
$$(\varphi,\eta,\psi,\xi)(-\psi,\xi,\lambda,\zeta)=(\varphi,\eta,\lambda,\zeta),$$
respectively.

The left hand side Banach $\mathcal{VB}$-groupoid $T^0(P_0\times P_0)\tto T^*P_0$ in (\ref{duzyVtrojkadual}) is a Banach subgroupoid of the intermediate $\mathcal{VB}$-groupoid in (\ref{duzyVtrojkadual}). So, its structure is defined by (\ref{TPproddual}) and (\ref{TPproddualcd}). 

The structure of the right hand side of (\ref{duzyVtrojkadual})  Banach $\mathcal{VB}$-groupoid 
 \be\label{prawy}\begin{picture}(11,4.6)
    \put(1,4){\makebox(0,0){${p_0\M_* p_0\times P_0\times P_0}$}}
    \put(8,4){\makebox(0,0){${P_0\times P_0}$}}
    \put(1,0){\makebox(0,0){$\{0\}^*\times P_0$}}
    \put(8,0){\makebox(0,0){$P_0$}}
    \put(1.2,3){\vector(0,-1){2}}
    \put(0.7,3){\vector(0,-1){2}}
    \put(8.2,3){\vector(0,-1){2}}
    \put(7.7,3){\vector(0,-1){2}}
    \put(4,4){\vector(1,0){2}}
    \put(3,0){\vector(1,0){3}}
    \put(0.1,1.4){\makebox(0,0){$ \  $}}
    \put(2.2,1.4){\makebox(0,0){$ \  $}}
    \put(9.1,1.4){\makebox(0,0){$ \  $}}
    \put(6.8,1.4){\makebox(0,0){$\   $}}
    \put(4.5,4.5){\makebox(0,0){$\  $}}
    \put(4.5,0.5){\makebox(0,0){$\  $}}
    \end{picture}\ee
    \bigskip\newline 
		is given by:
\be\label{517}\begin{array}{l}
\tilde\Ss_*(\mathcal{X},\eta,\xi):=(0,\xi),\\
\tilde\Tt_*(\mathcal{X},\eta,\xi):=(0,\eta),\\
\tilde{\mathbf{1}}_*(0,\eta):=(0,\eta,\eta),\\
 \tilde\lambda_*(\mathcal{X},\eta,\xi)=(\eta,\xi),\\
 \tilde 0_*(\eta,\xi)=(0,\eta,\xi),\\
 \lambda_*(0,\eta)=\eta,\\
  0_*(\eta)=(0,\eta)\end{array}\ee
and by
\be\label{518}\begin{array}{l}
\iota_*(\mathcal{X},\eta,\xi):=(-\mathcal{X},\eta,\xi),\\
(\mathcal{X},\eta,\xi) (\mathcal{Y},\xi,\zeta):=(\mathcal{X}+\mathcal{Y},\eta,\zeta),
\end{array}\ee
where in (\ref{517}) the structural maps and in (\ref{518}) the product and inverse map are defined .
We pay attention here to the fact that  the map $(\varphi,\eta,\psi,\xi)\mapsto(\psi,\xi, -\varphi,\eta)$ defines an isomorphism of the Banach $\mathcal{VB}$-groupoid (\ref{VBdualpair}) with the pair $\mathcal{VB}$-groupoid $T^*P_0\times T^*P_0\tto T^*P_0$. However, the distinction between these $\mathcal{VB}$-groupoids is crucial for further investigations.

\begin{rem}\label{rem52}  \textsl{Replacing in (\ref{duzyVtrojkadual}) $T^*P_0$ by $T_*P_0$ and $(p_0\M p_0)^*$ by $p_0\M_* p_0\cong (p_0\M p_0)_*$ one obtains the short exact sequence of Banach $\mathcal{VB}$-groupoids }

\be\label{duzyVtrojkapredual}\unitlength=5mm\begin{picture}(11,4.6)
    \put(-5,4){\makebox(0,0){$T^0(P_0\times P_0)$}}
    \put(-5,0){\makebox(0,0){$T_*P_0$}}
    \put(-1,4){\makebox(0,0){${P_0\times  P_0}$}}
    \put(-1,0){\makebox(0,0){$ {P_0}$}}
     \put(-5.1,3.5){\vector(0,-1){2.7}}
    \put(-4.9,3.5){\vector(0,-1){2.7}}
      \put(-1.1,3.5){\vector(0,-1){2.7}}
    \put(-0.9,3.5){\vector(0,-1){2.7}}
    \put(-3,4){\vector(1,0){0.8}}
    \put(-4,0){\vector(1,0){2.2}}
    \put(-4.5,3.5){\vector(4,-1){4.4}}
    \put(-4.2,-0.3){\vector(4,-1){5.2}}
      \put(-0,3.9){\line(4,-1){5}}
       \put(-0.1,3.8){\line(4,-1){5}}
         \put(-0.4,-0.1){\line(4,-1){5.5}}
       \put(-0.5,-0.2){\line(4,-1){5.5}}
        \put(-3,2.5){\makebox(0,0){$A_2^*  $}}

    \put(2,2){\makebox(0,0){${T_*P_0\times  T_*P_0}$}}
     \put(2,-2){\makebox(0,0){$T_*P_0$}}
      \put(6,2){\makebox(0,0){${P_0\times  P_0}$}}
     \put(6,-2){\makebox(0,0){$ {P_0}$}}
       \put(1.9,1.5){\vector(0,-1){2.7}}
    \put(2.1,1.5){\vector(0,-1){2.7}}
      \put(5.9,1.5){\vector(0,-1){2.7}}
    \put(6.1,1.5){\vector(0,-1){2.7}}
      \put(4,2){\vector(1,0){0.8}}
    \put(3,-2){\vector(1,0){2.2}}
     \put(2.5,1.5){\vector(4,-1){4.5}}
    \put(2.6,-2.3){\vector(4,-1){4.9}}
     \put(6.4,-2){\line(4,-1){5.9}}
       \put(6.3,-2.1){\line(4,-1){5.9}}
       \put(7,2){\line(4,-1){5.5}}
       \put(6.9,1.9){\line(4,-1){5.5}}
       \put(4.7,0.4){\makebox(0,0){$I_2^*  $}}

    \put(8.7,0){\makebox(0,0){${(p_0\M p_0)_*\!\!\times\!\! P_0\!\!\times\!\!P_0}$}}
   \put(9,-4){\makebox(0,0){$\{0\}\times P_0$}}
     \put(13.3,0){\makebox(0,0){${P_0\times  P_0}$}}
   \put(13,-4){\makebox(0,0){$ {P_0}.$}}
        \put(8.9,-0.5){\vector(0,-1){2.7}}
    \put(9.1,-0.5){\vector(0,-1){2.7}}
      \put(12.9,-0.5){\vector(0,-1){2.7}}
    \put(13.1,-0.5){\vector(0,-1){2.7}}
        \put(11.3,0){\vector(1,0){0.8}}
    \put(10.5,-4){\vector(1,0){1.7}}
    \end{picture}\ee
   \vspace{2.5 cm}
   \newline
 \textsl{ We recall here that $T^*P_0\cong (p_0\M p_0)^*\times P_0$ and $T_*P_0\cong (p_0\M p_0)_*\times P_0$ and  all morphisms,  structural maps and groupoid operations in (\ref{duzyVtrojkadual}) respect canonical inclusions $(p_0\M p_0)_*\subset (p_0\M p_0)^*$ and $(\M p_0)_*\subset (\M p_0)^*$ of the Banach spaces.
}
\end{rem}
\begin{rem}\label{rem53}  \textsl{The Banach bundles of (\ref{duzyVtrojkapredual}) are only the quasi Banach subbundles of their counterparts in (\ref{duzyVtrojkadual}). Because the predual Banach spaces $p_0\M_*p_0$ and $p_0\M_*$ do not have respective Banach complements in $(p_0\M p_0)^*$ and  $(\M p_0)^*$. }
\end{rem}

Let us note that the epimorphism $I_2^*$, as well as the inclusion $A^*_2$, have the $G_0$-equivariance property, i.e.
$$ I_2^*(g^{-1}\varphi,\eta g, g^{-1}\psi,\xi g)=(Ad^*_{g^{-1}}(\varphi\eta+\psi\xi),\eta g,\xi g),$$
where $g\in G_0$. Hence we have

\begin{rem} \label{rem54} \textsl{All arrows in (\ref{duzyVtrojkapredual}) are equivariant with respect to $G_0$. The actions of $G_0$ on (\ref{duzyVtrojkadual}) and (\ref{duzyVtrojkapredual}) are free and quotient maps defined by them are surjective submersions.}
\end{rem} 

\bigskip
Taking into account Remark \ref{rem52} and Remark \ref{rem54}, and then quotienting (\ref{duzyVtrojkapredual}) by $G_0$ we obtain the following short exact sequence of Banach $\mathcal{VB}$-groupoids:
  \be\label{C}\begin{picture}(11,4.6)
    \put(-5,4){\makebox(0,0){$T_*\left(\frac{P_0\times P_0}{G_0}\right)$}}
    \put(-5.3,0){\makebox(0,0){$T_*P_0/G_0$}}
    \put(-1,4){\makebox(0,0){$\frac{P_0\times  P_0}{G_0}$}}
    \put(-1,0){\makebox(0,0){$ {P_0}/{G_0}$}}
     \put(-5.1,3.5){\vector(0,-1){2.7}}
    \put(-4.9,3.5){\vector(0,-1){2.7}}
      \put(-1.1,3.5){\vector(0,-1){2.7}}
    \put(-0.9,3.5){\vector(0,-1){2.7}}
    \put(-3,4){\vector(1,0){0.8}}
    \put(-4,0){\vector(1,0){2.2}}
    \put(-4.5,3.5){\vector(4,-1){4.4}}
    \put(-4.2,-0.3){\vector(4,-1){5.2}}
      \put(-0,3.9){\line(4,-1){5}}
       \put(-0.1,3.8){\line(4,-1){5}}
         \put(-0.4,-0.1){\line(4,-1){5.5}}
       \put(-0.5,-0.2){\line(4,-1){5.5}}
        \put(-3,2.5){\makebox(0,0){$a_2^*  $}}
         \put(-2,-1.4){\makebox(0,0){$ id$}}

    \put(2,2){\makebox(0,0){$\frac{T_*P_0\times  T_*P_0}{G_0}$}}
     \put(1.8,-2){\makebox(0,0){$T_*P_0/G_0$}}
      \put(6,2){\makebox(0,0){$\frac{P_0\times  P_0}{G_0}$}}
     \put(6,-2){\makebox(0,0){$ {P_0}/{G_0}$}}
       \put(1.9,1.5){\vector(0,-1){2.7}}
    \put(2.1,1.5){\vector(0,-1){2.7}}
      \put(5.9,1.5){\vector(0,-1){2.7}}
    \put(6.1,1.5){\vector(0,-1){2.7}}
      \put(4,2){\vector(1,0){0.8}}
    \put(3,-2){\vector(1,0){2.2}}
     \put(2.5,1.5){\vector(4,-1){4.3}}
    \put(2.6,-2.3){\vector(4,-1){4.7}}
     \put(6.4,-2){\line(4,-1){5.9}}
       \put(6.3,-2.1){\line(4,-1){5.9}}
       \put(7,2){\line(4,-1){5.5}}
       \put(6.9,1.9){\line(4,-1){5.5}}
       \put(4.7,0.4){\makebox(0,0){$\iota_2^*  $}}
         \put(5,-3.4){\makebox(0,0){$[\pi^*]$}}

    \put(9,0){\makebox(0,0){$\fontsize{9pt}{1pt}{\frac{p_0\M_*p_0\times  P_0\times P_0}{G_0}}$}}
   \put(8.7,-4){\makebox(0,0){$\{0\}\times P_0/G_0$}}
     \put(13,0){\makebox(0,0){$\frac{P_0\times  P_0}{G_0}$}}
   \put(13,-4){\makebox(0,0){$ {P_0}/{G_0},$}}
        \put(8.9,-0.5){\vector(0,-1){2.7}}
    \put(9.1,-0.5){\vector(0,-1){2.7}}
      \put(12.9,-0.5){\vector(0,-1){2.7}}
    \put(13.1,-0.5){\vector(0,-1){2.7}}
        \put(10.5,0){\vector(1,0){1.3}}
    \put(10.5,-4){\vector(1,0){1.7}}
    \end{picture}\ee
   \vspace{2.5 cm}
   \newline
	which have the gauge groupoid $\frac{P_0\times P_0}{G_0}\tto P_0/G_0$ as their side groupoid. As it was mentioned at the beginning of this section these $\mathcal{VB}$-groupoids will be the main object of investigations in this section.
	
	\bigskip
	
	Since the dualization procedure commute with the quotienting by $G_0$ we easily show that:
		
		\begin{rem} \textsl{After application of the dualization procedure to  (\ref{C}) we come back to the short exact sequence (\ref{duzyVtrojkaG})}.
		\end{rem}
				
		\begin{rem} \textsl{The  Banach spaces considered here are not reflexive in general. So, the short exact sequence of $\mathcal{VB}$-groupoids dual to (\ref{duzyVtrojkaG}) can  not be equal to (\ref{C}).}
		\end{rem}
		
	
		

The upper horizontal part of (\ref{C}) 
\unitlength=5mm \be\label{Atiyapredual}\begin{picture}(11,4.6)
    \put(-2.8,4){\makebox(0,0){$T_*(\frac{P_0\times P_0}{G_0})$}}
    \put(5,4){\makebox(0,0){$\frac{T_*P_0\times T_* P_0}{G_0}$}}
    \put(13,4){\makebox(0,0){$\frac{p_0\M_*p_0\times P_0\times P_0}{G_0}$}}
    \put(12,0){\makebox(0,0){$\frac{P_0\times P_0}{G_0}$}}
    \put(5,0){\makebox(0,0){$\frac{P_0\times P_0}{G_0}$}}
    \put(-2,0){\makebox(0,0){$\frac{P_0\times P_0}{G_0}$}}

    \put(-2,3){\vector(0,-1){2}}
    \put(5,3){\vector(0,-1){2}}
    \put(12,3){\vector(0,-1){2}}
    \put(6.6,0){\vector(1,0){4}}
 \put(-0.8,0){\vector(1,0){4.5}}
 \put(-1,4){\vector(1,0){4}}
  \put(6.8,4){\vector(1,0){3.6}}
    \put(1.5,4.4){\makebox(0,0){$a^*_2$}}
     \put(8.3,4.6){\makebox(0,0){$ \iota^*_2$}}
    \put(8.5,0.5){\makebox(0,0){$\sim$}}
\put(1.5,0.5){\makebox(0,0){$ \sim$}}
    \end{picture}\ee
		is the predual Atiyah sequence for the $G_0$-principal bundle $\pi_2:P_0\times P_0\to \frac{P_0\times P_0}{G_0}$. Thus,  if one defines the sub Poisson structure on $T_*P_0\times T_*P_0$ by
		\be\label{bra515} \{f,g\}:=\langle\frac{\partial g}{\partial \eta},\frac{\partial f}{\partial \varphi}\rangle+
		\langle\frac{\partial g}{\partial \xi},\frac{\partial f}{\partial \psi}\rangle
		-\langle\frac{\partial f}{\partial \eta},\frac{\partial g}{\partial \varphi}\rangle-
		\langle\frac{\partial f}{\partial \xi},\frac{\partial g}{\partial \psi}\rangle,\ee
		where $f,g\in \mathcal{P}^\infty(T_*P_0\times T_*P_0)$, and on 
		$p_0\M_* p_0\times P_0\times P_0$  
		by 
		\be\label{sPbra2}
\{F,G\}_{sP}(\beta,\eta,\xi):=-\left\langle\beta,\left[\frac{\partial F}{\partial \beta}(\beta,\eta,\xi),\frac{\partial G}{\partial \beta}(\beta,\eta,\xi)\right]\right\rangle,\ee
where $f,g\in C^\infty(p_0	\M_* p_0\times P_0\times P_0)$, then Proposition \ref{prop:Palg}, Remark \ref{rem:42}, Theorem \ref{prop:43} and Theorem \ref{prop:44}  
are also valid for all ingredients of (\ref{Atiyapredual}).
	The Poisson algebra $\P^\infty(T_*P_0\times T_*P_0)$ in this case is defined by 
		\be \mathcal{P}^\infty(T_*P_0\times T_*P_0):=\{f\in C^\infty(T_*P_0\times T_*P_0): \frac{\partial f}{\partial\eta}(\varphi,\eta,\psi,\xi),\ \frac{\partial f}{\partial\xi}(\varphi,\eta,\psi,\xi)\in p_0\M_*\}.\ee	
		
		Now, similarly to the previous sections we discuss the coordinate description of the sub Poisson structures pictured by the diagrams (\ref{duzyVtrojkapredual}) and  (\ref{Atiyapredual}). For this reason we  present the list of charts consistent with the fibre bundle structures of the manifolds included in these diagrams and express the respective Poisson brackets using the suitable coordinates.
		\begin{enumerate} [(i)]
			\item On $\frac{P_0\times P_0}{G_0}\cong \G_{p_0}(\M)$  (see (\ref{gaugeisom}) for  this isomorphism) one has the coordinates $(y_p,z_{p\tilde p},\tilde y_{\tilde p})\in (1-p)\M p\times p\M\tilde p\times (1-\tilde p)\M\tilde p$ defined in (\ref{psipp}). Note that 
				\be\label{dzet} z_{p\tilde p}:=z_{pp_0}\tilde z_{\tilde p p_0}^{-1},\ee
				where $(y_p,z_{pp_0})$ and $(\tilde y_{\tilde p},\tilde z_{\tilde pp_0})$ are coordinates on $P_0\cap\Omega_{pp_0}$ and $P_0\cap\Omega_{\tilde pp_0}$, respectively.
		\item The coordinates 
		\be\label{coo1}(\alpha_p,\beta_{\tilde p p},\tilde \alpha_{\tilde p},y_p,z_{p\tilde p},\tilde y_{\tilde p})\in p\M_*(1-p)\times \tilde p\M_* p\times \tilde p\M_*(1-\tilde p) \times		(1-p)\M p\times p\M\tilde p\times (1-\tilde p)\M\tilde p \ee
		are the canonical coordinates on $T_*\left(\frac{P_0\times P_0}{G_0}\right)$, i.e. the variables $(\alpha_p,\beta_{\tilde p p},\tilde \alpha_{\tilde p})$ are the predual to  the variables $(a_p,b_{p\tilde p},\tilde a_{\tilde p})\in (1-p)\M p\times p\M\tilde p\times (1-\tilde p)\M p$.
		The Poisson bracket of $f,g\in \mathcal{P}^\infty(T_*\left(\frac{P_0\times P_0}{G_0}\right))$ defined by the canonical weak symplectic structure of $T_*\left(\frac{P_0\times P_0}{G_0}\right)$ written in the coordinates (\ref{coo1}) assumes the form
		\be\label{552} \{f,g\}=\left\langle\frac{\partial g}{\partial y_p},\frac{\partial f}{\partial \alpha_p}\right\rangle-
		\left\langle\frac{\partial f}{\partial y_p},\frac{\partial g}{\partial \alpha_p}\right\rangle+\ee
		$$+\left\langle\frac{\partial g}{\partial \tilde y_{\tilde p}},\frac{\partial f}{\partial  \tilde\alpha_{\tilde p}}\right\rangle-
		\left\langle\frac{\partial f}{\partial  \tilde y_{\tilde p}},\frac{\partial g}{\partial  \tilde\alpha_{\tilde p}}\right\rangle+
		\left\langle\frac{\partial g}{\partial \tilde z_{p\tilde p}},\frac{\partial f}{\partial  \tilde\beta_{\tilde p p}}\right\rangle-
		\left\langle\frac{\partial f}{\partial  \tilde z_{p\tilde p}},\frac{\partial g}{\partial  \tilde\beta_{\tilde p p}}\right\rangle.$$
		The coordinate formula of the sub Poisson anchor $\tilde{\#}_2:T^\flat(T_*(\frac{P_0\times P_0}{G_0}))\to T(T_*(\frac{P_0\times P_0}{G_0}))$ defined by (\ref{552}) is the following
		\be\label{anch3}\tilde{\#}_2(\stackrel{\circ}{\alpha_p},\stackrel{\circ}{\beta_{p\tilde p}},\stackrel{\circ}{\tilde \alpha_{\tilde p}}, \stackrel{\circ}{y_p},\stackrel{\circ}{z}_{p\tilde p},\stackrel{\circ}{\tilde y_{\tilde p}},
		{\alpha_p}, {\beta_{p\tilde p}}, {\tilde \alpha_{\tilde p}},y_p, {z}_{p\tilde p}, {\tilde y_{\tilde p}})=
		(	-\stackrel{\circ}{y_p},-\stackrel{\circ}{z}_{p\tilde p},-\stackrel{\circ}{\tilde y_{\tilde p}}, \stackrel{\circ}{\alpha_p},\stackrel{\circ}{\beta_{p\tilde p}},\stackrel{\circ}{\tilde \alpha_{\tilde p}},  {\alpha_p}, {\beta_{p\tilde p}}, {\tilde \alpha_{\tilde p}},	y_p, {z}_{p\tilde p}, {\tilde y_{\tilde p}}),\ee
		where 
		$(		\stackrel{\circ}{\alpha_p},\stackrel{\circ}{\beta_{p\tilde p}},\stackrel{\circ}{\tilde \alpha_{\tilde p}},
\stackrel{\circ}{y_p},\stackrel{\circ}{z}_{p\tilde p},\stackrel{\circ}{\tilde y_{\tilde p}},
		{\alpha_p}, {\beta_{p\tilde p}}, {\tilde \alpha_{\tilde p}},	y_p, {z}_{p\tilde p}, {\tilde y_{\tilde p}})\in 
				(1-p)\M p\times p\M \tilde{p}\times (1-\tilde{p})\M  \tilde{p}\times
p\M _*(1-p)\times \tilde{p}\M_*p\times  \tilde{p}\M _*(1-\tilde{p})\times
		p\M_*(1-p)\times \tilde{p}\M_*p\times \tilde{p}\M_*(1-\tilde{p})\times		(1-p)\M p\times p\M \tilde{p}\times  (1-\tilde{p})\M \tilde{p}
$ are coordinates consistent with the bundle structure of $T^\flat(T_*(\frac{P_0\times P_0}{G_0}))$.
		\item On $T_*P_0\times T_*P_0$ one can take the coordinates $(\alpha_p,\beta_{p},y_p,z_{pp_0},
		\tilde \alpha_{\tilde p}, \tilde\beta_{\tilde p},\tilde y_{\tilde p},\tilde z_{\tilde pp_0})$
		which are the product of the coordinates
		$(\alpha_p,\beta_{p},y_p,z_{pp_0})\in p\M_*(1-p)\times p\M_* p\times (1-p)\M p\times p\M p_0$ and $(\tilde \alpha_{\tilde p}, \tilde\beta_{\tilde p},\tilde y_{\tilde p},\tilde z_{\tilde pp_0})\in \tilde p\M_*(1-\tilde p)\times \tilde p\M_* \tilde p\times (1-\tilde p)\M \tilde p\times \tilde p\M p_0$
		on $T_*P_0$ defined in (\ref{chartTdual}). All components of (\ref{coo2}) except of $z_{pp_0}$ and $\tilde z_{\tilde pp_0}$ are $G_0$-invariant. Hence one can consider 
		\be\label{coo3}(\alpha_p,\beta_{p},\tilde \alpha_{\tilde p}, \tilde\beta_{\tilde p},y_p,z_{p\tilde p},\tilde y_{\tilde p})\in 
	p\M_*(1-p)\times p\M_* p\times 	\tilde p\M_*(1-\tilde p)\times \tilde p\M_* \tilde p\times (1-p)\M p\times p\M \tilde p\times (1-\tilde p)\M \tilde p ,\ee
		where $z_{p\tilde p}$ is defined in (\ref{dzet}), as a coordinates on $\frac{T_*P_0\times T_*P_0}{G_0}$. The Poisson bracket of $f,g\in \mathcal{P}^\infty\left(\frac{T_*P_0\times T_*P_0}{G_0}\right)$ defined by the sub Poisson structure of $\frac{T_*P_0\times T_*P_0}{G_0}$ written in the coordinates (\ref{coo3})  assumes  the following form
		\be\label{555} \{f,g\}=\left\langle\frac{\partial g}{\partial y_p},\frac{\partial f}{\partial \alpha_p}\right\rangle-
		\left\langle\frac{\partial f}{\partial y_p},\frac{\partial g}{\partial \alpha_p}\right\rangle+\ee
		$$+\left\langle\frac{\partial g}{\partial \tilde y_{\tilde p}},\frac{\partial f}{\partial  \tilde\alpha_{\tilde p}}\right\rangle-
		\left\langle\frac{\partial f}{\partial  \tilde y_{\tilde p}},\frac{\partial g}{\partial  \tilde\alpha_{\tilde p}}\right\rangle+		
		\left\langle\beta_p,\left[\frac{\partial g}{\partial \beta_{ p}},\frac{\partial f}{\partial  \beta_{ p}}\right]\right\rangle+
		\left\langle \tilde\beta_{\tilde p},\left[\frac{\partial g}{\partial  \tilde \beta_{\tilde p}},\frac{\partial f}{\partial  \tilde\beta_{\tilde p}}\right]\right\rangle+$$
		$$+\left\langle \frac{\partial g}{\partial z_{p\tilde p}},\frac{\partial f}{\partial  \beta_{ p}}z_{p\tilde p}-z_{p\tilde p}\frac{\partial f}{\partial  \tilde\beta_{\tilde p}}\right\rangle-
		\left\langle\frac{\partial f}{\partial  z_{p\tilde p}},\frac{\partial g}{\partial  \beta_{p}}z_{p\tilde p}-z_{p\tilde p}\frac{\partial g}{\partial  \tilde\beta_{\tilde p}}\right\rangle.	$$
		
		Similarly as in (\ref{tang}-\ref{bem})  we have isomorphisms
		\be\label{tang2}  T_{[(\varphi,\eta),(\psi,\xi)]}\left(\frac{T_*P_0\times T_*P_0}{G_0}\right)\cong 
		(\M p_0)_*\times (1-p_0)\M p_0\times p\M{\tilde p}\times (\M p_0)_*\times (1-p_0)\M p_0\times\{[(\varphi,\eta),(\psi,\xi)]\}\ee
\be\label{gw2}  T^*_{[(\varphi,\eta),(\psi,\xi)]}\left(\frac{T_*P_0\times T_*P_0}{G_0}\right)\cong 
\M p_0\times p_0 \M^* (1-p_0)\times {\tilde p}\M^* p\times \M p_0\times p_0 \M^* (1-p_0)\times\{[(\varphi,\eta),(\psi,\xi)]\}\ee
\be\label{bem2}  T^\flat_{[(\varphi,\eta),(\psi,\xi)]}\left(\frac{T_*P_0\times T_*P_0}{G_0}\right)\cong 
\M p_0\times p_0\M_*(1-p_0) \times {\tilde p}\M_* p\times\M p_0\times p_0\M_*(1-p_0) \{[(\varphi,\eta),(\psi,\xi)]\}\ee
where $[(\varphi,\eta),(\psi,\xi)]\in \frac{T_*P_0\times T_*P_0}{G_0}$.

		The coordinate expression for corresponding sub Poisson anchor $[\#_2]:T^\flat(\frac{T_*P_0\times T_*P_0}{G_0})\to T(\frac{T_*P_0\times T_*P_0}{G_0})$
		is 
		\be\label{anch4} [\#_2](\stackrel{\circ}{\alpha_p},\stackrel{\circ}{\beta_p},		\stackrel{\circ}{\tilde \alpha_{\tilde p}},\stackrel{\circ}{\tilde\beta_{\tilde p}},
		\stackrel{\circ}{y_p},\stackrel{\circ}{z}_{p\tilde p},
		\stackrel{\circ}{\tilde y_{\tilde p}}, 
		{\alpha_p}, \beta_p, 		{\tilde \alpha_{\tilde p}},{\tilde \beta_{\tilde p}},
		y_p,z_{p\tilde p}, 
		{\tilde y_{\tilde p}})=\qquad\qquad\ee
		$=\!(\!-\!\stackrel{\circ}{y_p},\! -ad^*_{\stackrel{\circ}{\beta_p}}(\beta_p)-z_{p\tilde p}\stackrel{\circ}{z_{p\tilde p}},		-\stackrel{\circ}{\tilde{y}_{\tilde p}}, -ad^*_{\stackrel{\circ}{\tilde{\beta}_{\tilde p}}}(\tilde{\beta}_{\tilde p})+\stackrel{\circ}{z_{p\tilde p}}{z_{p\tilde p}},
		\stackrel{\circ}{\alpha_p},\stackrel{\circ}{\beta_p}{z}_{p\tilde p}-{z}_{p\tilde p}\stackrel{\circ}{\tilde{\beta}_{\tilde p}},
		\stackrel{\circ}{\tilde{\alpha}_{\tilde p}},
		\alpha_p,\beta_p,		\tilde{\alpha}_{\tilde p},\tilde {\beta}_{\tilde p},
		y_p,{z}_{p\tilde p},
		\tilde {y}_{\tilde p}),$

where 
\be\label
{coo51} ({\alpha_p}, \beta_p,  {\tilde \alpha_{\tilde p}},{\tilde \beta_{\tilde p}},y_p,z_{p\tilde p},{\tilde y_{\tilde p}})\in p\M_*(1-p)\!\times\! p\M_* p\!\times\! 	\tilde p\M_*(1-\tilde p)\!\times\! \tilde p\M_* \tilde p\!\times\! (1-p)\M p\!\times\! p\M \tilde p\!\times\! (1-\tilde p)\M \tilde p\ee
and \be\label
{coo52} (\stackrel{\circ}{\alpha_p},\stackrel{\circ}{\beta_p},\stackrel{\circ}{\tilde \alpha_{\tilde p}},\stackrel{\circ}{\tilde\beta_{\tilde p}},
\stackrel{\circ}{y_p},\stackrel{\circ}{z}_{p\tilde p},\stackrel{\circ}{\tilde y_{\tilde p}})\in \!(1-p)\M p\!\times\!  p\M p\!\times\! 	(1-\tilde p)\M\tilde  p\!\times\! \tilde{\tilde p}\M  p\!\times\! p\M_*(1- p)\times\! \tilde p\M _* p\times\! \tilde p\M_* (1-\tilde p)\ee are the coordinates along the fibres of $T^\flat(\frac{T_*P_0\times T_*P_0}{G_0})$.
		\item  On $p_0\M_*p_0\times P_0\times P_0$ we take the coordinates
		\be\label{coo4} (\chi, y_p,z_{pp_0},\tilde y_{\tilde p}, \tilde z_{\tilde pp_0})\in 
		p_0\M_*p_0\times (1-p)\M p\times p\M p_0\times (1-\tilde p)\M \tilde p\times \tilde p\M p_0.\ee
		
		The Poisson bracket of $F,G\in \mathcal{P}^\infty(p_0\M_*p_0\times P_0\times P_0)$ in these coordinates has the form
		\be\label{533}
	\{F,G\}=-\left\langle\chi,\left[\frac{\partial F}{\partial\chi},
		\frac{\partial G}{\partial\chi}\right]\right\rangle.\ee

		The coordinate expression of  sub Poisson anchor in this case  is
		\be\label{anch5} \#(\stackrel{\circ}{\mathcal X},\stackrel{\circ}{{y}_{p}},\stackrel{\circ}{z}_{pp_0},\stackrel{\circ}{\tilde y_{\tilde p}},\stackrel{\circ}{z_{\tilde p p_0}},
			\mathcal X,y_p,z_{p p_0},{\tilde y_{\tilde p}},{z_{\tilde p p_0}})=\qquad\qquad\qquad\qquad\ee
		$$\qquad\qquad\qquad\qquad=(-ad^*_{\stackrel{\circ}{\mathcal X}}(\mathcal X),0,0,0,0,\mathcal X,y_p,z_{p p_0}, {\tilde y_{\tilde p}},{z_{\tilde p p_0}}),$$
		where $(\mathcal X,y_p,z_{p p_0},{\tilde y_{\tilde p}},{z_{\tilde p p_0}})\in p_0\M_*p_0\times (1-p)\M p\times p\M p_0\times (1-\tilde p)\M \tilde p\times \tilde p\M p_0$ and\\
		$(\stackrel{\circ}{\mathcal X},\stackrel{\circ}{{y}_{p}},\stackrel{\circ}{z}_{p p_0},\stackrel{\circ}{\tilde y_{\tilde p}},\stackrel{\circ}{z_{\tilde p p_0}})\in p_0\M p_0\times p\M_*(1- p)\times p_0\M_* p\times p\M_* \tilde(1-\tilde p) \times \tilde p_0\M_* p$.
		\item As coordinates on $\frac{p_0\M_*p_0\times P_0\times P_0}{G_0}$ one can take 
		\be\label{coo5} (\chi_p,y_p,\tilde y_{\tilde p}, z_{p\tilde p})\in p_0\M_*p_0 \times (1-p)\M p\times (1-\tilde p)\M \tilde p\times p\M\tilde p\ee
		where $z_{p\tilde p}$ is defined in (\ref{dzet}) and
		\be\label{chip}\chi_p:=z_{pp_0}\chi z_{pp_0}^{-1}.\ee
		Hence for $F,G\in C^\infty(\frac{p_0\M_*p_0\times P_0\times P_0}{G_0})$  the Poisson bracket (\ref{533}) in the coordinates (\ref{coo5}) has the form
		\be\label{560} \{F,G\}(\chi_p,y_p,\tilde y_{\tilde p}, z_{p\tilde p})=-\left\langle\chi_p,\left[\frac{\partial F}{\partial\chi_p}(\chi_p,y_p,\tilde y_{\tilde p}, z_{p\tilde p}),\frac{\partial G}{\partial\chi_p}(\chi_p,y_p,\tilde y_{\tilde p}, z_{p\tilde p})\right]\right\rangle.\ee		
		We have isomorphisms
$$T_{[(x,\xi,\eta)]}\left(\frac{p_0\M_* p_0\times P_0\times P_0}{G_0}\right)\cong p_0\M_* p_0\times (1-p)\M p\times (1-\tilde p)\M\tilde p\times p\M \tilde p\times [(x,\xi,\eta)],$$
$$T^*_{[(x,\xi,\eta)]}\left(\frac{p_0\M_* p_0\times P_0\times P_0}{G_0}\right)\cong p_0\M p_0\times p\M^*(1-p) \times \tilde p\M^*(1-\tilde p)\times \tilde p\M^*  p\times [(x,\xi,\eta)],$$
\be\label{flat1}T^\flat_{[(x,\xi,\eta)]}\left(\frac{p_0\M_* p_0\times P_0\times P_0}{G_0}\right)\cong p_0\M p_0\times p\M_*(1-p) \times \tilde p\M_*(1-\tilde p)\times \tilde p\M_*  p\times [(x,\xi,\eta)].\ee

		In this case the sub Poisson anchor  is given by 
		\be\label{anch6} \#(\stackrel{\circ}{\mathcal X}_p,\stackrel{\circ}{y}_p,\stackrel{\circ}{\tilde z_{p\tilde p}},\stackrel{\circ}{\tilde y}_{\tilde p},
		{\mathcal X}_p, y_p,\tilde z_{p\tilde p},{\tilde y}_{\tilde p})=
		(-ad^*_{\stackrel{\circ}{\mathcal X}_p}({\mathcal X}_p),0,0,0,{\mathcal X}_p, y_p,\tilde z_{p\tilde p},{\tilde y}_{\tilde p})
		\ee
		where
		$(\stackrel{\circ}{\mathcal X}_p,\stackrel{\circ}{y}_p,\stackrel{\circ}{\tilde z_{p\tilde p}},\stackrel{\circ}{\tilde y}_{\tilde p},
		{\mathcal X}_p, y_p,\tilde z_{p\tilde p},{\tilde y}_{\tilde p})\in p_0\M p_0\times p\M^* (1-p)\times \tilde p\M^* p\times \tilde p \M^*(1-p)\times p_0\M_*p_0\times (1-p)\M p\times p\M\tilde p\times (1-\tilde p)\M \tilde p$.
		
		\bigskip
		\end{enumerate}
		
		Let us mention here that the variables which are marked above by $\circ $ concern those parts of coordinate systems which are taken along the fibres of the considered $T^\flat$- bundles.
		
		\bigskip 
		
		The morphisms of $\mathcal{VB}$-groupoids, i.e. the horizontal arrows of (\ref{C}) as well as their structural maps written in the coordinates listed above assume exceptionally simple forms. Namely  for $a^*_2$ and $\iota^*_2$ we have
		\be\label{562} (\alpha_p,\beta_p,y_p,z_{p\tilde p},\tilde \alpha_{\tilde p},\tilde \beta_{\tilde p},\tilde  y_{\tilde p})=a^*_2(\alpha_p, \beta_{\tilde pp}, \tilde\alpha_{\tilde p},y_p,z_{p\tilde p}, \tilde y _{\tilde p})=\ee
		$$=(\alpha_p,z_{p\tilde p}\beta_{\tilde pp},y_p,z_{p\tilde p},\tilde \alpha_{\tilde p},-\tilde \beta_{\tilde pp}z_{p\tilde p},\tilde  y_{\tilde p})$$
		and
		\be\label{563} (\chi_p,y_p,z_{p\tilde p},\tilde  y_{\tilde p})=\iota^*_2(\alpha_p,\beta_p,y_p,z_{p\tilde p},\tilde \alpha_{\tilde p},\tilde \beta_{\tilde p},\tilde  y_{\tilde p})=(\beta_p +z_{p\tilde p}\tilde\beta_{\tilde p}z_{p\tilde p}^{-1},y_p,z_{p\tilde p},\tilde  y_{\tilde p}),\ee
		respectively. 
		
		The structural maps of the $\mathcal{VB}$-groupoid $\frac{T_*P_0\times T_*P_0}{G_0}\tto T_*P_0/G_0$ which are obtained  by the quotienting of the structural maps (\ref{TPproddual}), written in the coordinates $(\alpha_p,\beta_p,y_p,z_{p\tilde p},\tilde \alpha_{\tilde p},\tilde \beta_{\tilde p},\tilde  y_{\tilde p})$, assume the following form
		\be \tilde\Ss_*(\alpha_p,\beta_p,y_p,z_{p\tilde p},\tilde \alpha_{\tilde p},\tilde \beta_{\tilde p},\tilde  y_{\tilde p})=(-\tilde \alpha_{\tilde p},-\tilde \beta_{\tilde p},\tilde  y_{\tilde p})\ee
		\be \tilde\Tt_*(\alpha_p,\beta_p,y_p,z_{p\tilde p},\tilde \alpha_{\tilde p},\tilde \beta_{\tilde p},\tilde  y_{\tilde p})=(\alpha_p,\beta_p,y_p)\ee
		\be\tilde{\mathbf 1}_*(\alpha_p,\beta_p,y_p)=(\alpha_p,\beta_p,y_p,p,-\alpha_p,-\beta_p,y_p)\ee
		\be \lambda_*(\alpha_p,\beta_p,y_p)=y_p\ee
				\be 0_*(y_p)=(0,0,y_p)\ee
				\be \tilde\lambda_*(\alpha_p,\beta_p,y_p,z_{p\tilde p},\tilde \alpha_{\tilde p},\tilde \beta_{\tilde p},\tilde  y_{\tilde p})=(y_p,z_{p\tilde p},\tilde  y_{\tilde p})\ee
				\be \tilde 0_*(y_p,z_{p\tilde p},\tilde  y_{\tilde p})=(0,0,y_p,z_{p\tilde p},0,0,\tilde  y_{\tilde p})\ee
				and the groupoid product is given by 
		$$(\alpha_p,\beta_p,y_p,z_{p\tilde p},-\tilde\alpha_{\tilde p},-\tilde\beta_{\tilde p},\tilde y_{\tilde p})\cdot
		({\tilde\alpha}_{\tilde p},{\tilde\beta}_{\tilde p},{\tilde y}_{\tilde p},{z}_{\tilde p\tilde{\tilde p}},\tilde{\tilde\alpha}_{\tilde{\tilde p}},\tilde{\tilde\beta}_{\tilde{\tilde p}},\tilde{\tilde y}_{\tilde{\tilde p}})=$$
		$$=(\alpha_p,\beta_p,y_p,z_{p\tilde p}z_{\tilde p\tilde {\tilde p}},\tilde{\tilde\alpha}_{\tilde{\tilde p}},\tilde{\tilde\beta}_{\tilde{\tilde p}},\tilde{\tilde y}_{\tilde{\tilde p}}).$$
				
		\bigskip


 Now let us consider the $\mathcal{VB}$-groupoid 
		\be\label{prop:prol}\begin{picture}(11,4.6)
    \put(0,4){\makebox(0,0){$T(T_*P_0\times T_*P_0)$}}
    \put(8,4){\makebox(0,0){$T_*P_0\times T_*P_0$}}
    \put(1,-1){\makebox(0,0){$T(T_*P_0)$}}
    \put(8,-1){\makebox(0,0){$T_*P_0$}}
    \put(1.2,3){\vector(0,-1){3}}
    \put(0.7,3){\vector(0,-1){3}}
    \put(8.2,3){\vector(0,-1){3}}
    \put(7.7,3){\vector(0,-1){3}}
    \put(2.8,4){\vector(1,0){3}}
    \put(2.6,-1){\vector(1,0){3.9}}
    \put(0.1,1.4){\makebox(0,0){$ T\tilde\Tt_*  $}}
    \put(2.2,1.4){\makebox(0,0){$ T\tilde\Ss_* $}}
    \put(8.8,1.4){\makebox(0,0){$ \tilde\Ss_*  $}}
    \put(7.1,1.4){\makebox(0,0){$\tilde\Tt_*    $}}
    \put(4.5,4.5){\makebox(0,0){$\  $}}
    \put(4.5,-0.5){\makebox(0,0){$\  $}}
    \end{picture},\ee
    \bigskip\newline
		which is the tangent prolongation of the
		groupoid $T_*P_0\times T_*P_0\tto T_*P_0$, which  is a subgroupoid of the groupoid $T^*P_0\times T^*P_0\tto T^*P_0$. So, its structure   is defined in (\ref{TPproddual}) and (\ref{TPproddualcd}). Taking into account isomorphisms
		\be\label{prop:1izo2} T(T_*P_0)\cong  p_0\M_*\times\M p_0\times p_0\M _*\times P_0\ee
		and
			\be\label{prop:izo2} T(T_*P_0\times T_*P_0)\cong  (\M p_0)_*\times\M p_0\times (\M p_0)_*\times\M p_0\times(\M p_0)_*\times P_0\times (\M p_0)_*\times P_0\ee
	we find that the structural maps of (\ref{prop:prol}) are:
	\be\begin{array}{l} T\tilde\Ss_*(\dot{\varphi},\dot{\eta},\dot{\psi},\dot{\xi},\varphi,\eta,\psi,\xi)=(-\dot{\psi},\dot{\xi},-\psi,\xi)\\
	 T\tilde\Tt_*(\dot{\varphi},\dot{\eta},\dot{\psi},\dot{\xi},\varphi,\eta,\psi,\xi)=(\dot\varphi,\dot\eta,\varphi,\eta)\\
	 T\tilde\iota_*(\dot{\varphi},\dot{\eta},\dot{\psi},\dot{\xi},\varphi,\eta,\psi,\xi)=(-\dot\psi,\dot\xi,-\dot\varphi,\dot\eta,-\psi,\xi,-\varphi,\eta)\\
	 T\tilde{\mathbf{1}}_*(\dot{\varphi},\dot{\eta},\varphi,\eta)=(\dot{\varphi},\dot{\eta},-\dot{\varphi},\dot{\eta},\varphi,\eta,-\varphi,\eta)\\
	 \tilde\lambda(\dot{\varphi},\dot{\eta},\dot{\psi},\dot{\xi},\varphi,\eta,\psi,\xi)=(\varphi,\eta,\psi,\xi)\\
	\lambda(\dot{\varphi},\dot{\eta},\varphi,\eta)=(\varphi,\eta)\\
	\tilde 0(\varphi,\eta,\psi,\xi)=(0,0,0,0,\varphi,\eta,\psi,\xi)\\
	0(\varphi,\eta)=(0,0,\varphi,\eta)\end{array}\ee
	and its groupoid product is given by
		\be (\dot{\varphi},\dot{\eta},\dot{\psi},\dot{\xi},\varphi,\eta,\psi,\xi)(-\dot\psi,\dot\xi,\dot\sigma,\dot\lambda,-\psi,\xi,\sigma,\lambda)=(\dot\varphi,\dot\eta,\dot\sigma,\dot\lambda,\varphi,\eta,\sigma,\lambda).\ee
One easy sees that the core of (\ref{prop:prol}) is isomorphic with $T(T_*P_0)$. So,  the dual $\mathcal{VB}$-groupoid of (\ref{prop:prol}) is
		\be\label{prop:proldual}\begin{picture}(11,4.6)
    \put(0,4){\makebox(0,0){$T^*(T_*P_0\times T_*P_0)$}}
    \put(8,4){\makebox(0,0){$T_*P_0\times T_*P_0$}}
    \put(1,-1){\makebox(0,0){$T^*(T_*P_0)$}}
    \put(8,-1){\makebox(0,0){$T_*P_0$}}
    \put(1.2,3){\vector(0,-1){3}}
    \put(0.7,3){\vector(0,-1){3}}
    \put(8.2,3){\vector(0,-1){3}}
    \put(7.7,3){\vector(0,-1){3}}
    \put(3,4){\vector(1,0){3}}
    \put(2.7,-1){\vector(1,0){3.7}}
    \put(0.1,1.4){\makebox(0,0){$ T^*\tilde\Tt_*  $}}
    \put(2.2,1.4){\makebox(0,0){$ T^*\tilde\Ss_* $}}
    \put(8.8,1.4){\makebox(0,0){$ \tilde\Ss_*  $}}
    \put(7.1,1.4){\makebox(0,0){$\tilde\Tt_*    $}}
    \put(4.5,4.5){\makebox(0,0){$\  $}}
    \put(4.5,-0.5){\makebox(0,0){$\  $}}
    \end{picture}\ee
    \bigskip\newline 
		Using isomorphisms:
			\be T^*(T_*P_0)\cong  \M p_0\times(\M p_0)^*\times(\M p_0)_*\times P_0\ee
		and
			\be\label{prop:iso4} T^*(T_*P_0\times T_*P_0)\cong  (\M p_0)\times(\M p_0)^*\times \M p_0\times(\M p_0)^*\times(\M p_0)_*\times P_0\times (\M p_0)_*\times P_0\ee we write the structural maps  of (\ref{prop:proldual}) as follows:
		\be\begin{array}{l} 
		T^*\tilde\Ss_*(\stackrel{\circ}{\varphi},\stackrel{\circ}{\eta},\stackrel{\circ}{\psi},\stackrel{\circ}{\xi},\varphi,\eta,\psi,\xi)=
		(\stackrel{\circ}{\psi},-\stackrel{\circ}{\xi},-\psi,\xi)\\
		T^*\tilde\Tt_*(\stackrel{\circ}{\varphi},\stackrel{\circ}{\eta},\stackrel{\circ}{\psi},\stackrel{\circ}{\xi},\varphi,\eta,\psi,\xi)=(\stackrel{\circ}{\varphi},\stackrel{\circ}{\eta},\varphi,\eta)\\
	 T^*\tilde\iota_* (\stackrel{\circ}{\varphi},\stackrel{\circ}{\eta},\stackrel{\circ}{\psi},\stackrel{\circ}{\xi},\varphi,\eta,\psi,\xi)= (\stackrel{\circ}{\psi},-\stackrel{\circ}{\xi},\stackrel{\circ}{\varphi},-\stackrel{\circ}{\eta},-\psi,\xi,-\varphi,\eta)\\
  T^*\tilde\tilde{\mathbf{1}}_*(\stackrel{\circ}{\varphi},\stackrel{\circ}{\eta},\varphi,\eta)=(\stackrel{\circ}{\varphi},\stackrel{\circ}{\eta},\stackrel{\circ}{\varphi},-\stackrel{\circ}{\eta},\varphi,\eta,-\varphi,\eta)\\
	 \tilde\lambda(\stackrel{\circ}{\varphi},\stackrel{\circ}{\eta},\stackrel{\circ}{\psi},\stackrel{\circ}{\xi},\varphi,\eta,\psi,\xi)=(\varphi,\eta,\psi,\xi)\\
	\lambda(\stackrel{\circ}{\varphi},\stackrel{\circ}{\eta},\varphi,\eta)=(\varphi,\eta)\\
	\tilde 0(\varphi,\eta,\psi,\xi)=(0,0,0,0,\varphi,\eta,\psi,\xi)\\
	0(\varphi,\eta)=(0,0,\varphi,\eta)\end{array}\ee
	and the groupoid product	 is given by 
	\be(\stackrel{\circ}{\varphi},\stackrel{\circ}{\eta},\stackrel{\circ}{\psi},\stackrel{\circ}{\xi},\varphi,\eta,\psi,\xi)
	(\stackrel{\circ}{\psi},-\stackrel{\circ}{\xi},\stackrel{\circ}{\sigma},\stackrel{\circ}{\lambda},-\psi,\xi,\sigma,\lambda)= 
	(\stackrel{\circ}{\varphi},\stackrel{\circ}{\eta},\stackrel{\circ}{\sigma},\stackrel{\circ}{\lambda},\varphi,\eta,\sigma,\lambda).\ee

		The quasi Banach subbundles $T^\flat(T_*P_0)\subset T^*(T_*P_0)$ and $T^\flat(T_*P_0\times T_*P_0)\subset T^*(T_*P_0\times T_*P_0)$  are isomorphic to
		\be\label{prop:1izo3} T^\flat(T_*P_0)\cong \M p_0 \times p_0\M_* \times  p_0\M_*\times P_0\ee
		and 
		\be\label{prop:izo3} T^\flat(T_*P_0\times T_*P_0)\cong (\M p_0)\times(\M p_0)_*\times \M p_0\times(\M p_0)_* \times(\M p_0)_*\times P_0\times (\M p_0)_*\times P_0 ,\ee
		respectively.
		 Using the isomorphisms (\ref{prop:1izo2}), (\ref{prop:izo2}), (\ref{prop:1izo3}) and (\ref{prop:izo3}) we write the sub Poisson  maps $ \#_1: T^\flat(T_*P_0)\to T(T_*P_0)$ and  $ \#_2:T^\flat(T_*P_0\times T_*P_0)\to T(T_*P_0\times T_*P_0)$, which are defined by the brackets (\ref{nawias3}) and (\ref{bra515}), as follows
		\be\label{poi1} \#_1(\stackrel{\circ}{\varphi},\stackrel{\circ}{\eta},\varphi,\eta)=(-\stackrel{\circ}{\eta},\stackrel{\circ}{\varphi},\varphi,\eta).\ee 
		and 
		\be\label{poi2}  \#_2(\stackrel{\circ}{\varphi},\stackrel{\circ}{\eta},\stackrel{\circ}{\psi},\stackrel{\circ}{\xi},\varphi,\eta,\psi,\xi)=
		(-\stackrel{\circ}{\eta},\stackrel{\circ}{\varphi},-\stackrel{\circ}{\xi},\stackrel{\circ}{\psi},\varphi,\eta,\psi,\xi).\ee 
		
		We note here that the $\mathcal{VB}$-subgroupoid 
		\be\label{bemol}\begin{picture}(11,4.6)
    \put(0,4){\makebox(0,0){$T^\flat(T_*P_0\times T_*P_0)$}}
    \put(8,4){\makebox(0,0){$T^*(T_*P_0\times T_*P_0)$}}
    \put(1,-1){\makebox(0,0){$T^\flat(T_*P_0)$}}
    \put(8,-1){\makebox(0,0){$T^*(T_*P_0)$}}
    \put(1.2,3){\vector(0,-1){3}}
    \put(0.7,3){\vector(0,-1){3}}
    \put(8.2,3){\vector(0,-1){3}}
    \put(7.7,3){\vector(0,-1){3}}
    \put(3,4){\vector(1,0){2.3}}
    \put(2.4,-1){\vector(1,0){3.7}}
    \put(0.1,1.4){\makebox(0,0){$ \   $}}
    \put(2.2,1.4){\makebox(0,0){$ \  $}}
    \put(8.8,1.4){\makebox(0,0){$ \  $}}
    \put(7.1,1.4){\makebox(0,0){$\   $}}
    \put(4.5,4.5){\makebox(0,0){$\   $}}
    \put(4.5,-0.5){\makebox(0,0){$\   $}}
    \end{picture}\ee
    \bigskip\newline
		of the $\mathcal{VB}$-groupoid  (\ref{prop:proldual}) will be crucial for the following considerations. Also the momentum maps $J_{1\flat}:T^\flat(T_*P_0)\to p_0\M_*p_0$ and $J_{2\flat}:T^\flat(T_*P_0\times T_*P_0)\to p_0\M_*p_0$ which are given by
		\be J_{1\flat}(\stackrel{\circ}{\varphi},\stackrel{\circ}{\eta},\varphi,\eta)=\stackrel{\circ}{\eta}\eta-\varphi\stackrel{\circ}{\varphi}\ee
		and 
		\be J_{2\flat}(\stackrel{\circ}{\varphi},\stackrel{\circ}{\eta},\stackrel{\circ}{\psi},\stackrel{\circ}{\xi},\varphi,\eta,\psi,\xi)=\stackrel{\circ}{\eta}\eta-\varphi\stackrel{\circ}{\varphi}+\stackrel{\circ}{\xi}\xi-\psi\stackrel{\circ}{\psi}\ee
		will be important in subsequel. 
		
		One has a sequence of quasi Banach vector subbundles
		\be\label{bundleduzy}\begin{picture}(11,4.6)
\put(-6,4){\makebox(0,0){$(T^*\tilde{\Tt}_*)^{-1}(J_{1\flat}^{-1}(0))\cap(T^*\tilde{\Ss}_*)^{-1}(J_{1\flat}^{-1}(0))$}}
    \put(4.3,4){\makebox(0,0){$J_{2\flat}^{-1}(0)$}}
    \put(14,4){\makebox(0,0){$T^\flat\left({T_*P_0\times T_*P_0}\right)$}}
    
    \put(-6,0){\makebox(0,0){$T_*P_0\times T_*P_0$}}
    \put(4,0){\makebox(0,0){$T_*P_0\times T_*P_0$}}
		\put(14,0){\makebox(0,0){$T_*P_0\times T_*P_0$}}
   
      \put(-6,3){\vector(0,-1){2}}
    \put(4,3){\vector(0,-1){2}}
		  \put(14,3){\vector(0,-1){2}}
       \put(0,4){\vector(1,0){3}}
		\put(6,4){\vector(1,0){4.3}}
   
\put(-3.5,0){\vector(1,0){5}}
		\put(6.5,0){\vector(1,0){5}}
   
    \put(0.1,1.4){\makebox(0,0){$ \ $}}
    \put(2.2,1.4){\makebox(0,0){$  \ $}}
    \put(8.8,1.4){\makebox(0,0){$  \  $}}
    \put(7.1,1.4){\makebox(0,0){$ \    $}}
    \end{picture}\ee
		of the vector bundle  $T^\flat\left({T_*P_0\times T_*P_0}\right)\to T_*P_0\times T_*P_0$. 
		
		Let us define the bundle $\mathfrak{J}\to J_2^{-1}(0)$ as the restriction of the first subbundle in (\ref{bundleduzy}) to the submanifold $J_2^{-1}(0){\hookrightarrow} T_*P_0\times T_*P_0$.
		
		\begin{lem} One has the following sequence of $\mathcal{VB}$-groupoids morphisms
				\be\label{lem:duzy}\begin{picture}(11,4.6)
\put(-8,4){\makebox(0,0){$\mathfrak{J}$}}
    \put(-2.2,4){\makebox(0,0){$J_{2\flat}^{-1}(0)$}}
    \put(4,4){\makebox(0,0){$T^\flat\left({T_*P_0\!\times\! T_*P_0}\right)$}}
    \put(10,4){\makebox(0,0){$T\left({T_*P_0\!\times\! T_*P_0}\right)$}}
		 \put(16,4){\makebox(0,0){$\frac{T\left({T_*P_0\times T_*P_0}\right)}{T_eG_0}$}}
    \put(-8,-1){\makebox(0,0){$J_{1\flat}^{-1}(0)$}}
    \put(-2,-1){\makebox(0,0){${T^\flat({T_*P_0})}$}}
		\put(4,-1){\makebox(0,0){${T^\flat({T_*P_0})}$}}
    \put(10,-1){\makebox(0,0){${T({T_*P_0})}$}}
		\put(16,-1){\makebox(0,0){$\frac{T({T_*P_0})}{T_eG_0},$}}
      \put(-8.2,3){\vector(0,-1){3}}
    \put(-7.8,3){\vector(0,-1){3}}
		  \put(-2.2,3){\vector(0,-1){3}}
    \put(-1.8,3){\vector(0,-1){3}}
		\put(4.2,3){\vector(0,-1){3}}
    \put(3.8,3){\vector(0,-1){3}}
    \put(10.2,3){\vector(0,-1){3}}
    \put(9.8,3){\vector(0,-1){3}}
		  \put(16.2,3){\vector(0,-1){3}}
    \put(15.8,3){\vector(0,-1){3}}
    \put(-7,4){\vector(1,0){3.4}}
		\put(-1,4){\vector(1,0){2.3}}
    \put(6.5,4){\vector(1,0){1}}
    \put(12.5,4){\vector(1,0){1.3}}

\put(-6.5,-1){\vector(1,0){3}}
		\put(-0.5,-1){\vector(1,0){3}}
    \put(5.5,-1){\vector(1,0){3}}
    \put(11.5,-1){\vector(1,0){3}}
		\put(-5.5,4.4){\makebox(0,0){$ \iota_2 $}}
    \put(0.3,4.4){\makebox(0,0){$  \tilde{\iota}_2 $}}
    \put(7,4.4){\makebox(0,0){$  \#_2  $}}
    \put(13,4.4){\makebox(0,0){$ Q_2    $}}
    \put(-5,-0.6){\makebox(0,0){$ \iota_1 $}}
    \put(1,-0.6){\makebox(0,0){$  id $}}
    \put(7,-0.6){\makebox(0,0){$  \#_1  $}}
    \put(13,-0.6){\makebox(0,0){$ Q_1    $}}
    \end{picture}\ee
    \bigskip\newline
				where $\#_1$ and $\#_2$ are as in (\ref{poi1}) and (\ref{poi2}). The maps $Q_1$ and $Q_2$ are the respective quotient maps, see (\ref{Q1}) and (\ref{Q2}). The $\mathcal{VB}$-groupoid $\mathfrak{J}\tto J_{1\flat}^{-1}(0)$ has $J_{2\flat}^{-1}(0)\tto T_*P_0$ as its side groupoid. The side groupoid of others $\mathcal{VB}$-groupoids in (\ref{lem:duzy}) is $T_*P_0\times T_*P_0\tto T_*P_0$.

		\end{lem}
		\begin{proof}
		
		The equalities 
			\be\begin{split} (T\tilde\Tt_*\circ  \#_2)&(\stackrel{\circ}{\varphi},\stackrel{\circ}{\eta},\stackrel{\circ}{\psi},\stackrel{\circ}{\xi},\varphi,\eta,\psi,\xi)=T\tilde\Tt_*(-\stackrel{\circ}{\eta},\stackrel{\circ}{\varphi},-\stackrel{\circ}{\xi},\stackrel{\circ}{\psi},\varphi,\eta,\psi,\xi)=\\
			&=(-\stackrel{\circ}{\eta},\stackrel{\circ}{\varphi},\varphi,\eta)=\#_1(\stackrel{\circ}{\varphi},\stackrel{\circ}{\eta},\varphi,\eta)=(\#_1\circ T^*\tilde{\Tt}_*)(\stackrel{\circ}{\varphi},\stackrel{\circ}{\eta},\stackrel{\circ}{\psi},\stackrel{\circ}{\xi},\varphi,\eta,\psi,\xi)\end{split}\ee
			\be \begin{split}(T\tilde\Ss_*\circ  \#_2)&(\stackrel{\circ}{\varphi},\stackrel{\circ}{\eta},\stackrel{\circ}{\psi},\stackrel{\circ}{\xi},\varphi,\eta,\psi,\xi)=T\tilde\Ss_*(-\stackrel{\circ}{\eta},\stackrel{\circ}{\varphi},-\stackrel{\circ}{\xi},\stackrel{\circ}{\psi},\varphi,\eta,\psi,\xi)=\\
			&=(\stackrel{\circ}{\xi},\stackrel{\circ}{\psi},-\psi,\xi)=\#_1(\stackrel{\circ}{\psi},-\stackrel{\circ}{\xi},-\psi,\xi)=(\#_1\circ T^*\tilde{\Ss}_*)(\stackrel{\circ}{\varphi},\stackrel{\circ}{\eta},\stackrel{\circ}{\psi},\stackrel{\circ}{\xi},\varphi,\eta,\psi,\xi)\end{split}\ee
	\be\begin{split} (\#_2\circ T^*\tilde{\iota}_*)&(\stackrel{\circ}{\varphi},\stackrel{\circ}{\eta},\stackrel{\circ}{\psi},\stackrel{\circ}{\xi},\varphi,\eta,\psi,\xi)=\#_2(\stackrel{\circ}{\psi},-\stackrel{\circ}{\xi},\stackrel{\circ}{\varphi},-\stackrel{\circ}{\eta},-\psi,\xi,-\varphi,\eta)=
	(\stackrel{\circ}{\xi},\stackrel{\circ}{\psi},\stackrel{\circ}{\eta},\stackrel{\circ}{\varphi},-\phi,\xi,-\varphi,\eta)=\\
	&=T\tilde{\iota}_*(-\stackrel{\circ}{\eta},\stackrel{\circ}{\varphi},-\stackrel{\circ}{\xi},\stackrel{\circ}{\psi},\varphi,\eta,\psi,\xi)=(T\tilde{\iota}_*\circ \#_2)(\stackrel{\circ}{\varphi},\stackrel{\circ}{\eta},\stackrel{\circ}{\psi},\stackrel{\circ}{\xi},\varphi,\eta,\psi,\xi)\end{split}\ee
			\be\begin{split}( \#_2(\stackrel{\circ}{\varphi},\stackrel{\circ}{\eta},\stackrel{\circ}{\psi},\stackrel{\circ}{\xi},&\varphi,\eta,\psi,\xi))\cdot (\#_2(\stackrel{\circ}{\psi},-\stackrel{\circ}{\xi},\stackrel{\circ}{\sigma},\stackrel{\circ}{\lambda},-\psi,\xi,\sigma,\lambda))=\\
			&=(-\stackrel{\circ}{\eta},\stackrel{\circ}{\varphi},-\stackrel{\circ}{\xi},\stackrel{\circ}{\psi},\varphi,\eta,\psi,\xi)(\stackrel{\circ}{\xi},\stackrel{\circ}{\psi},-\stackrel{\circ}{\lambda},\stackrel{\circ}{\sigma},-\psi,\xi,\sigma,\lambda)=\\
			&\qquad=	(-\stackrel{\circ}{\eta},\stackrel{\circ}{\varphi},-\stackrel{\circ}{\lambda},\stackrel{\circ}{\sigma},\varphi,\eta,\sigma,\lambda)
			=\#_2(\stackrel{\circ}{\varphi},\stackrel{\circ}{\eta},\stackrel{\circ}{\sigma},\stackrel{\circ}{\lambda},\varphi,\eta,\sigma,\lambda)= \\
	&\qquad\qquad=\#_2\left((\stackrel{\circ}{\varphi},\stackrel{\circ}{\eta},\stackrel{\circ}{\psi},\stackrel{\circ}{\xi},\varphi,\eta,\psi,\xi)
	(\stackrel{\circ}{\psi},-\stackrel{\circ}{\xi},\stackrel{\circ}{\sigma},\stackrel{\circ}{\lambda},-\psi,\xi,\sigma,\lambda)\right)\end{split}\ee
		show that $\#_1$ and $\#_2$ define the  groupoids morphism
			\be\label{prop:morph}\begin{picture}(11,4.6)
    \put(0,4){\makebox(0,0){$T^\flat(T_*P_0\times T_*P_0)$}}
    \put(8.3,4){\makebox(0,0){$T(T_*P_0\times T_*P_0)$}}
    \put(1,-1){\makebox(0,0){$T^\flat(T_*P_0)$}}
    \put(8,-1){\makebox(0,0){$T(T_*P_0)$}}
    \put(1.2,3){\vector(0,-1){3}}
    \put(0.7,3){\vector(0,-1){3}}
    \put(8.2,3){\vector(0,-1){3}}
    \put(7.7,3){\vector(0,-1){3}}
    \put(3,4){\vector(1,0){2.8}}
    \put(2.4,-1){\vector(1,0){3.7}}
    \put(0.1,1.4){\makebox(0,0){$ \Tt^*  $}}
    \put(2.2,1.4){\makebox(0,0){$ \Ss^* $}}
    \put(8.8,1.4){\makebox(0,0){$ T\tilde\Ss_*  $}}
    \put(7.1,1.4){\makebox(0,0){$T\tilde\Tt_*    $}}
    \put(4.5,4.5){\makebox(0,0){$\#_2  $}}
    \put(4.5,-0.5){\makebox(0,0){$\#_1  $}}
    \end{picture}\ee
    \bigskip\newline 
			In order to see that $(\iota_1,\iota_2) $ and $(id,\tilde{\iota}_2)$ define the groupoids morphisms we note that the coordinate description of the considered manifolds is the following
			\be\label{J1bemol} J_{1\flat}^{-1}(0)=\{(\stackrel{\circ}{\varphi},\stackrel{\circ}{\eta},\varphi,\eta):\quad \stackrel{\circ}{\eta}\eta-\varphi\stackrel{\circ}{\varphi}=0\}.\ee
		
		\be\label{J2bemol} J_{2\flat}^{-1}(0)=\{(\stackrel{\circ}{\varphi},\stackrel{\circ}{\eta},\stackrel{\circ}{\psi},\stackrel{\circ}{\xi},\varphi,\eta,\psi,\xi):\quad \stackrel{\circ}{\eta}\eta-\varphi\stackrel{\circ}{\varphi}+\stackrel{\circ}{\xi}\xi-\psi\stackrel{\circ}{\psi}=0\}\ee
		and 
			\be\label{Jbezbemol} \mathfrak{J}=\{(\stackrel{\circ}{\varphi},\stackrel{\circ}{\eta},\stackrel{\circ}{\psi},\stackrel{\circ}{\xi},\varphi,\eta,\psi,\xi):\quad \stackrel{\circ}{\eta}\eta-\varphi\stackrel{\circ}{\varphi}=0,\quad \stackrel{\circ}{\xi}\xi-\psi\stackrel{\circ}{\psi}=0,\ \ \varphi\eta+\psi\xi=0 \}.\ee
			Next we note that the conditions mentioned in (\ref{J1bemol}-\ref{Jbezbemol}) are invariant with respect to the groupoids operations. The quotient maps $Q_1$ and $Q_2$ are defined by 
			\be\label{Q1} Q_1(\dot\varphi,\dot\eta,\varphi,\eta):=[(\dot\varphi,\dot\eta,\varphi,\eta)]=
			\{(\dot\varphi-x\varphi,\dot\eta+\eta x,\varphi,\eta):\ x\in p_0\M p_0\}\ee
			\be\label{Q2} Q_1(\dot\varphi,\dot\eta,\dot\psi,\dot\xi,\varphi,\eta,\psi,\xi):=[(\dot\varphi,\dot\eta,\dot\psi,\dot\xi,\varphi,\eta,\psi,\xi)]=\{(\dot\varphi-x\varphi,\dot\eta+\eta x,\dot\psi-x\psi,\dot\xi+\xi x,\varphi,\eta,\psi,\xi):\ x\in p_0\M p_0\}.\ee
			Using (\ref{Q1}) and (\ref{Q2}) we can show by the directly calculation that $(Q_1,Q_2)$ also defines a groupoid morphism.
		
\end{proof}

\bigskip

	Now we consider the $\mathcal{VB}$-groupoid 
		 \be\label{prawytang}\begin{picture}(11,4.6)
    \put(0.6,4){\makebox(0,0){${T(p_0\M_* p_0\!\times\! P_0\!\times\! P_0})$}}
    \put(8.5,4){\makebox(0,0){${p_0\M_* p_0\times P_0\times P_0}$}}
    \put(1,0){\makebox(0,0){$T(\{0\}^*\times P_0)$}}
    \put(8,0){\makebox(0,0){$\{0\}^*\times P_0$}}
    \put(1.2,3){\vector(0,-1){2}}
    \put(0.7,3){\vector(0,-1){2}}
    \put(8.2,3){\vector(0,-1){2}}
    \put(7.7,3){\vector(0,-1){2}}
    \put(4,4){\vector(1,0){1.3}}
    \put(3.3,0){\vector(1,0){3}}
    \put(0.1,1.4){\makebox(0,0){$ \  $}}
    \put(2.2,1.4){\makebox(0,0){$ \  $}}
    \put(9.1,1.4){\makebox(0,0){$ \  $}}
    \put(6.8,1.4){\makebox(0,0){$\   $}}
    \put(4.5,4.5){\makebox(0,0){$\  $}}
    \put(4.5,0.5){\makebox(0,0){$\  $}}
    \end{picture}\ee
    \bigskip\newline
	tangent  	to the $\mathcal{VB}$- groupoid (\ref{prawy}). The structural maps  for  (\ref{prawytang}) are the following:
		\be\label{517}\begin{array}{l}
 T\tilde{\Ss}_*(\dot\chi,\chi,v,\eta,w,\xi):=(0,w,\xi),\\
 T\tilde{\Tt}_*(\dot\chi,\chi,v,\eta,w,\xi):=(0,v,\eta),\\
 T\tilde{\mathbf{1}}_*(\chi,v,\eta):=(0,0,v,\eta,v,\eta),\\
 T\tilde {\lambda}_*(\dot\chi,\chi,v,\eta,w,\xi)=(\chi,\eta,\xi),\\
 \tilde 0_*(\chi,\eta,\xi)=(0,\chi,0,\eta,0,\xi)\\
 \lambda_*(\chi,v,\eta)=(0,\eta),\\
  0_*(0,\eta)=(0,0,\eta)\end{array}\ee
and the inverse map and the groupoid product are
\be\label{518}\begin{array}{l}
T\tilde{\iota}_*(\dot\chi,\chi,v,\eta,w,\xi):=(-\dot\chi, -\chi,w,\xi,v,\eta,),\\
(\dot\chi,\chi,v,\eta,w,\xi) (\dot{\mathcal{Y}},\mathcal{Y},w,\xi,z,\zeta):=(\dot\chi+\dot{\mathcal{Y}},\chi+\mathcal{Y},v,\eta,z,\zeta).
\end{array}\ee

For the $\mathcal{VB}$-groupoid 
 \be\label{prawytangdual}\begin{picture}(11,4.6)
    \put(0,4){\makebox(0,0){${T^*(p_0\M_* p_0\times P_0\times P_0})$}}
    \put(9,4){\makebox(0,0){${p_0\M_* p_0\times P_0\times P_0}$}}
    \put(0.8,0){\makebox(0,0){$p_0\M p_0\times T^*P_0$}}
    \put(8,0){\makebox(0,0){$\{0\}\times P_0$}}
    \put(1.2,3){\vector(0,-1){2}}
    \put(0.7,3){\vector(0,-1){2}}
    \put(8.2,3){\vector(0,-1){2}}
    \put(7.7,3){\vector(0,-1){2}}
    \put(4,4){\vector(1,0){2}}
    \put(3.2,0){\vector(1,0){3}}
    \put(0.1,1.4){\makebox(0,0){$ \  $}}
    \put(2.2,1.4){\makebox(0,0){$ \  $}}
    \put(9.1,1.4){\makebox(0,0){$ \  $}}
    \put(6.8,1.4){\makebox(0,0){$\   $}}
    \put(4.5,4.5){\makebox(0,0){$\  $}}
    \put(4.5,0.5){\makebox(0,0){$\  $}}
    \end{picture}\ee
    \bigskip\newline
		being the dualization of  (\ref{prawytang}) the structural maps and the groupoid product  are the following:
		\be\label{517}\begin{array}{l}
 T^*\tilde{\Ss}_*(\stackrel{\circ}{\chi},\chi,\varphi,\eta,\psi,\xi):=(\stackrel{\circ}{\chi},-\psi,\xi),\\
 T\tilde{\Tt}_*(\stackrel{\circ}{\chi},\chi,\varphi,\eta,\psi,\xi):=(\stackrel{\circ}{\chi},\varphi,\eta),\\
 T\tilde{\mathbf{1}}_*(\stackrel{\circ}{\chi},\chi,\varphi,\eta):=(\stackrel{\circ}{\chi},0,-\varphi,\eta,\varphi,\eta),\\
 \tilde \lambda_*(\stackrel{\circ}{\chi},\chi,\varphi,\eta,\psi,\xi)=(\chi,\eta,\xi),\\
 \tilde 0_*(\chi,\eta,\xi)=(0,\chi,0,\eta,0,\xi)\\
 \lambda_*(\chi,\varphi,\eta)=(0,\eta),\\
  0_*(0,\eta)=(0,0,\eta).\end{array}\ee
The inverse map and groupoid product for (\ref{prawytangdual}) are given by
\be\label{518}\begin{array}{l}
T\tilde{\iota}_*(\stackrel{\circ}{\chi},\chi,\varphi,\eta,\psi,\xi):=(\stackrel{\circ}{\chi},-\chi,-\psi,\xi,\varphi,\eta),\\
(\stackrel{\circ}{\chi},\chi,\varphi,\eta,\psi,\xi) (\stackrel{\circ}{\chi},\mathcal{Y},-\psi,\xi,\theta,\zeta):=(\stackrel{\circ}{\chi},\chi+\mathcal{Y}, \varphi,\eta,\theta,\zeta).
\end{array}\ee

\begin{lem} One has the following sequence of $\mathcal{VB}$-groupoids morphisms
			\be\label{lem:duzyprawy}\begin{picture}(11,4.6)
\put(-8,4){\makebox(0,0){$J_\flat^{-1}(0)$}}
    \put(-0.3,4){\makebox(0,0){$T^\flat(p_0\M_*p_0\!\times\! P_0\!\times\! P_0)$}}
    \put(8,4){\makebox(0,0){$T(p_0\M_*p_0\!\times\! P_0\!\times\! P_0)$}}
    		 \put(16.3,4){\makebox(0,0){$\frac{T(p_0\M_*p_0\times P_0\times P_0)}{T_eG_0}$}}
    \put(-8,-1){\makebox(0,0){$p_0\M p_0\times T^\flat P_0$}}
    \put(0,-1){\makebox(0,0){${p_0\M p_0\times T^\flat P_0}$}}
		\put(8,-1){\makebox(0,0){$\{0\}\times TP_0$}}
  		\put(16,-1){\makebox(0,0){$\{0\}\times\frac{TP_0}{T_eG_0},$}}
      \put(-8.2,3){\vector(0,-1){3}}
    \put(-7.8,3){\vector(0,-1){3}}
		  \put(-0.2,3){\vector(0,-1){3}}
    \put(0.2,3){\vector(0,-1){3}}
		\put(8.2,3){\vector(0,-1){3}}
    \put(7.8,3){\vector(0,-1){3}}
    \put(16.2,3){\vector(0,-1){3}}
    \put(15.8,3){\vector(0,-1){3}}
		      \put(-6.5,4){\vector(1,0){2.5}}
		\put(3.3,4){\vector(1,0){1.3}}
    \put(11.4,4){\vector(1,0){1.8}}
   
\put(-5.7,-1){\vector(1,0){3}}
		\put(2.4,-1){\vector(1,0){3}}
    \put(10.5,-1){\vector(1,0){3}}
    
		\put(-4.8,4.4){\makebox(0,0){$ \iota_2 $}}
    \put(4,4.4){\makebox(0,0){$  \tilde\#_2 $}}
    \put(12.1,4.4){\makebox(0,0){$  \tilde {Q}_2  $}}
    \put(-4,-0.6){\makebox(0,0){$ id $}}
    \put(4,-0.6){\makebox(0,0){$   \tilde{a}_* $}}
    \put(12,-0.6){\makebox(0,0){$ \tilde {Q}_1   $}}
    \end{picture}\ee
    \bigskip\newline
		where $J_\flat:T^\flat (p_0\M_*p_0\times P_0\times P_0)\to p_0\M_*p_0$ is defined by 
		\be J_\flat(\stackrel{\circ}{\chi},\chi,\varphi,\eta,\psi, \xi):=ad^*_{\stackrel{\circ}{\chi}}(\chi)+\varphi\eta+\psi \xi,\ee
		for $T^\flat (p_0\M_*p_0\times P_0\times P_0)\cong T^*(p_0\M _* p_0)\times T_*(P_0\times P_0)$.
		The anchor maps $\tilde{\#}_2$ and $\tilde{a}_*$ are defined by 
		\be \tilde{\#}_2(\stackrel{\circ}{\chi},\chi,\varphi,\eta,\psi, \xi):=(-ad^*_{\stackrel{\circ}{\chi}}(\chi),\chi,0,\eta,0, \xi)\ee
		\be \tilde{a}_*(\stackrel{\circ}{\chi},\varphi,\eta):=(0,\eta),\ee
		respectively.
		The bundle morphisms $\tilde {Q}_2$ and $\tilde {Q}_1$ in (\ref{lem:duzyprawy}) are the projections on the respective quotient bundles given by
		\be \tilde{Q}_1(0,v,\eta)=[(0,v,\eta)]=\{(0,v+\eta x,\eta);\quad x\in p_0\M p_0\}\ee
		\be \tilde{Q}_2(\dot{\chi},\chi,v,\eta,w,\xi)=[(\dot{\chi},\chi,v,\eta,w,\xi)]=\{(\dot\chi+ad^*_x(\chi),\chi,v+\eta x,\eta, w+\xi x,\xi);\quad x\in p_0\M p_0\}.\ee
\end{lem}
\begin{proof} By the direct verification.
\end{proof}

The following theorem summarizes important facts concerning the fibre-wise linear sub Poisson structures related to the gauge groupoid $\frac{P_0\times P_0}{G_0}\tto P_0/G_0$. Recall that this gauge groupoid is isomorphic to Banach-Lie groupoid $\G_{p_0}(\M)\tto \Ll_{p_0}(\M)$.

			\begin{thm}\label{thm:54} All Banach-Lie groupoids in the front of the spatial diagram (\ref{C}) are sub Poisson groupoids and the corresponding horizontal arrows of (\ref{C}) define sub Poisson morphisms between of them.
		\end{thm}
		\begin{proof} All horizontal maps in (\ref{lem:duzy}) are $G_0$-aquivariant groupoid morphisms. So, quotienting (\ref{lem:duzy}) by $G_0$ we obtain a sequence of  morphisms of the quotient groupoids. Let us describe these groupoids.
		
		 We have the following bundle isomorphisms:
		\be J_{2\flat}^{-1}(0)/G_0\cong T^\flat\left(\frac{T_*P_0\times T_*P_0}{G_0}\right),\ee
		\be \frac{T(T_*P_0\times T_*P_0)}{T_eG_0}/G_0\cong T\left(\frac{T_*P_0\times T_*P_0}{G_0}\right),\ee
		\be \frac{T(T_*P_0)}{T_eG_0}/G_0\cong T\left(\frac{T_*P_0}{G_0}\right).\ee

		Thus, the quotienting of $J_{2\flat}^{-1}(0)\tto T^\flat (T_*P_0)$ and $\frac{T(T_*P_0\times T_*P_0)}{T_eG_0}\tto  \frac{T(T_*P_0)}{T_eG_0}$ by $G_0$ leads to the groupoid morphism 
		
		\be\label{thm:1}\begin{picture}(11,4.6)
    \put(0.3,4){\makebox(0,0){$T^\flat\left(\frac{T_*P_0\times T_*P_0}{G_0}\right)$}}
    \put(8.5,4){\makebox(0,0){$T\left(\frac{T_*P_0\times T_*P_0}{G_0}\right)$}}
    \put(0.7,-1){\makebox(0,0){$\frac{T^\flat({T_*P_0})}{G_0}$}}
    \put(8.3,-1){\makebox(0,0){$T(\frac{T_*P_0}{G_0}).$}}
    \put(1.2,3){\vector(0,-1){3}}
    \put(0.7,3){\vector(0,-1){3}}
    \put(8.2,3){\vector(0,-1){3}}
    \put(7.7,3){\vector(0,-1){3}}
    \put(3,4){\vector(1,0){2.8}}
    \put(2.7,-1){\vector(1,0){3.2}}
    \put(0.1,1.4){\makebox(0,0){$ \ $}}
    \put(2.2,1.4){\makebox(0,0){$  \ $}}
    \put(8.8,1.4){\makebox(0,0){$  \  $}}
    \put(7.1,1.4){\makebox(0,0){$ \    $}}
    \put(4.5,4.3){\makebox(0,0){$ [\#_2]  $}}
    \put(4.5,-0.7){\makebox(0,0){$[\#_1]   $}}
    \end{picture}\ee
    \bigskip\newline
		where the sub Poisson anchors $[\#_1]$ and $[\#_2]$ in (\ref{thm:1}) are defined as quotients of morphisms $Q_1\circ \#_1\circ id$ and $Q_2\circ \#_2\circ \tilde{\iota}_2$ by $G_0$, respectively. From (\ref{thm:1}) we conclude that $\frac{T_*P_0\times T_*P_0}{G_0}\tto T_*P_0/G_0$ is a sub Poisson $\mathcal{VB}$-groupoid. Let us note here that the bundle $J_{2\flat}^{-1}(0)\to T_*P_0\times T_*P_0$ can be consideredas the bundle dual to $T\left(\frac{T_*P_0\times T_*P_0}{G_0}\right)\to T_*P_0\times T_*P_0$.
		
		 The quotient groupoid of the first $\mathcal{VB}$-groupoid in (\ref{lem:duzy}) is isomorphic to the $\mathcal{VB}$-groupoid
			
		\be\label{thm:2}\begin{picture}(11,4.6)
    \put(0.3,4){\makebox(0,0){$T^\flat\left(T_*(\frac{P_0\times P_0}{G_0})\right)$}}
    \put(8.5,4){\makebox(0,0){$T_*\left(\frac{P_0\times P_0}{G_0}\right)$}}
    \put(0.7,-1){\makebox(0,0){$T^\flat(\frac{T_*P_0}{G_0})$}}
    \put(8.3,-1){\makebox(0,0){$\frac{T_*P_0}{G_0}.$}}
    \put(1.2,3){\vector(0,-1){3}}
    \put(0.7,3){\vector(0,-1){3}}
    \put(8.2,3){\vector(0,-1){3}}
    \put(7.7,3){\vector(0,-1){3}}
    \put(3,4){\vector(1,0){2.8}}
    \put(2.7,-1){\vector(1,0){3.2}}
    \put(0.1,1.4){\makebox(0,0){$ \ $}}
    \put(2.2,1.4){\makebox(0,0){$  \ $}}
    \put(8.8,1.4){\makebox(0,0){$  \  $}}
    \put(7.1,1.4){\makebox(0,0){$ \    $}}
    \put(4.5,4.3){\makebox(0,0){$ \   $}}
    \put(4.5,-0.7){\makebox(0,0){$\    $}}
    \end{picture}\ee
    \bigskip\newline
		which has $T_*(\frac{P_0\times P_0}{G_0})\tto T_*P_0/G_0$ as its side groupoid.
		Thus after quotienting  (\ref{lem:duzy}) by $G_0$ we obtain the $\mathcal{VB}$-groupoid morphism
		\be\label{thm:3}\begin{picture}(11,4.6)
    \put(0.3,4){\makebox(0,0){$T^\flat\left(T_*(\frac{P_0\times P_0}{G_0})\right)$}}
    \put(8.5,4){\makebox(0,0){$T\left(T_*(\frac{P_0\times P_0}{G_0})\right)$}}
    \put(0.7,-1){\makebox(0,0){$T^\flat(\frac{T_*P_0}{G_0})$}}
    \put(8.3,-1){\makebox(0,0){$T(\frac{T_*P_0}{G_0}).$}}
    \put(1.2,3){\vector(0,-1){3}}
    \put(0.7,3){\vector(0,-1){3}}
    \put(8.2,3){\vector(0,-1){3}}
    \put(7.7,3){\vector(0,-1){3}}
    \put(3,4){\vector(1,0){2.8}}
    \put(2.7,-1){\vector(1,0){3.2}}
    \put(0.1,1.4){\makebox(0,0){$ \ $}}
    \put(2.2,1.4){\makebox(0,0){$  \ $}}
    \put(8.8,1.4){\makebox(0,0){$  \  $}}
    \put(7.1,1.4){\makebox(0,0){$ \    $}}
      \put(4.5,4.4){\makebox(0,0){$ \tilde{\#}_2  $}}
    \put(4.5,-0.6){\makebox(0,0){$\tilde{\#}_1   $}}
    \end{picture}\ee
    \bigskip\newline where $\tilde{\#}_1 $ and $\tilde{\#}_2$ are defined as the quotients of  $Q_1\circ \#_1\circ id\circ \iota_1$ and $Q_2\circ \#_2\circ \tilde{\iota}_2\circ \iota_2$ by $G_0$, respectively. Note here that $\tilde{\#}_1 $ and $\tilde{\#}_2$ expressed in $G_0$-invariant coordinates assume the form presented in (\ref{sPois1}) and in (\ref{anch3}). We also recall that $\#_2:T^\flat\left(\frac{T_*P_0\times T_* P_0}{G_0})\right)\to T\left(\frac{T_*P_0\times T_* P_0}{G_0})\right)$ maps elements of $T^\flat\left(\frac{T_*P_0\times T_* P_0}{G_0})\right)$ onto vectors tangent to the symplectic leaves of $\frac{T_*P_0\times T_* P_0}{G_0}$, so,  in the particular case onto vectors tangent to $T_*(\frac{P_0\times P_0}{G_0})\cong J_{2\flat}^{-1}(0)/G_0\subset \frac{T_*P_0\times T_*P_0}{G_0}$.		Therefore, it follows from (\ref{thm:3})  that $T_*(\frac{P_0\times P_0}{G_0})\tto T_*P_0/G_0$ is a weak symplectic groupoid.

		Taking the quotient of (\ref{lem:duzyprawy}) by $G_0$ we obtain the $\mathcal{VB}$-groupoids morphisms
		 \be\label{prawyprawy}\begin{picture}(11,4.6)
    \put(1,4){\makebox(0,0){$T^\flat(\frac{p_0\M_* p_0\times P_0\times P_0}{G_0})$}}
    \put(9,4){\makebox(0,0){$T(\frac{p_0\M_* p_0\times P_0\times P_0}{G_0})$}}
    \put(1,0){\makebox(0,0){$\frac{p_0\M p_0\times T^\flat P_0}{G_0}$}}
    \put(8,0){\makebox(0,0){$T(P_0/G_0)$}}
    \put(1.2,3){\vector(0,-1){2}}
    \put(0.7,3){\vector(0,-1){2}}
    \put(8.2,3){\vector(0,-1){2}}
    \put(7.7,3){\vector(0,-1){2}}
    \put(4,4){\vector(1,0){2}}
    \put(3,0){\vector(1,0){3}}
    \put(0.1,1.4){\makebox(0,0){$ \  $}}
    \put(2.2,1.4){\makebox(0,0){$ \  $}}
    \put(9.1,1.4){\makebox(0,0){$ \  $}}
    \put(6.8,1.4){\makebox(0,0){$\   $}}
    \put(4.5,4.5){\makebox(0,0){$[\#]  $}}
    \put(4.5,0.5){\makebox(0,0){$[\tilde{a}_*]  $}}
    \end{picture}\ee
    \bigskip\newline
		where $[\#]:=[\tilde{Q}_2\circ \#\circ\iota]$ and $[\tilde{a}_*]:=[\tilde{Q}_1\circ \tilde{a}_*\circ id]$ are the projectivizations of the respective maps from diagram (\ref{lem:duzyprawy}). Note that in order to obtain (\ref{prawyprawy}) we have used the bundle isomorphism
		\be J_\flat^{-1}(0)/G_0\cong T^\flat\left(\frac{p_0\M_* p_0\times P_0\times P_0}{G_0}\right).\ee
		From (\ref{prawyprawy}) we find that $\frac{p_0\M_* p_0\times P_0\times P_0}{G_0}\tto \{0\}\times P_0/G_0$ is a Poisson groupoid. Ending we note that $T^\flat$-subbundles of $T^*$-bundles were defined in (\ref{flat1}).
		
		 One can check by the straithforward verification that the horizontal arrows in (\ref{C}) define the $\mathcal{VB}$-groupoids morphism. Applying Theorem \ref{prop:44} to the case of the principal bundle $P_0\times P_0\to \frac{P_0\times P_0}{G_0}$ we find that they are also the Poisson morphisms.

			\end{proof}

		\section{Concluding remarks}
		
		In this section we will present cursory review of those questions which were not touched  on in the paper but are crucial for the theory  investigated here. We begin describing interrelation between the Banach Lie algebroid  structure of $\A\G(\M)\to \Ll(\M)$ and the linear sub Poisson structure on $\A_*\G(\M)\to \Ll(\M)$ which is not so obvious in a  not-reflexive Banach case.
		\subsection{Algebroid structure of $TP_0/G_0$ and the linear sub Poisson structure on $T_*P_0/G_0$.}
		
		In the previous sections we investigated the Atiyah sequence (\ref{Atiyahalgebr}), the predual Atiyah sequence (\ref{Atiyahalgebrdual}) as well as the short exact sequences of $\mathcal{VB}$-groupoids (\ref{duzyVtrojkaG}) and (\ref{C}). All of these sequences are defined in a canonical way by the structure of $W^*$-algebra $\M$.  The predual Atiyah sequence (\ref{Atiyahalgebr}) is an ingredient of (\ref{C}) which after the dualization gives the short exact sequence of $\mathcal{VB}$-groupoids (\ref{duzyVtrojkaG}). However, except of the case when $\M$ has the finite dimention , the dualizations of (\ref{duzyVtrojkaG}) and (\ref{Atiyahalgebr}) does not give back (\ref{C}) and (\ref{Atiyahalgebrdual}), i.e. the dualization precedure is not reversive in general. Let us discuss this question closely comparing the algebroid structure of $TP_0/G_0$ with the sub Poisson structure on $T_*P_0/G_0$.
		
		The section $\V\in  \Gamma ^\infty_{G_0}TP_0\cong \Gamma^\infty(TP_0/G_0)$ and the $G_0$-invariant function $\rho\in C^\infty_{G_0}(P_0)\cong C^\infty(P_0/G_0)$ define   a function $f_{\V,\rho}\in C^\infty(T_*P_0/G_0)$ on  $T_*P_0/G_0$ by
		\be\label{61} f_{\V,\rho}(\varphi,\eta):=\langle\varphi,\V(\eta)\rangle+\rho(\eta).\ee
 The bracket (\ref{Pbracket}) taken on  $f_{\V_1,\rho_1}$, $f_{\V_2,\rho_2}\in  C^\infty(T_*P_0/G_0)$ fulfilled the equality
	\be\label{62}\{f_{\V_1,\rho_1},f_{\V_2,\rho_2}\}=f_{[\V_1,\V_2], \V_1(\rho_1)- \V_2(\rho_2)}\ee
	where $[\V_1,\V_2]$ is the Lie algebroid bracket given by (\ref{braketVV}). From (\ref{62}) we conclude that the space $\Ll^\infty_{G_0}(T_*P_0)$ of functions $f_{\V,\rho}$ is a Lie algebra.
	
	Now let us observe that the derivation $\{f_{\V,\rho},\cdot\}$ is a section of $T^{**}(T_*P_0)$ in general. It will be a vector field $\{f_{\V,\rho},\cdot\}\in \Gamma ^\infty T(T_*P_0)\subsetneq \Gamma ^\infty T^{**}(T_*P_0)$ if and only if  $f_{\V,\rho}\in \mathcal{P}^\infty_{G_0}(T_*P_0)\cong \mathcal{P}^\infty(T_*P_0/G_0)$.
	
	\begin{rem} \begin{enumerate}[(i)]
	\item \textsl{The function $f_{\V,\rho}$ belongs to $\mathcal{P}^\infty_{G_0}(T_*P_0)$ iff 
	\be\label{63}\varphi\circ \frac{\partial \V}{\partial\eta}(\eta)\in p_0\M_*\quad {\rm and}\quad\frac{\partial \rho}{\partial\eta}(\eta)\in p_0\M_*\ee
	for any $\varphi\in p_0\M_*$ and $\eta\in P_0$}.
	\item\textsl{ The Lie subalgebra $\A^\infty_{G_0}(T_*P_0)\subset\Ll^\infty_{G_0}(T_*P_0)$, consisting functions $f_\V:=f_{\V,\rho}$ such that $\rho=0$, is isomorphic to the Lie algebra $\Gamma^\infty_{G_0}TP_0$ of the algebroid of the groupoid $\G_{p_0}(\M)\tto \Ll_{p_0}(\M)$.}
	\end{enumerate}
	\end{rem}
	The first condition in (\ref{63}) means that the Banach subspace $p_0\M_*\subset(\M p_0)^*$ is invariant with respect to the bounded operator 
		$\frac{\partial \V}{\partial\eta}(\eta)^*\in L^\infty((\M p_0)^*)$ dual to $\frac{\partial \V}{\partial\eta}(\eta)\in L^\infty(\M p_0)$ what means that the operator $\frac{\partial \rho}{\partial\eta}(\eta)$ is continuous with respect to $\sigma(\M p_0, p_0\M_*)$-topology of $\M p_0$.
		
		The proposition presented below summaries some important properties of the Lie subalgebra $\Ll^\infty_{G_0}(T_*P_0)\cap \mathcal{P}^\infty_{G_0}(T_*P_0)$.
		
		\begin{prop} \begin{enumerate}[(i)]
	\item The Poisson algebra   $\mathcal{P}^\infty_{G_0}(T_*P_0)$is generated by functions from the Lie subalgebra $\Ll^\infty_{G_0}(T_*P_0)\cap \mathcal{P}^\infty_{G_0}(T_*P_0)$.
		\item If $f\in \Ll^\infty_{G_0}(T_*P_0)\cap \mathcal{P}^\infty_{G_0}(T_*P_0)$ then the cotangent lift $L_t^*:T^*P_0\to T_*P_0$ of the left translation flow $L_t:P_0\to P_0$ tangent to $\V\in \Gamma^\infty_{G_0}(TP_0)$ preserves the precotangent bundle $T_*P_0\subset T^*P_0$. The vector field $\{f_{\V,\rho},\cdot\}\in \Gamma^\infty T(T^*P_0)$ is tangent to $L^*_t$.	
	\end{enumerate}
	\end{prop}
	Taking $T^*P_0/G_0$ instead of $T_*P_0/G_0$ and applying the bracket (\ref{Pbracket}) to the functions (\ref{61}) now defined on $T^*P_0$ we find that the space of such functions $ \Ll^\infty_{G_0}(T^*P_0)$ is a Lie algebra. In this case the derivation 
	\be\label{64}\{f_{\V},\cdot\}=\V(\eta)\frac{\partial}{\partial\eta}-\left\langle\varphi,\frac{\partial\V}{\partial\eta}(\eta)\right\rangle\frac{\partial}{\partial\rho}\ee
	is a linear vector field on $T^*P_0$ which, however, after restriction to $T_*P_0\subset T^*P_0$ will be a section of $T^{**}(T_*P_0)$ in general. The vector field $\{f_{\V},\cdot\}\in \Gamma^\infty T(T^*P_0)$ is tangent to $L^*_t$.
	
			Summing up we conclude that in the case considered here the correspondence between the sections of the algebroid $TP_0/G_0$, the linear vector field on $T_*P_0/G_0$ or  $T^*P_0/G_0$ and their tangent flows is not so univocal as it has place in the finite dimensional case described by Proposition 3.4.2 of \cite{mac}.
			
			\subsection{Realification and the subalgebroid $\U(\M)\tto \Ll(\M)$ of the partial isometries}
			All structures from the previous sections were investigated in the framework of the category of complex (holomorphic) Banach manifolds. Passing to the underlying real Banach manifolds with underlying real structures of the Banach Lie groupoids, Banach Lie algebroids and the Banach sub Poisson manifolds one can reformulate statements of the previous sections to their real versions.
			
			In particular case after realification one can consider $\G(\M)\tto \Ll(\M)$ as a real Banach Lie groupoid. By $\U(\M)$ we denote the set of partial isometries of the $W^*$-algebra $\M$. The inverse map $\iota:\G(\M)\to \G(\M)$ and the conjugation map $*:\G(\M)\to \G(\M)$ define the involution 
			\be\label{65}  J(x):=\iota(x)^*=\iota(x^*)\ee
			which is an authomorphism of the real Banach Lie groupoid $\G(\M)\tto \Ll(\M)$.
			One easily see that $x\in \U(\M)$ iff $J(x)=x$. Since $J:\G(\M)\to \G(\M)$ is an involutive authomorphism of the real Banach Lie groupoid $\G(\M)\tto \Ll(\M)$, see \cite{OS}, we conclude that the groupoid of partial isometries $\U(\M)\tto \Ll(\M)$  is a wide real Banach Lie subgroupoid of $\G(\M)\tto \Ll(\M)$. 
			
			For any $p_0\in \Ll(\M)$ one has the transitive subgroupoid $\U_{p_0}(\M)\tto\Ll_{p_0}(\M)$ of $\U(\M)\tto\Ll(\M)$ and a variant of Proposition \ref{prop:11} is valid for this case. As for  $\G_{p_0}(\M)\tto\Ll_{p_0}(\M)$ one has a groupoid isomorphism 
			 \unitlength=5mm \be\label{gaugeisomU}\begin{picture}(11,4.6)
    \put(1,4){\makebox(0,0){$\frac{P_0^u\times P_0^u}{U_0}$}}
    \put(8,4){\makebox(0,0){$\U_{p_0}(\M)$}}
    \put(1,-1){\makebox(0,0){$P_0^u/U_0$}}
    \put(8,-1){\makebox(0,0){$\Ll_{p_0}(\M)$}}
    \put(1.2,3){\vector(0,-1){3}}
    \put(0.7,3){\vector(0,-1){3}}
    \put(8.2,3){\vector(0,-1){3}}
    \put(7.7,3){\vector(0,-1){3}}
    \put(3,4){\vector(1,0){3}}
    \put(2.7,-1){\vector(1,0){3.7}}
    \put(0.1,1.4){\makebox(0,0){$\ $}}
    \put(2.2,1.4){\makebox(0,0){$\ $}}
    \put(9.1,1.4){\makebox(0,0){$\Ss$}}
    \put(6.8,1.4){\makebox(0,0){$\Tt$}}
    \put(4.5,4.5){\makebox(0,0){$\phi$}}
    \put(4.5,-0.5){\makebox(0,0){$\varphi $}}
    \end{picture},\ee
    \bigskip
		\ \\
		where $P_0^u:=P_0\cap\U(\M)$ and $U_0$ is the group of unitary elements of $W^*$-subalgebra $p_0\M p_0$.
		
			In order to express $J$ in the coordinate (\ref{psipp}) we note that $J:\Omega_{p\tilde p}\to \Omega_{p\tilde p}$ and $\psi_{p\tilde p}(J(x))=(y_p,J(z_{p\tilde p}),\tilde y_{\tilde p})$. So, $x\in \Omega_{p\tilde p}\cap\U(\M)$ iff $z_{p\tilde p}z_{p\tilde p}^*=p$ (or equivalently 
			$z_{p\tilde p}^*z_{p\tilde p}=\tilde p$). Hence, fixing $z_{p\tilde p}^0\in \Tt^{-1}(p)\cap \Ss^{-1}(\tilde p)$  we can parematrize $z_{p\tilde p}=z_{p\tilde p}^0 g$ univocally by $g\in U(\tilde p\M \tilde p)$, where $ U(\tilde p\M \tilde p)$ is the group of unitary elements of $W^*$-subalgebra $\tilde p\M \tilde p$. A local chart on a properly choosen open subset $\Omega_{\tilde p}\subset U(\tilde p\M \tilde p)$ is given by $\log:\Omega_{\tilde p}\to i(\tilde p\M \tilde p)^h$, where $(\tilde p\M \tilde p)^h$ is the hermitian part of $\tilde p\M \tilde p$. Summarizing the above facts we obtain the atlas of charts parametrized by $p,\tilde p\in \Ll(\M)$:
			\be\label{66} \psi_{p\tilde p}^u:\Omega_{p\tilde p}\cap\U(\M)\ni x\mapsto \psi_{p\tilde p}^u(x)=(y_p,-i\log[(z_{p\tilde p}^0)^{-1}z_{p\tilde p}],\tilde y_{\tilde p})\in (1-p)\M p\oplus  (\tilde p\M \tilde p)^h\oplus (1-\tilde p)\M \tilde p,\ee
			where one consider $ (1-p)\M p$ and $(1-\tilde p)\M \tilde p$ as a real Banach spaces. This atlas  defines the structure of a real submanifold on $\U(\M)\tto \Ll(\M)$. Hence, similarly to the complex case one can apply the coordinate description to the groupoid of partial isometries.

	One can investigate  the Atiyah sequence, the predual Atiyah sequence and the short exact sequence of $\mathcal{VB}$-groupoids canonically related to the groupoid $\U(\M)\tto \Ll(\M)$ but we will not investigate this subject here.		

\subsection{Some remarks about the case of $\M=L^\infty(\H)$}
Let $\H$ be a separable complex Hilbert space. By $L^\infty(\H)$ we denote the $W^*$-algebra of bounded operators on $\H$. The predual Banach space $L^\infty(\H)_*$ of $L^\infty(\H)$ is the ideal $L^1(\H)\subset L^\infty(\H)$ of the trace class operators and the pairing between $(\rho,x)\in L^1(\H)\times  L^\infty(\H)$ is given by 
\be\label{pairing} \langle\rho,x\rangle:= Tr(\rho x).\ee

The lattice $\Ll(L^\infty(\H))$ of orthogonal projections is canonically isomorphic with the lattice $\Ll(\H)$ of Hilbert subspaces of $\H$.

From Banach inverse operator theorem it follows that $\G(L^\infty(\H))$ consists of operators with a closed image, e.g. $A\in \G(L^\infty(\H))\subset L^\infty(\H)$ if and only if $Im A=\ol{Im A}$. The set of operators with the closed images let us denote by $\G(\H)$. Therefore we can identify $\G(L^\infty(\H))\tto \Ll(L^\infty(\H))$ with the groupoid $\G(\H)\tto \Ll(\H)$, where $\Ss(A)=(kerA)^\bot$ and $\Tt(A)=ImA$.

			The groupoid $\G(\H)\tto \Ll(\H)$ of the partially invertible operators on $\H$ splits on the groupoid $\G_{fin}(\H)\tto \Ll_{fin}(\H)$ of the finite-rank operators defined by
			\be \Ll_{fin}(\H):=\bigcup_{N=1}^\infty \Ll_N(\H),\ee
					\be \G_{fin}(\H):=\Tt^{-1}(\Ll_{fin}(\H))\cap \Ss^{-1}(\Ll_{fin}(\H))\ee
					and the groupoid  $\G_{\infty}(\H)\tto \Ll_{\infty}(\H)$ of the infinite  dimensional range partially invertible operators, i.e. $A\in \G_{\infty}(\H)$ if and only if $dim_{\mathbb{C}}ImA=\infty$. We mention that $\Ll_N(\H)$ consists of projections of rank $N$.
					
					We define the Fredholm subgroupoid $\G_{Fred}(\H)\tto \Ll_{Fred}(\H)$ of the groupoid $\G_{\infty}(\H)\tto \Ll_{\infty}(\H)$ as follows
					\be \Ll_{Fred}(\H):=\bot(\Ll_{fin}(\H))\ee
					\be \G_{Fred}(\H):=\Tt^{-1}(\Ll_{Fred}(\H))\cap \Ss^{-1}(\Ll_{Fred}(\H))\ee
					where the involution  $\bot:\Ll(\H)\to \Ll(\H)$ is defined by the orthogonal complements 
					\be \bot(p)=p^{\bot}:=1-p\ee
					of $p\in \Ll(\H)$.
					\begin{prop} The involution   $\bot:\Ll(\H)\to \Ll(\H)$ is an authomorphism of the complex analytic Banach manifold  $\Ll(\H)$.
					\end{prop}
					\begin{proof} At first we show that $\bot(\Pi_p)=\Pi_{p^\bot}$. For this reason we note that the Banach splitting 
					\be L^\infty(\H)=q L^\infty(\H)\oplus (1-p)L^\infty(\H)\ee
					is equivalent to the splitting 
					\be\label{splitH} \H=q\H\oplus(1-p)\H\ee
					of $\H$ on the corresponding Hilbert subspaces, where $q\in \Pi_p$. Since existence of the splitting (\ref{splitH}) is equivalent  to the existence of the splitting
					\be\label{splitHort} \H=(1-q)\H\oplus p\H\ee
					we find that $1-p\in \Pi_{p^\bot}$.  Thus one has $\bot(\Pi_p)=\Pi_{p^\bot}$.
					
					Now using (\ref{rozklad})  we obtain 
					\be \varphi_p^{-1}(y_p)=q=l(p+y_p)=(p+y_p)(p+y_p)^{-1}.\ee
					Thus and from (\ref{varphi}) we find that 
					\be y_{p^\bot}=(\varphi_{p^\bot}\circ \bot \circ \varphi_{p}^{-1})(y_p)=[(1-p)(1-(p+y_p)(p+y_p)^{-1})]^{-1}-(1-p)\ee
					for $y_p\in \varphi_p(\Pi_p)$ and $y_{p^\bot}\in \varphi_{p^\bot}(\Pi_{p^\bot})$. The above shows that $\bot$ is a complex analytic authomorphism of $\Ll(\H)$.
					\end{proof}
					
					From the above proposition we conclude
					\begin{cor} The Fredholm groupoid $\G_{Fred}(\H)\tto \Ll_{Fred}(\H)$ is a complex Banach Lie subgroupoid of the Banach Lie groupoid $\G(\H)\tto \Ll(\H)$. 
					\end{cor}
					
					The all results of the previous sections concerning fibre-wise linear sub Poisson structures one can apply and investigate in the case of Banach Lie groupoids $\G(\H)\tto\Ll(\H)$, $\G_{Fred}(\H)\tto \Ll_{Fred}(\H)$ and $\G_{fin}(\H)\tto \Ll_{fin}(\H)$ important from geometrical as well as physical point of view. We back to this investigations in the next paper.

\section{Appendix}
	\subsection{$\mathcal{VB}$-groupoids}\label{Ap:VB}
	The concept of $\mathcal{VB}$-groupoids goes back to Pradines, \cite{Pradines:1988}. It is abstracted the vector bundle structure of groupoid $TG\tto TM$ tangent to a groupoid $G\tto M$.
	
	In a diagramatic presentation a $\mathcal{VB}$-groupoid is a  structure 
	 \unitlength=5mm \be\label{App1}\begin{picture}(11,4.6)
    \put(1,4){\makebox(0,0){$\Omega$}}
    \put(8,4){\makebox(0,0){$\Gamma$}}
    \put(1,-1){\makebox(0,0){$E$}}
    \put(8,-1){\makebox(0,0){$M$}}
    \put(1.2,3){\vector(0,-1){3}}
    \put(1,3){\vector(0,-1){3}}
    \put(7.9,3){\vector(0,-1){3}}
    \put(7.7,3){\vector(0,-1){3}}
		 \put(0.5,0){\vector(0,1){3}}
    \put(8.5,0){\vector(0,1){3}}
    \put(3,4){\vector(1,0){3}}
    \put(2.4,-1){\vector(1,0){3.7}}
		\put(6,3.6){\vector(-1,0){3}}
    \put(6.1,-1.4){\vector(-1,0){3.7}}
    \put(0.8,1.4){\makebox(0,0){$ \tilde\Tt  $}}
    \put(1.5,1.4){\makebox(0,0){$ \tilde\Ss  $}}
    \put(8.2,1.4){\makebox(0,0){$ \Ss  $}}
    \put(7.5,1.4){\makebox(0,0){$\Tt   $}}
     \put(0.3,1.4){\makebox(0,0){$ \tilde{\mathbf{1} } $}}
    \put(8.8,1.4){\makebox(0,0){$ \mathbf{1}  $}}
		\put(4.5,4.5){\makebox(0,0){$\tilde\lambda $}}
    \put(4.5,-0.5){\makebox(0,0){$\lambda  $}}		
		 \put(4.5,3.1){\makebox(0,0){$\tilde  0$}}
    \put(4.5,-2){\makebox(0,0){$0  $}}
    \end{picture},\ee
    \bigskip\newline
		in which 
		\begin{enumerate}[(i)]
		\item $\Omega\tto E$ and $\Gamma\tto M$ are Lie groupoids;
		\item $\Omega\stackrel{\tilde\lambda}{\to }\Gamma$ and $E\stackrel{\lambda}{\to }M$ are vector bundles;
		\item the above structures are subjected to the consistency conditions  such that the groupoids structural maps: $\Ss, \Tt, \mathbf{1}, \tilde\Ss, \tilde\Tt ,\tilde{\mathbf{1}}$ and groupoid operations, i.e. the product and the inverse map, are vector bundle morphisms; 
		\item the vector bundle projection $\tilde\lambda$ and $\lambda$ as well as null sections $\tilde 0$ and $0$ define groupoid morphisms;
		\item the "double source map" $(\tilde\lambda,\tilde\Ss):\Omega\to \Gamma\times_M E$, where $\Gamma\times_M E:=\{(\gamma,e)\in \Gamma\times E:\quad \Ss(\gamma)=\lambda(e)\}$, is surjective submersion,
		\item for   $w_1,w_2,\nu_1,\nu_2\in \Omega$ such that $\tilde\Tt(\nu_1)=\tilde\Ss(w_1)$, $\tilde\Tt(\nu_2)=\tilde\Ss(w_2)$, $\tilde\lambda(w_1)=\tilde\lambda(w_2)$, $\tilde\lambda(\nu_1)=\tilde\lambda(\nu_2)$ one has the condition
		$$(w_1+w_2)(\nu_1+\nu_2)=w_1\nu_1+w_2\nu_2$$
		called the interchange law.
	\end{enumerate}
	\subsection{The dual of a $\mathcal{VB}$-groupoid}
	Here we disscuss briefly the procedure of the dualization of (\ref{App1}). For use of this procedure one defines the core of a $\mathcal{VB}$-groupoid to be
	\be\label{core} K:=\{\omega\in \Omega:\quad \exists_{m\in M}\ \tilde\Ss(\omega)=0_m\quad and \quad \tilde\lambda(\omega)=\mathbf{1}_m\}\ee
	where $0:M\to E$ is the zero section of $E\stackrel{\lambda}{\to }M$  and $\mathbf{1}:M\to \Gamma$ the identity map of $\Gamma\tto M$. Defining $\lambda_K:K\to M$ by $\lambda_K:=\Ss\circ\tilde\lambda$ one easily verifies thet the core is a vector bundle over $M$. 
	The $\mathcal{VB}$-groupoid
	 \unitlength=5mm \be\label{App1dual}\begin{picture}(11,4.6)
    \put(1,4){\makebox(0,0){$\Omega^*$}}
    \put(8,4){\makebox(0,0){$\Gamma$}}
    \put(1,-1){\makebox(0,0){$K^*$}}
    \put(8,-1){\makebox(0,0){$M$}}
     \put(1.2,3){\vector(0,-1){3}}
    \put(1,3){\vector(0,-1){3}}
    \put(7.9,3){\vector(0,-1){3}}
    \put(7.7,3){\vector(0,-1){3}}
		 \put(0.3,0){\vector(0,1){3}}
    \put(8.5,0){\vector(0,1){3}}
    \put(3,4){\vector(1,0){3}}
    \put(2.4,-1){\vector(1,0){3.7}}
		\put(6,3.6){\vector(-1,0){3}}
    \put(6.1,-1.4){\vector(-1,0){3.7}}
     \put(0.8,1.4){\makebox(0,0){$ \tilde\Tt_*  $}}
    \put(1.6,1.4){\makebox(0,0){$ \tilde\Ss_*  $}}
    \put(8.2,1.4){\makebox(0,0){$ \Ss  $}}
    \put(7.5,1.4){\makebox(0,0){$\Tt   $}}
     \put(0.1,1.4){\makebox(0,0){$ \tilde{\mathbf{1} }_* $}}
    \put(8.8,1.4){\makebox(0,0){$ \mathbf{1}  $}}
    \put(4.5,4.5){\makebox(0,0){$\tilde\lambda_* $}}
    \put(4.5,-0.5){\makebox(0,0){$\lambda_*  $}}
		 \put(4.5,3.1){\makebox(0,0){$\tilde 0_* $}}
    \put(4.5,-2){\makebox(0,0){$0_*  $}}
    \end{picture},\ee
    \bigskip\newline
		dual to a $\mathcal{VB}$-groupoid (\ref{App1}) is defined as follows. 
		
		Take $\gamma\in \Gamma_m^n:=\Ss^{-1}(m)\cap\Tt^{-1}(n)$ and $\varphi\in \Omega_\gamma^*$, where $ \Omega_\gamma:=\tilde\lambda^{-1}(\gamma)$, then 
		\be \langle \tilde\Ss_*(\varphi),k\rangle:=\langle \varphi,-\tilde 0_\gamma k^{-1}\rangle \quad for \quad k\in K_m\ee
		and 
		\be \langle \tilde\Tt_*(\varphi),k\rangle:=\langle \varphi,k \tilde 0_\gamma\rangle \quad for \quad k\in K_n.\ee
		The composition $\varphi\psi\in \Omega_{\gamma\delta}^*$ of  $\varphi\in \Omega_{\gamma}^*$ and  $\psi\in \Omega_{\delta}^*$ with $\tilde\Ss_*(\varphi)=\tilde\Tt_*(\psi)$ one defines by 
		\be \langle \varphi\psi,wv\rangle:=\langle \varphi,w\rangle+\langle \psi,v\rangle,\ee
		where $w\in \Omega_\gamma$ and $v\in \Omega_\delta$ satisfy $\tilde\Ss(w)=\tilde\Tt(v)$. In order to define the identity map $\tilde{\mathbf{1}}_*:K^*\to \Omega^*$ at $\lambda\in K_m^*$ we note that $k=\omega-\tilde{\mathbf{1}}_{\tilde\lambda(\omega)}\in K_m$ where $m=\lambda(\tilde\Ss(\omega))=\Ss(\tilde\lambda(\omega))$. Now one defines
		\be \langle\tilde{\mathbf{1}}_{*\lambda},w\rangle=\langle\tilde{\mathbf{1}}_{*\lambda},k+\tilde{\mathbf{1}}_{\tilde\lambda(\omega)}\rangle:=\langle \lambda,k\rangle.\ee
		Since the inverse map $\iota:\Omega\to \Omega$ is an automorphism of the vector bundle $\tilde\lambda:\Omega\to \Gamma$ one defines the inverse map $\iota_*:\Omega^*\to \Omega^*$ as 
		\be\langle \iota_*\varphi,w\rangle:=-\langle\varphi,\iota(w)\rangle.\ee
		
		The straightforward verification shows the correctness of the above definitions.
		
		If the $\mathcal{VB}$-groupoid structure of (\ref{App1}) is modelled on the reflexive Banach spaces (what has a place in the finite dimensional case) then dualizing (\ref{App1dual}) we obtain back the initial $\mathcal{VB}$-groupoid. 
		\subsection{Short exact sequence of $\mathcal{VB}$-groupoids}\label{Ap:shortexact}
		A short exact sequence of $\mathcal{VB}$-groupoids
		 \be\label{App3}\begin{picture}(11,4.6)
    \put(-5.5,4){\makebox(0,0){$\Omega_1$}}
    \put(-5,0){\makebox(0,0){$E_1$}}
    \put(-0.4,4){\makebox(0,0){$\Gamma$}}
    \put(-1,0){\makebox(0,0){$M$}}
     \put(-5.1,3.5){\vector(0,-1){2.7}}
    \put(-4.9,3.5){\vector(0,-1){2.7}}
      \put(-1.1,3.5){\vector(0,-1){2.7}}
    \put(-0.9,3.5){\vector(0,-1){2.7}}
    \put(-3,4){\vector(1,0){0.8}}
    \put(-4,0){\vector(1,0){2.2}}
    \put(-4.5,3.5){\vector(4,-1){4.4}}
    \put(-4.2,-0.3){\vector(4,-1){5.2}}
      \put(-0,3.9){\line(4,-1){5}}
       \put(-0.1,3.8){\line(4,-1){5}}
         \put(-0.4,-0.1){\line(4,-1){5.5}}
       \put(-0.5,-0.2){\line(4,-1){5.5}}

  \put(-3,2.7){\makebox(0,0){$F$}}
    \put(-3,-1){\makebox(0,0){$f$}}

    \put(1.8,2){\makebox(0,0){$\Omega_2$}}
     \put(2,-2){\makebox(0,0){$E_2$}}
      \put(6.4,2){\makebox(0,0){$\Gamma$}}
     \put(6,-2){\makebox(0,0){$ M$}}
       \put(1.9,1.5){\vector(0,-1){2.7}}
    \put(2.1,1.5){\vector(0,-1){2.7}}
      \put(5.9,1.5){\vector(0,-1){2.7}}
    \put(6.1,1.5){\vector(0,-1){2.7}}
      \put(4,2){\vector(1,0){0.8}}
    \put(3,-2){\vector(1,0){2.2}}
     \put(2.5,1.5){\vector(4,-1){4.8}}
    \put(2.6,-2.3){\vector(4,-1){4.7}}
     \put(6.4,-2){\line(4,-1){5.9}}
       \put(6.3,-2.1){\line(4,-1){5.9}}
       \put(7,2){\line(4,-1){5.5}}
       \put(6.9,1.9){\line(4,-1){5.5}}

  \put(4,0.7){\makebox(0,0){$H$}}
     \put(4,-3){\makebox(0,0){$h$}}
		
    \put(9,0){\makebox(0,0){$\Omega_3$}}
   \put(8.6,-4){\makebox(0,0){$E_3$}}
     \put(13.4,0){\makebox(0,0){$\Gamma$}}
   \put(13,-4){\makebox(0,0){$M$}}
        \put(8.9,-0.5){\vector(0,-1){2.7}}
    \put(9.1,-0.5){\vector(0,-1){2.7}}
      \put(12.9,-0.5){\vector(0,-1){2.7}}
    \put(13.1,-0.5){\vector(0,-1){2.7}}
        \put(10.5,0){\vector(1,0){1.3}}
    \put(10.5,-4){\vector(1,0){1.7}}
		\end{picture}\ee
\vspace{2.5 cm}
\newline
consists of the groupoids $\Omega_k\tto E_k$, $k=1,2,3$, having $\Gamma\tto M $ as a side groupoid  and the groupoids morphisms $(F,f)$ and $(H,h)$ are such that
\be \label{App4}\Omega_1\stackrel{F}{\longrightarrow}\Omega_2\stackrel{H}{\longrightarrow}\Omega_3\ee
		is a short exact sequence of vector bundles over $\Gamma$. 
		
		The following statements are valid:
		\begin{enumerate}[(i)]
		\item $E_1\stackrel{f}{\to}E_2\stackrel{h}{\to}E_3$ is a short exact sequence of vector bundles over $M$
		\item $K_1\stackrel{F_K}{\to}K_2\stackrel{H_K}{\to}K_3$ is a short exact sequence of vector bundles over $M$, where the morphisms $F_K$ and $H_K$ are induced by $F$ and $H$ respectively.
		\item The dualization  of (\ref{App3}) results in 
		 \be\label{App5}\begin{picture}(11,4.6)
    \put(-5.5,4){\makebox(0,0){$\Omega_3^*$}}
    \put(-5,0){\makebox(0,0){$K_3^*$}}
    \put(-0.4,4){\makebox(0,0){$\Gamma$}}
    \put(-1,0){\makebox(0,0){$M$}}
     \put(-5.1,3.5){\vector(0,-1){2.7}}
    \put(-4.9,3.5){\vector(0,-1){2.7}}
      \put(-1.1,3.5){\vector(0,-1){2.7}}
    \put(-0.9,3.5){\vector(0,-1){2.7}}
    \put(-3,4){\vector(1,0){0.8}}
    \put(-4,0){\vector(1,0){2.2}}
    \put(-4.5,3.5){\vector(4,-1){4.4}}
    \put(-4.2,-0.3){\vector(4,-1){5.2}}
      \put(-0,3.9){\line(4,-1){5}}
       \put(-0.1,3.8){\line(4,-1){5}}
         \put(-0.4,-0.1){\line(4,-1){5.5}}
       \put(-0.5,-0.2){\line(4,-1){5.5}}

  \put(-3,2.7){\makebox(0,0){$H^*$}}
    \put(-3,-1){\makebox(0,0){$h^*$}}

    \put(1.8,2){\makebox(0,0){$\Omega_2^*$}}
     \put(2,-2){\makebox(0,0){$K_2^*$}}
      \put(6.4,2){\makebox(0,0){$\Gamma$}}
     \put(6,-2){\makebox(0,0){$ M$}}
       \put(1.9,1.5){\vector(0,-1){2.7}}
    \put(2.1,1.5){\vector(0,-1){2.7}}
      \put(5.9,1.5){\vector(0,-1){2.7}}
    \put(6.1,1.5){\vector(0,-1){2.7}}
      \put(4,2){\vector(1,0){0.8}}
    \put(3,-2){\vector(1,0){2.2}}
     \put(2.5,1.5){\vector(4,-1){4.8}}
    \put(2.6,-2.3){\vector(4,-1){4.7}}
     \put(6.4,-2){\line(4,-1){5.9}}
       \put(6.3,-2.1){\line(4,-1){5.9}}
       \put(7,2){\line(4,-1){5.5}}
       \put(6.9,1.9){\line(4,-1){5.5}}

  \put(4,0.7){\makebox(0,0){$F^*$}}
     \put(4,-3){\makebox(0,0){$f^*$}}
		
    \put(9,0){\makebox(0,0){$\Omega_1^*$}}
   \put(8.6,-4){\makebox(0,0){$K_1^*$}}
     \put(13.4,0){\makebox(0,0){$\Gamma$}}
   \put(13,-4){\makebox(0,0){$M$}}
        \put(8.9,-0.5){\vector(0,-1){2.7}}
    \put(9.1,-0.5){\vector(0,-1){2.7}}
      \put(12.9,-0.5){\vector(0,-1){2.7}}
    \put(13.1,-0.5){\vector(0,-1){2.7}}
        \put(10.5,0){\vector(1,0){1.3}}
    \put(10.5,-4){\vector(1,0){1.7}}
		\end{picture}\ee
\vspace{2.5 cm}
\newline
which is also a short exact sequence of $\mathcal{VB}$-groupoids.
	\end{enumerate}	
	\subsection{Poisson groupoid}  Let $\Gamma\tto M$ be a Lie groupoid with a Poisson structure  $\pi$ on the manifold 
		$\Gamma$. Then $(\Gamma,\pi)$ is called a Poisson groupoid if the Poisson anchor $\pi^\# :T^*\Gamma\to T\Gamma$ is a morphism of groupoids over some map $A^*\Gamma\to TM$.

\end{document}